\documentclass[10pt]{article}

\usepackage{url}
\usepackage{mathtools}
\usepackage{amssymb}
\usepackage{amsthm}
\usepackage{mathrsfs}
\usepackage{empheq}
\usepackage{latexsym}
\usepackage{enumitem}
\usepackage{eurosym}
\usepackage{dsfont}
\usepackage{appendix}
\usepackage{color} 
\usepackage[unicode]{hyperref}
\usepackage{frcursive}
\usepackage[utf8]{inputenc}
\usepackage[T1]{fontenc}
\usepackage{geometry}
\usepackage{multirow}
\usepackage{lmodern}
\usepackage{anyfontsize}
\usepackage{pgfplots}
\usepackage{upgreek}
\usepackage{comment}
\usepackage{cleveref}
\usepackage[main=british]{babel}
\usepackage[babel=true]{csquotes}
\usepackage[textwidth=0.75in]{todonotes}
\usepackage{natbib}

\setcitestyle{numbers,open={[},close={]}}

\definecolor{red}{rgb}{0.7,0.15,0.15}
\definecolor{green}{rgb}{0,0.5,0}
\definecolor{blue}{rgb}{0,0,0.7}
\hypersetup{colorlinks, linkcolor={red},citecolor={green}, urlcolor={blue}}
			
\makeatletter \@addtoreset{equation}{section}

\DeclarePairedDelimiter{\ceil}{\lceil}{\rceil}

\newtheorem{theorem}{Theorem}[section]
\newtheorem{assumption}{Assumption}

\newtheorem*{assumption*}{Assumption}
\newtheorem{corollary}[theorem]{Corollary}
\newtheorem*{corollary*}{Corollary}

\newtheorem{lemma}[theorem]{Lemma}
\newtheorem*{lemma*}{Lemma}
\newtheorem{proposition}[theorem]{Proposition}

\newtheorem{definition}[theorem]{Definition}
\newtheorem{remark}[theorem]{Remark}
\newtheorem*{theorem*}{Theorem}
\pgfplotsset{compat=1.14}
\newtheorem{lemmaA}{Lemma}[subsection]
\newtheorem{theoremA}{Theorem}[subsection]

\DeclareUnicodeCharacter{014D}{\=o}
\def \A{\mathbb{A}}

\def \D{\mathbb{D}}
\def \E{\mathbb{E}}
\def \F{\mathbb{F}}
\def \G{\mathbb{G}}
\def \H{\mathbb{H}}
\def \I{\mathbb{I}}

\def \L{\mathbb{L}}
\def \M{\mathbb{M}}
\def \N{\mathbb{N}}

\def \P{\mathbb{P}}
\def \Q{\mathbb{Q}}
\def \R{\mathbb{R}}
\def \S{\mathbb{S}}

\def\Ac{{\cal A}}
\def\Bc{{\cal B}}
\def\Cc{{\cal C}}
\def\Dc{{\cal D}}
\def\Ec{{\cal E}}
\def\Fc{{\cal F}}
\def\Gc{{\cal G}}
\def\Hc{{\cal H}}

\def\Jc{{\cal J}}
\def\Kc{{\cal K}}
\def\Lc{{\cal L}}
\def\Mc{{\cal M}}
\def\Nc{{\cal N}}
\def\Oc{{\cal O}}
\def\Pc{{\cal P}}

\def\Sc{{\cal S}}
\def\Tc{{\cal T}}
\def\Uc{{\cal U}}
\def\Vc{{\cal V}}

\def\Xc{{\cal X}}
\def\Yc{{\cal Y}}
\def\Zc{{\cal Z}}

\DeclareMathAlphabet{\pazocal}{OMS}{zplm}{m}{n}

\def\Af{{\mathfrak A}}

\def\Mf{{\mathfrak M}}

\def\Tf{{\mathfrak T}}

\def\f{{\mathsf{f}}}
\def\g{{\mathsf{g}}}
\def\xb{ \mathbf{x}}

\def\Hs{\mathscr{H}}
\def\eps{\varepsilon}
\def\d{{\mathrm{d}}}
\def\1{\mathbf{1}}
\def\e{{\mathrm{e}}}
\def\o{{\mathrm{o}}}

\def\sigmah{\widehat{\sigma}}
\def\mh{\widehat{m}}
\def\nh{\widehat{n}}

\def\f{{\mathsf{f}}}

\DeclareMathOperator*{\sgn}{sgn} 
 
\DeclareMathOperator*{\argmax}{arg\,max} 
\DeclareMathOperator*{\es}{ess\,sup^\P}

\DeclareMathOperator*{\Prob}{Prob} 
\DeclareMathOperator*{\as}{\text{\rm a.s.}} 
\let\ae\relax
\DeclareMathOperator*{\ae}{\text{\rm a.e.}} 
\DeclareMathOperator*{\qs}{\text{\rm q.s.}} 
\DeclareMathOperator*{\qe}{\text{\rm q.e.}} 

\DeclareMathOperator*{\Tr}{Tr}

\renewcommand{\|}{\Vert}
\renewcommand{\t}{\top}

\makeatletter
\def\moverlay{\mathpalette\mov@rlay}
\def\mov@rlay#1#2{\leavevmode\vtop{%
   \baselineskip\z@skip \lineskiplimit-\maxdimen
   \ialign{\hfil$\m@th#1##$\hfil\cr#2\crcr}}}
\newcommand{\charfusion}[3][\mathord]{
    #1{\ifx#1\mathop\vphantom{#2}\fi
        \mathpalette\mov@rlay{#2\cr#3}
      }
    \ifx#1\mathop\expandafter\displaylimits\fi}
\makeatother

\allowdisplaybreaks

\title{Me, myself and I: a general theory of non-Markovian time-inconsistent stochastic control for sophisticated agents\footnote{The authors gratefully acknowledge the support of the ANR project PACMAN ANR-16-CE05-0027.} }

\author{Camilo {\sc Hern\'andez} \footnote{Columbia University, IEOR department, USA, camilo.hernandez@columbia.edu.} \and Dylan {\sc Possama\"{i}} \footnote{ETH Z\"urich, Mathematics department, Switzerland, dylan.possamai@math.ethz.ch}}
 
\date{\today}

\begin{document}

\maketitle

\begin{abstract}
We develop a theory for continuous-time non-Markovian stochastic control problems which are inherently time-inconsistent. Their distinguishing feature is that the classical Bellman optimality principle no longer holds. Our formulation is cast within the framework of a controlled non-Markovian forward stochastic differential equation, and a general objective functional setting. We adopt a game-theoretic approach to study such problems, meaning that we seek for \emph{sub-game perfect Nash equilibrium} points. As a first novelty of this work, we introduce and motivate a refinement of the definition of equilibrium that allows us to establish a direct and rigorous proof of an \emph{extended dynamic programming principle}, in the same spirit as in the classical theory. This in turn allows us to introduce a system consisting of an infinite family of backward stochastic differential equations analogous to the classical HJB equation. We prove that this system is fundamental, in the sense that its well-posedness is both necessary and sufficient to characterise the value function and equilibria. As a final step we provide an existence and uniqueness result. Some examples and extensions of our results are also presented.

\vspace{5mm}
 \vspace{5mm}

\end{abstract}

\section{Introduction}

This paper is concerned with developing a general theory to address time-inconsistent stochastic control problems for sophisticated agents. In particular, we address this task in a continuous-time and non-Markovian framework. This is, to the best of our knowledge, the first work that studies these problems at such level of generality from a probabilistic point of view. 

\medskip

The distinctive feature in these situations is that human beings do not necessarily behave as what neoclassical economists refer to \emph{perfectly rational decision-makers}. Such idealised individuals are aware of their alternatives, form expectations about any unknowns, have clear preferences, and choose their actions deliberately after some process of optimisation, see \citeauthor*{osborne1994course} \cite[Chapter 1]{osborne1994course}. In reality, their criteria for evaluating their well-being are in many cases a lot more involved than the ones considered in the classic literature. 
In mathematical terms, this translates into stochastic control problems in which the classic dynamic programming principle, or in other words the Bellman optimality principle, is not satisfied.
\medskip

Let us consider the form of payoff functionals at the core of the continuous-time optimal stochastic control literature in a non-Markovian framework. Given a time reference $t\in [0,T]$, where $T>0$ is a fix time horizon, a past trajectory $x$ for the state process $X$, whose path up to $t$ we denote by $X_{\cdot \wedge t}$, and an action plan $\M=(\P,\nu)$, that is to say, a probability distribution for $X$ and an action process, the reward derived by an agent is 
\begin{align}\label{pay--off0}
J(t,x,\M)=\E^{\P}\bigg[\int_t^T f_r(X_{\cdot \wedge r},\nu_r)\d r + \xi(X_{\cdot \wedge T})\bigg].
\end{align}

However, as pointed out by \citeauthor*{samuelson1937note} \cite{samuelson1937note}, 
\textquote{the solution to the problem of maximising \eqref{pay--off0} holds only for an agent deciding her actions throughout the period at the beginning of it, and, as she moves along in time, there is a \textit{perspective phenomenon} in that her view of the future in relation to her instantaneous time position remains invariant, rather than her evaluation of any particular year.[...] Moreover, these results will remain unchanged even if she were to discount from the existing point of time rather than from the beginning of the period. Therefore, the fact that this is so is in itself a presumption that individuals do behave in terms of functionals like \eqref{pay--off0}}. Consequently, understanding the rationale behind the actions of a broader class of economic individuals calls for the incorporation of functionals able to include the previous one as a particular case. This is the motivation behind any theory of time-inconsistency.\medskip

Time-inconsistency is generally the fact that marginal rates of substitution between goods consumed at different dates change over time, see \citeauthor*{strotz1955myopia} \cite{strotz1955myopia}, \citeauthor*{laibson1997golden} \cite{laibson1997golden}, \citeauthor*{gul2001temptation} \cite{gul2001temptation}, \citeauthor*{fudenberg2006dual} \cite{fudenberg2006dual}, \citeauthor*{odonoghue1999doing} \cite{odonoghue1999doing, odonoghue1999incentives}. 
In many applications, these time-inconsistent preferences introduce a conflict between \textquote{an impatient \emph{present self} and a patient \emph{future self}}, see \citeauthor*{brutscher2011payment} \cite{brutscher2011payment}. In \cite{strotz1955myopia}, where this phenomenon was first treated, three different types of agents are described: the pre-committed agent does not revise her initially decided strategy even if that makes her strategy time-inconsistent; the naive agent revises her strategy without taking future revisions into account even if that makes her strategy time-inconsistent; the sophisticated agent revises her strategy taking possible future revisions into account, and by avoiding such makes her strategy time-consistent. In this paper we are interested in the latter type. \medskip

The study of time-inconsistency has a long history. The game-theoretic approach started with \cite{strotz1955myopia} where the phenomenon was introduced in a continuous setting, and it was proved that preferences are time-consistent if, and only if, the discount factor is exponential with a constant discount rate. \citeauthor*{pollak1968consistent} \cite{pollak1968consistent} gave the solution to the problem for both naive and sophisticated agents under a logarithmic utility function. For a long period of time, most of the attention was given to the discrete-time setting introduced by \citeauthor*{phelps1968second} \cite{phelps1968second}. This was, presumably, due to the unavailability of system of equations providing a general method for solving the problem, at least for sophisticated agents. Nonetheless, the theory evolved and results were extended to new frameworks, although this was mostly on a case-by-case basis. For example, \citeauthor*{barro1999ramsey} \cite{barro1999ramsey} studied the neoclassical growth model that includes a variable rate of time preference, and \cite{laibson1997golden} considered the case of quasi-hyperbolic time preferences. Notably, \citeauthor*{basak2010dynamic} \cite{basak2010dynamic} treated the mean--variance portfolio problem and derived its time-consistent solution by cleverly decomposing the nonlinear term and then applying dynamic programming. In addition, \citeauthor*{goldman1980consistent} \cite{goldman1980consistent} presented one of the first proof of existence of an equilibrium under quite general conditions. More recently, \citeauthor*{vieille2009multiple} \cite{vieille2009multiple} showed how for infinite horizon dynamic optimisation problems with non-exponential discounting, the multiplicity of solutions (with different payoffs) was the rule rather than the exception. 

\medskip
To treat these problems in a systematic way, the series of works carried out by \citeauthor*{ekeland2006being} \cite{ekeland2006being,ekeland2010golden}, and \citeauthor*{ekeland2008investment} \cite{ekeland2008investment} introduced the first notion of sub-game perfect equilibria in continuous-time, where the source of inconsistency is non-exponential discounting. \cite{ekeland2006being} consider a deterministic setting, whereas \cite{ekeland2010golden} extends these ideas to Markovian diffusion dynamics. In \cite{ekeland2008investment}, the authors provide the first existence result in a Markovian context encompassing the one in their previous works. This was the basis for a general Markovian theory developed by \citeauthor*{bjork2014theory} \cite{bjork2014theory} in discrete-time and \citeauthor*{bjork2017time} \cite{bjork2016time2, bjork2017time} in continuous-time. Inspired by the notion of equilibrium in \cite{ekeland2010golden} and their study of the discrete case in \cite{bjork2014theory}, in \cite{bjork2017time} the authors consider a general Markovian framework with diffusion dynamics for the controlled state process $X$, and provide a system of PDEs whose solution allows to construct an equilibrium for the problem. Recently, \citeauthor*{he2019equilibrium} \cite{he2019equilibrium} fills in a missing step in \cite{bjork2017time} by deriving rigorously the PDE system and refining the definition of equilibrium, {\color{black} while \citeauthor*{lindensjo2019regular} \cite{lindensjo2019regular}\footnote{We are grateful to an anonnymous referee for calling our attention to this work.} shows that solving the PDE system is a necessary condition for a refinement of the notion of equilibrium by assuming the regularity of the value function. Nonetheless, so far none of these approaches were able to handle the non-Markovian analogue of these problems nor did they provide a full characterisation of equilibria and their associated value functions. These results become essential in applications that go beyond solving a time-inconsistent control problem, for example, in contracting problems involving a principal and time-inconsistent agents (see \citeauthor*{cvitanic2015dynamic} \cite{cvitanic2015dynamic}).}

\medskip

Simultaneously, extensions have been considered, and unsatisfactory seemingly simple scenarii have been identified. \citeauthor*{bjork2014mean} \cite{bjork2014mean} study the time-inconsistent version of the portfolio selection problem for diffusion dynamics and a mean--variance criterion. \citeauthor*{czichowsky2013time} \cite{czichowsky2013time} considers an extension of this problem for general semi-martingale dynamics. \citeauthor*{hu2012time} \cite{hu2012time,hu2017time} provide a rigorous characterisation of the linear--quadratic model, and \citeauthor*{huang2021strong} \cite{huang2021strong} perform a careful study in a Markov chain environment. Regarding the expected utility paradigm, \citeauthor*{karnam2016dynamic} \cite{karnam2016dynamic} introduce the idea of the \emph{dynamic utility} under which an original time-inconsistent problem (under the originally fixed utility) becomes a time-consistent one. \citeauthor*{he2019forward} \cite{he2019forward} propose the concept of forward rank-dependent performance processes, by means of the notion of conditional nonlinear expectation introduced by \citeauthor*{ma2018time} \cite{ma2018time}, to incorporate probability distortions without assuming that the model is fully known at the initial time. \citeauthor*{landriault2018equilibrium} \cite{landriault2018equilibrium} present an example, stemming from a mean--variance investment problem, in which uniqueness of the equilibrium via the PDE characterisation of \cite{bjork2014theory} fails. \medskip

A different approach is presented in \citeauthor*{yong2012time} \cite{yong2012time} and \citeauthor*{wei2017time} \cite{wei2017time}, where, in the framework of recursive utilities, an equilibrium is defined as a limit of discrete-time games leading to a system of forward--backward SDEs. Building upon the analysis in \cite{wei2017time}, \citeauthor*{wang2019time} \cite{wang2019time} consider the case where the cost functional is determined by a backward stochastic Volterra integral equation (BSVIE, for short) which covers the general discounting situation with a recursive feature. A Hamilton--Jacobi--Bellman equation (HJB equation, for short) is associated in order to obtain a verification result. Moreover, \citeauthor*{wang2019time} \cite{wang2019time} establish the well-posedness of the HJB equation and derive a probabilistic representation in terms of a novel type of BSVIEs. As explained in \Cref{Section:BSDEassociated}, for the class of problems considered in this paper, our approach helps making explicit the connection between time-inconsistent problems and BSVIEs. \citeauthor*{han2019time} \cite{han2019time} study the case where the state variable is a Volterra process and, by associating an extended path-dependent Hamilton--Jacobi--Bellman equation system, obtains a verification theorem for non-Markovian and non-semimartingale models. Finally, \citeauthor*{mei2020closed} \cite{mei2020closed} deals with time-inconsistent control problems for McKean--Vlasov dynamics which are, for example, a natural framework to study mean--variance problems. {\color{black}We highlight that the previous works focused of establishing verification results. In the present work, we go beyond this as we introduce a system which is fundamental for time-inconsistent control problems in the sense that its well-posedness is both necessary and sufficient. In wrods, all equilibria arise as solutions to such system.}\medskip

When it comes to time-inconsistent stopping problems, recent works have progressed in understanding this setting, yet many peculiarities and questions remain open. A novel treatment of optimal stopping for a Markovian diffusion process with a payoff functional involving probability distortions, for both naïve and sophisticated agents, is carried out by \citeauthor*{huang2020general} \cite{huang2020general}. \citeauthor*{huang2020optimal} \cite{huang2020optimal} consider a stopping problem under non-exponential discounting, and looks for an \emph{optimal equilibrium}, one which generates larger values than any other equilibrium does on the entire state space. \citeauthor*{he2019optimal} \cite{he2019optimal} study the problem of a pre-committed gambler and compare his behaviour to that of a naïve one.  Another series of works is that of \citeauthor*{christensen2019time} \cite{christensen2019time, christensen2019timeb, christensen2019moment} and \citeauthor*{bayraktar2019notions} \cite{bayraktar2019notions}.  \cite{christensen2019time} study a discrete-time Markov chain stopping problem and propose a definition of sub-game perfect Nash equilibrium for which necessary and sufficient equilibrium conditions are derived, and an equilibrium existence result is obtained. The extension to the continuous setting is performed in \cite{christensen2019timeb}, and \cite{christensen2019moment} studies the the pre-committed and sophisticated solutions to a moment constrained version of the optimal dividend problem. Independently, \cite{bayraktar2019notions} studied a continuous Markov chain process and proposed another notion of equilibrium. The authors thoroughly obtain the relation between the notions of optimal-mild, weak and strong equilibrium introduced in \cite{huang2020general}, \cite{christensen2019timeb} and \cite{bayraktar2019notions}, respectively, and provide a novel iteration method which directly constructs an optimal-mild equilibrium bypassing the need to find first all mild equilibria. Notably, \citeauthor*{tan2018failure} \cite{tan2018failure} gives an example of nonexistence of an equilibrium stopping plan. On the other hand, \citeauthor*{nutz2020conditional} \cite{nutz2020conditional} provide a first approach to the recently introduced conditional optimal stopping problems which are inherently time-inconsistent. 

\medskip
{\color{black}The contributions of this paper can be summarised as follows, see \Cref{Section:OurResults} for detailed statements. First, regarding the set-up of the problem, we present the first probabilistic approach to non-Markovian time-inconsistent stochastic control problems. In particular, the dynamics of the controlled state process are prescribed in weak formulation, and control on both the drift and the volatility are allowed. We address time-inconsistent control problems from a sophisticated agent point of view, and seek for equilibrium actions. Keeping in mind the rationale behind such agents, and reviewing the existing literature, \cite{strotz1955myopia} stated that: \textquote{[The] problem [of a sophisticated agent] is then to find the best plan among those that [she] will actually follow}. This indicates that the agent chooses a plan that coordinates her future preferences, and therefore, the equilibrium action is time-consistent. Consequently, \emph{the value associated with an equilibrium is expected to satisfy a dynamic programming principle} (DPP for short).  However, the only statement of such DPP in the existing literature is, to best of our knowledge, \cite[Proposition 8.1]{bjork2016time2} and, uncommonly, it exploits the PDE system introduced by the same authors in \cite{bjork2017time}. We further discuss this in \Cref{Section:NotionsofEquilibrium}. Motivated by the structure of the classical theory of control, we propose a refinement on the notion of equilibrium, see \Cref{Def:equilibrium}, that allow us to bypass such argument and obtain an extended DPP from this new notion of equilibria, see \Cref{Theorem:DPP:Limit}. Not surprisingly, its statement agrees with the connection between the game-theoretic approach to time-inconsistent control problems and classic time-consistent control problems in \cite{bjork2016time2}. We nonetheless emphasise that our result is a direct consequence of our notion of equilibrium. Once a DPP is available, we can naturally associate a system of backward stochastic differential equations (BSDEs for short) to it, see System \eqref{HJB}. This system features the same fully coupled structure and agrees with the corresponding PDE system known in the Markovian framework, see \Cref{Thm:PDEsystem:markovian}. Naturally, we address the question of stating what a \emph{solution} to the system is, see \Cref{Def:sol:systemH}, and we are able to obtain a verification result, see \Cref{Verification}. Notably, we also show that \eqref{HJB} is, at the same time, necessary to the study of this problem, see \Cref{Theorem:Necessity}. By this we mean that given an equilibrium, its corresponding value function is naturally associated with a solution to \eqref{HJB}. In particular, we prove that any equilibrium must necessarily maximise the Hamiltonian functional of the agent. Consequently, \eqref{HJB} is fundamental to the study of time-inconsistent control problems for sophisticated agents as all equilibria arise as solutions to such system. Our definition of equilibrium and the proof of the extended DPP, which bypasses the argument in \cite[Proposition 8.1]{bjork2016time2}, are key for this last result and as such we believe are valuable contributions to the theory of time-inconsistent control problems. Finally, we provide a well-posedness result in the case where the volatility of the state process is not controlled, see \Cref{Thm:wellposedness:driftcontrol}. This result, in combination with \Cref{Theorem:Necessity} and \Cref{Verification}, establishes the uniqueness of equilibria in such setting.}

\medskip

This paper is organised as follows. \Cref{Section:ProblemFormulation} is devoted to the formulation of the problem. It presents our probabilistic framework, motivates a game formulation to time-inconsistent non-Markovian stochastic control problems, and introduces our refinement of the definition of equilibrium. \Cref{Section:OurResults} is dedicated to state and describe our results, and to compare our definition of equilibrium with the ones available in the literature. \Cref{Section:Examples} illustrates our results with an example. \Cref{Section:DPP} takes care of rigorously proving the extend dynamic programming principle. \Cref{Section:Analysis} contains the analysis and main results of the proposed methodology, that is to say, the necessity and the sufficiency of \eqref{HJB} in a general setting. \Cref{Section:Extensions} presents direct extensions of our model to more general reward functionals. Additionally, the appendix includes the proof of the well-posedness of System \eqref{HJB} in the case when only drift control is allowed, as well as some auxiliary and technical results.

\medskip

{\bf Notations:}
Throughout this document we take the convention $\infty-\infty:=-\infty$, and we fix a time horizon $T>0$. $\R_+$ and $\R_+^\star$ denote the sets of non-negative and positive real numbers, respectively. Given $(E,\|\cdot \|)$ a Banach space, a positive integer $p$, and a non-negative integer $q$, $\Cc^{p}_{q}(E)$ (resp. $\Cc^{p}_{q,b}(E)$) will denote the space of functions from $E$ to $\R^{p}$ which are at least $q$ times continuously differentiable (resp. and bounded with bounded derivatives). We set $\Cc_{q,b}(E):=\Cc^{1}_{q,b}(E)$, \emph{i.e.} the space of $q$ times continuously differentiable bounded functions with bounded derivatives from $E$ to $\R$ . Whenever $E=[0,T]$ (resp. $q=0$ or $b$ is not specified), we suppress the dependence on $E$ (resp. on $q$ or $b$), \emph{e.g.} $\Cc^p$ denotes the space of continuous functions from $[0,T]$ to $\R^{p}$. Given $x\in \Cc^p$, we denote by $x_{\cdot \wedge t}$ the path of $x$ stopped at time $t$, \emph{i.e.} $x_{\cdot \wedge t}:=(x(r\wedge t), r\geq 0)$. Given $(x,\tilde x)\in \Cc^p\times\Cc^p$ and $t\in[0,T]$, we define their concatenation $x\otimes_t  \tilde x\in\Cc^p$ by $(x\otimes_t  \tilde x)(r):= x(r)1_{\{r\leq t\}}+(x(t)+\tilde x(r)-\tilde x(t))1_{\{r\geq t\}}$, $r\in[0,T]$. 
\medskip

For $\varphi \in \Cc^{p}_q(E)$ with $q\geq 2$, $\partial_{xx}^2 \varphi$ will denote its Hessian. For a function $\phi:[0,T]\times E$ with $s \longmapsto \phi(s,\alpha)$ uniformly continuous uniformly in $\alpha$, we denote by $\rho_{\phi }:[0,T]\longrightarrow \R$ its modulus of continuity, which we recall satisfies $\rho_\phi(\ell)\longrightarrow 0$ as $\ell \downarrow 0$. For $(u,v) \in \R^p\times\R^p$, $u\cdot b$ will denote their usual inner product, and $|u|$ the corresponding norm. For positive integers $m$ and $n$, we denote by $\Mc_{m,n}(\R)$ the space of $m\times n$ matrices with real entries. By $0_{m,n}$ and $\text{I}_n$ we denote the $m\times n$ matrix of zeros and the identity matrix of $\Mc_n(\R):=\Mc_{n,n}(\R)$, respectively. $\S_n^+(\R)$ denotes the set of $n\times n$ symmetric positive semi-definite matrices. $\Tr [M]$ denotes the trace of a matrix $M\in  \Mc_{n}(\R)$.

\medskip
For $(\Omega, \Fc)$ a measurable space, $\Prob(\Omega)$ denotes the collection of all probability measures on $(\Omega, \Fc)$. For a filtration $\F:=(\Fc_t)_{t\in [0,T]}$ on $(\Omega, \Fc)$, $\Pc_{\rm prog}(E,\F)$ (resp. $\Pc_{\rm pred}(E,\F)$, $\Pc_{\rm opt}(E,\F)$) will denote the set of $E$-valued, $\F$--progressively measurable processes (resp. $\F$-predictable processes, $\F$-optional processes). For $\P\in \Prob(\Omega)$ and a filtration $\F$, $\F^\P:=(\Fc_t^\P)_{t\in[0,T]},$ denotes the $\P$-augmentation of $\F$. We recall that for any $t\in [0,T]$, $\Fc^\P_t:=\Fc_t\vee \sigma(\Nc^\P)$, where $\Nc^\P:=\{N\subseteq \Omega: \exists B \in \Fc, N \subseteq B \text{ and } \P[B]=0\}$. With this, the probability measure $\P$ can be extended so that $(\Omega,\Fc, \F^\P,\P)$ becomes a complete probability space, see  \citeauthor*{karatzas1991brownian} \cite[Chapter II.7]{karatzas1991brownian}. $\F^\P_+$ denotes the right limit of $\F^\P$, \emph{i.e.} $\Fc_{t+}^\P:=\bigcap_{\eps>0} \Fc_{t+\eps}^\P$, $t\in[0,T)$, and $\Fc_{T+}^\P:=\Fc_T^\P$, so that $\F^{\P}_+$ is the minimal filtration that contains $\F$ and satisfies the usual conditions. Moreover, given $\Pc\subseteq \Prob(\Omega)$ we introduce the set of $\Pc$-polar sets $\Nc^\Pc:=\{ N\subseteq \Omega: N\subseteq B, \text{ for some } B\in \Fc \text{ with } \sup_{\P\in \Pc} \P[B]=0\}$, as well as the $\Pc$-completion of $\F$, $\F^\Pc:=(\Fc_t^\Pc)_{t\in [0,T]}$, with $\Fc_t^\Pc:= \Fc_t \vee \sigma(\Nc^\Pc)$, $t\in [0,T]$ together with the corresponding right-continuous limit $\F^\Pc_+:=(\Fc_{t+}^\Pc)_{t\in [0,T]}$, with $\Fc^\Pc_{t+}:=\bigcap_{\eps>0} \Fc_{t+\eps}^\Pc$, $t\in[0,T)$, and $\Fc_{T+}^\Pc:=\Fc_T^\Pc$. For $\{s,t\}\subseteq [0,T]$, with $s\leq t$, $\Tc_{s,t}(\F)$ denotes the collection of $[t,T]$-valued $\F$--stopping times. 
\medskip

Additionally, given $A \subseteq \R^k$, $\A$ denotes the collection of finite and positive Borel measures on $[0,T]\times A$ whose projection on $[0,T]$ is the Lebesgue measure. This is, any $q\in \A$ can be disintegrated as $q(\d t, \d a)=q_t(\d a) \d t$, for an appropriate Borel-measurable kernel $q_t$, unique up to (Lebesgue) almost everywhere equality. We are interested in the set $\A_0$, of $q\in \A$ of the form $q=\delta_{\phi_t}(\d a) \d t$, for $\delta_{\phi}$ the Dirac mass at a Borel-measurable function $\phi:[0,T]\longrightarrow A$.

\section{Problem formulation}\label{Section:ProblemFormulation}

\subsection{Probabilistic framework}\label{Section:ProbabilisticFramework}
 
Let  $d$ and $n$ be two positive integers, and $\Xc:= \Cc^d$. We will work on the canonical space $ \Omega:=\Xc \times \Cc^n \times \A$, whose elements we will denote generically by $\omega:=(x,\text{w},q)$, and with canonical process $(X,W,\Lambda)$, where 
\[X_t(\omega):=x(t), \; W_t(\omega):=\text{w}(t), \; \Lambda(\omega):=q, \; (t,\omega)\in [0,T]\times \Omega.\]

$\Xc$ and $\Cc^d$ are endowed with the topology $\Tf_\infty$, induced by the norm $\| x \|_\infty:=\sup_{0\leq t \leq T} |x(t)|$, $x\in\Xc$, while $\A$ is endowed with the topology $\Tf_{\text{w}}$ induced by weak convergence, which we recall is metrisable, for instance, by the Prohorov metric, see \citeauthor*{stroock2007multidimensional} \cite[Theorem 1.1.2]{stroock2007multidimensional}. With these norms, both spaces are Polish.\medskip

For $(t,\varphi) \in [0,T]\times \Cc_b([0,T]\times A)$, we define 
\begin{align*}
\Delta_t[\varphi]:=\iint_{[0,t]\times A} \varphi(r,a) \Lambda(\d r, \d a), \; \text{so that}\; \Delta_t[\varphi] (\omega)=\iint_{[0,t]\times A}  \varphi(r,a)q_r( \d a)\d r,\; \text{for any}\; \omega\in\Omega.
\end{align*}

We denote by $\Fc$ the Borel $\sigma$-field on $ \Omega$. In this paper, we will work with the filtrations $\F:=(\Fc_t)_{t\in[0,T]}$ and $\F^X:=(\Fc^X_t)_{t\in[0,T]}$ defined for $t \in [0,T]$ by
\begin{align*}
\Fc_t :=\sigma \Big( (X_r,W_r,\Delta_r [\varphi]): (r,\varphi)\in [0,t]\times \Cc_b([0,T]\times A)\Big),\; \Fc^X_t:=\sigma \Big( (X_r,\Delta_r [\varphi]): (r,\varphi)\in [0,t]\times \Cc_b([0,T]\times A)\Big).
\end{align*}

Additionally, we will work with processes $\psi:[0,T]\times \Xc \longrightarrow E, (t,x)\longmapsto \psi(t,x)$, for $E$ some Polish space, which are $\G$-optional, with $\G$ an arbitrary filtration, \emph{i.e.} $\Pc_{\rm opt}(E,\G)$-measurable. In particular, these processes are automatically non-anticipative, that is to say, $\psi_r(X):=\psi_r(X_{\cdot \wedge r})$ for any $r\in[0,T]$. We denote by $\pi^{\Xc}$ the canonical projection from $\Omega$ to $\Xc$ and let $\pi^{\Xc}_\# \P:=\P\circ (\pi^{\Xc})^{-1}$ denote the push-forward of $\P$. As the previous processes are defined on $[0,T]\times \Xc \subsetneq [0,T]\times \Omega$, we emphasise that throughout this paper, the assertion 
\begin{align}\label{Convetion:P}
`\P\text{--}\ae x \in \Xc\text{'}, \; \text{will always mean that}\; \big(\pi^{\Xc}_\# \P\big)[\Xc]=1.
\end{align}

$\P \in \Prob( \Omega)$ will be called a semi-martingale measure if $X$ is an $(\F,\P)$--semi-martingale. By \citeauthor*{karandikar1995pathwise} \cite{karandikar1995pathwise}, there then exists an $\F$-predictable process, denoted by $\langle X\rangle=(\langle X\rangle_t)_{t\in [0,T]}$, which coincides with the quadratic variation of $X$, $\P\text{--}\as$, for every semi-martingale measure $\P$. Thus, we can introduce the $d\times d$ symmetric positive semi-definite matrix $\widehat \sigma$ as the square root of $\widehat \sigma^2$ given by
\begin{align}\label{Def:quadraticvwrtL}
\widehat \sigma_t^2:=\limsup_{\eps \searrow 0} \frac{\langle X \rangle_t-\langle X\rangle_{t-\eps}}{\eps}, \, t\in [0,T].
\end{align}

\subsection{Conditioning and concatenation of probability measures}\label{Section:Rcpd}

In this section, we recall the celebrated result on the existence of a well-behaved $\omega$-by-$\omega$ version of the conditional expectation. We also introduce the concatenation of a measure and a stochastic kernel. These objects are key for the statement of our results in the level of generality we are working with.

\medskip
Recall $\Omega$ is a Polish space and $\Fc$ is a countably generated $\sigma$-algebra. For $\P\in \Prob(\Omega)$ and $\tau \in \Tc_{0,T}(\F)$, $\Fc_\tau$ is also countably generated, so there exists an associated regular conditional probability distribution (r.c.p.d. for short) $(\P_\omega^\tau)_{\omega \in \Omega}$, see \cite[Theorem 1.3.4]{stroock2007multidimensional}, satisfying
\begin{enumerate}[label=$(\roman*)$, ref=.$(\roman*)$,wide, labelindent=0pt]
\item for every $\omega \in \Omega$, $\P^\tau_\omega$ is a probability measure on $(\Omega,\Fc)$;
\item for every $E\in \Fc$, the mapping $\omega\longmapsto \P^\tau_\omega[E]$ is $\Fc_\tau$-measurable;
\item the family $(\P_\omega^\tau)_{\omega\in \Omega}$ is a version of the conditional probability measure of $\P$ given $\Fc_\tau$, that is to say for every $\P$-integrable, $\Fc$-measurable random variable $\xi$, we have $\E^\P[\xi|\Fc_\tau](\omega)=\E^{\P^\tau_\omega}[\xi]$, for $\P\text{--}\ae\ \omega \in \Omega$;
\item for every $\omega \in \Omega$, $\P^\tau_\omega [\Omega^\omega_\tau]=1$, where $\Omega^\omega_\tau:=\{ \omega' \in \Omega : \omega'(r)=\omega(r), 0\leq r\leq \tau(\omega)\}$.
\end{enumerate}

Moreover, for $\P\in \Prob(\Omega)$ and an $\Fc_\tau$-measurable stochastic kernel $(\Q_\omega^\tau)_{\omega\in\Omega}$ such that $\Q^\tau_\omega [\Omega^\omega_\tau]=1$ for every $\omega \in \Omega$, the concatenated probability measure is defined by
\begin{align}\label{Def:ConcMeasure}
\P \otimes_\tau \Q_\cdot [A]:=\int_\Omega \P(\d \omega) \int_\Omega \1_A(\omega \otimes_{\tau(\omega)}\tilde \omega) \Q_\omega(\d \tilde \omega),\; \forall A \in \Fc.
\end{align}
The following result, see \cite[Theorem 6.1.2]{stroock2007multidimensional}, gives a rigorous characterisation of the concatenation procedure.
\begin{theorem}[Concatenated measure]\label{Thm:Concatenated:M}
Consider a stochastic kernel $(\Q_\omega)_{\omega\in\Omega}$,  and let $\tau\in\Tc_{0,T}(\F)$. Suppose the map $\omega \longmapsto \Q_\omega$ is $\Fc_\tau$-measurable and $\Q_\omega[\Omega_\tau^\omega]=1$ for all $\omega\in \Omega$. Given $\P \in \Prob(\Omega)$, there is a unique probability measure $\P\otimes_{\tau(\cdot)} \Q_\cdot$ on $(\Omega,\Fc)$ such that $\P\otimes_{\tau(\cdot)} \Q_\cdot$ equals $\P$ on $(\Omega,\Fc_\tau)$ and $(\delta_{\omega} \otimes_{\tau(\omega)} \Q_\omega)_{\omega \in \Omega}$ is an {\rm r.c.p.d}. of $\P\otimes_{\tau(\cdot)}\Q_\cdot | \Fc_\tau$. For some $t\in[0,T]$, suppose that $\tau \geq t$, that $M:[t,T]\times \Omega \longrightarrow \R$ is a right-continuous, $\F$--progressively measurable function after $t$, such that $M_t$ is $\P\otimes_{\tau(\cdot)} \Q_\cdot$-integrable, that for all $r\in[t,T]$, $(M_{r\wedge \tau})_{r\in[t,T]}$ is an $(\F,\P)$-martingale, and that $(M_r- M_{r\wedge\tau(\omega)})_{r\in[ t,T]}$ is an $(\F,\Q_\omega)$-martingale, for all $\omega \in \Omega$. Then $(M_r)_{r\in [t,T]}$ is an $(\F,\P\otimes_{\tau(\cdot)}\Q_\cdot)$-martingale.
\end{theorem} 

In particular, for an $\Fc$-measurable function $\xi$, $\E^{\P\otimes_\tau \P^\tau_\cdot}[\xi]=\E^\P[\E^\P[\xi|\Fc_\tau]]=\E^\P[\xi]$. This is the classical tower property. Additionally, the reverse implication in the last statement in \Cref{Thm:Concatenated:M} holds by \cite[Theorem 1.2.10]{stroock2007multidimensional}.

\subsection{Controlled state dynamics}\label{Section:ControlledDynamics}
Let $k$ be a positive integer, and let $A\subseteq \R^k$. An action process $\nu$ is an $A$-valued $\F^X$-predictable process. Given an action process $\nu$, the \emph{controlled state equation} is given by the stochastic differential equation (SDE for short)
\begin{align}\label{Eq:dynamics}
X_t= x_0+ \int_0^t \sigma_r(X,\nu_r)  \big(b_r(X, \nu_r)\d r+\d W_r\big),\; t\in [0,T].
\end{align}
where $W$ is an $n$-dimensional Brownian motion, $x_0\in\R^d$, and
\begin{align}\label{DynamicsCoefficients}
\begin{split}
 \sigma&:[0,T]\times \Xc \times A \longrightarrow \Mc_{d, n}(\R),\; \text{is bounded, and} \;  (t,x)\longmapsto \sigma(t,x, a)\text{ is } \F^X\text{-optional for any } a\in A,\\
 b&:[0,T]\times  \Xc \times A\longrightarrow \R^n,\; \text{is bounded,} \;  (t,x)\longmapsto b(t,x, a)\text{ is } \F^X\text{-optional for any } a\in A.\end{split}
\end{align}
In this work we characterise the controlled state equation in terms of \textit{weak solutions} to \eqref{Eq:dynamics}. These come from so-called martingale problems, see \cite[Chapter 6]{stroock2007multidimensional}. Let $\overline X:=(X,W)$ and $\overline \sigma: [0,T]\times  \Xc \longrightarrow \Mc_{n+d}(\R)$ given by
\begin{align*}
\overline  \sigma:= \begin{pmatrix}
 \sigma &0_{d,n}\\
 \text{I}_{n}& 0_{n,d} \end{pmatrix}.
\end{align*}

For any $(t,x)\in [0,T]\times \Xc$, we define $\Pc(t,x)$ as the collection of $\P\in \Prob(\Omega)$ such that
\begin{enumerate}[label=$(\roman*)$, ref=.$(\roman*)$,wide, labelindent=0pt]
\item there exists $\text{w}\in \Cc^n$ such that $\P\circ ( X_{\cdot \wedge t},W_{\cdot \wedge t})^{-1}=\delta_{ (x_{\cdot\wedge t},\text{w}_{\cdot \wedge t})}$;
\item for all $\varphi \in \Cc_{2,b}(\R^{d+n})$, the process $M^\varphi :[t,T]\times  \Omega \longrightarrow \R$ defined by
\begin{align}\label{Eq:MartingaleProblem}
M_r^\varphi:=\varphi(\overline X_r)-\iint_{[t,r]\times A} \frac{1}{2}\Tr\big[(\overline \sigma \overline \sigma^\t)_u( X, a)  (\partial^2_{\overline x\overline x} \varphi)(\overline X_u)\big] \Lambda(\d u, \d a), \; r\in [t,T],
\end{align}
is an $(\F,\P)$--local martingale;
\item $\P[\Lambda \in \A_0]=1.$
\end{enumerate}

There are classical conditions ensuring that the set $\Pc(t,x)$ is non-empty. For instance, it is enough that the mapping $x\longmapsto \overline \sigma_t(x,a)$ is continuous for some constant control $a$, see \cite[Theorem 6.1.6]{stroock2007multidimensional}. We also recall that uniqueness of a solution, \emph{i.e.} there is a unique element in $\Pc(t,x)$	, holds when in addition $ \overline{ \sigma \sigma}^\t_t(x, a)$ is uniformly positive away from zero, \emph{i.e.} there is $ \lambda>0$ s.t. $\theta^\t \overline{ \sigma \sigma}^\t_t(x,a) \theta \geq  \lambda |\theta|^2,\; (t,x,\theta)\in [0,T]\times \Xc\times \R^d$,
see \cite[Theorem 7.1.6]{stroock2007multidimensional}.\medskip

For $\P\in \Pc(x):= \Pc(0,x)$, $W$ is an $n$-dimensional $\P$--Brownian motion and there is an $A$-valued process $\nu$ such that
\begin{align}\label{Eq:driftlessSDE}
X_t=x_0+\int_0^t \sigma_r(X,\nu_r) \mathrm{d}W_r, \; t \in [0,T], \; \P\text{\rm--a.s.}
\end{align}

\begin{remark}\label{Remark:martingaleproblem}
We remark some properties of the previous martingale problem which, in particular, justify \eqref{Eq:driftlessSDE}
\begin{enumerate}[label=$(\roman*)$, ref=.$(\roman*)$,wide, labelindent=0pt]
\item for any $\P\in \Pc(t,x)$ and $\nu$ verifying $ \P[\Lambda\in \A_0]=1$, \eqref{Eq:MartingaleProblem} implies
\begin{align}\label{Eq:SupportSigma}
\widehat \sigma^2_r=(\sigma \sigma^\t)_r(X,\nu_r),\;  \d r \otimes \d \P\text{--}\ae,\; \text{\rm on}\;[t,T]\times\Omega;
\end{align}

\item we highlight the fact that our approach is to enlarge the canonical space right from the beginning of the formulation. This is in contrast to, for instance, {\rm \citeauthor*{karoui2015capacities2} \cite[Remark 1.6]{karoui2015capacities2}}, where the canonical space is taken as $\Xc\times \A$ and enlargements are considered as properly needed. As $\Pc(t,x)$ ought to describe the law of $X$ as in \eqref{Eq:driftlessSDE}, this cannot be done unless the canonical space is extended. We feel our approach simplifies the readability and understanding of the analysis at no extra cost. Indeed, the extra canonical process $W$ allows us to get explicitly the existence of a $\P$--Brownian motion by virtue of L\'evy's characterisation. By \emph{\cite[Theorem 4.5.2]{stroock2007multidimensional}} and \eqref{Eq:SupportSigma}, since
\begin{align}\label{Eq:Remark:Wprocess}
 \begin{pmatrix}
 \sigma & 0 \\
 I_{n} & 0 \end{pmatrix}  \begin{pmatrix}
 \sigma^\t & I_{n}\\
0& 0 \end{pmatrix} =\begin{pmatrix}
\widehat \sigma^2 &  \sigma \\
 \sigma^\t & I_n
\end{pmatrix},\; 
\text{\rm it follows}\; 
\begin{pmatrix}
X_t\\
W_t
\end{pmatrix}=\begin{pmatrix}
x_0\\
\emph{w}_0
\end{pmatrix}+\int_0^t
 \begin{pmatrix}
 \sigma &0\\
 I_{n}& 0 \end{pmatrix} 
 \d W_s,\  t\in[0,T];
\end{align}

\item the reader might notice that the notation $\Pc(t,x)$ does not specify an initial condition for neither the process $W$, nor for the measure valued process $\Lambda$. Arguably, given our choice of $\Omega$, one is naturally led to introduce $\Pc(t,\omega)$, with initial condition $\omega=(x,\emph{w},q) \in \Omega$. Nevertheless, by \eqref{Eq:MartingaleProblem} and \eqref{Eq:Remark:Wprocess}, we see that the dynamics of $X$ depends on the increments of the application $[0,T] \ni t \longmapsto \Delta_t [\sigma \sigma^\t](\omega) \in \Mc_{d}(\R)$ for $\omega \in \Omega$. It is clear from this that the initial condition on $W$ and $\Lambda$ are irrelevant. This yields $\Pc(t,\omega)=\Pc(t,\tilde \omega)$ for all $\tilde \omega=(x, \tilde{\emph{w}},\tilde q)\in \Omega$.

\end{enumerate}
\end{remark}

We now introduce the class of admissible actions. We let $\Af$ denote the set of $A$-valued and $\F^X$-predictable processes. At the formal level, we will say $\nu\in \Af$ is admissible whenever \eqref{Eq:driftlessSDE} has a unique weak solution. A proper definition requires first to introduce some additional notations. We will denote by $(\P^{\nu}_{t,x})_{(t,x)\in [0,T]\times \Xc}$ the corresponding family of solutions associated to $\nu$. Moreover, we recall that uniqueness guarantees the measurability of the application $(t,x)\longmapsto \P^{\nu}_{t,x}$, see \cite[Exercise 6.7.4]{stroock2007multidimensional}. For $(t,x,\P)\in [0,T]\times \Xc\times  \Pc(t,x)$, we define 
\begin{align*}
\begin{split}
 \Ac^0 (t,x,\P)&:=\big\{ \nu\in \Af:  \Lambda(\d r, \d a)=\delta_{\nu_r}(\d a)\d r,\; \d r \otimes \d \P\text{--}\ae\; \text{on}\; [t,T]\times \Omega\big\},\\
\Mf^0(t,x)&:=\big\{ (\P,\nu)\in \Pc(t,x)\times  \Ac^0(t,x,\P)\big\}, \text{ and }
\Pc^0(t,x,\nu):=\big\{\P \in \Prob(\Omega):(\P,\nu)\in \Mf^0(t,x)\big\}.
\end{split}
\end{align*}
Letting $\Ac^0(t,x):=\bigcup_{\P\in \Pc(t,x)} \Ac^0(t,x,\P)$, we define rigorously the class of admissible actions
\begin{align}\label{Def:admissible}
\Ac(t,x):=\big\{\nu \in \Ac^0(t,x): \Pc^0(t,x,\nu)=\{\P^\nu_{t,x}\}\big\}.
\end{align}
We set $\Ac(x):=\Ac(0,x)$ and define similarly, $\Ac(t,x,\P)$, $\Pc(t,x,\nu)$, $\Mf(t,x)$, $\Ac(x,\P)$, $\Mf(x)$ and $\Pc(x,\nu)$.

\begin{remark}\label{Remark:admissibility}
\begin{enumerate}[label=$(\roman*)$, ref=.$(\roman*)$,wide, labelindent=0pt]
\item We remark that the sets $\Mf(t,x)$ and $\Ac(t,x)$ are equivalent parametrisations of the admissible solutions to \eqref{Eq:driftlessSDE}. This follows since uniqueness of weak solutions for fixed actions,  as required by \eqref{Def:admissible}, implies that the sets $\Ac(t,x,\P)$ are disjoint.

\item Introducing the sets $\Ac(t,x,\P)$ allows us to better handle action processes for which the quadratic variation of $X$ is the same. For different $\P\in \Pc(t,x)$ the discrepancy among such probability measures can be read from \eqref{Eq:SupportSigma}, \emph{i.e.} in the $($support of the$)$ quadratic variation of $X$. This reflects the fact that different diffusion coefficients of \eqref{Eq:driftlessSDE} might induce mutually singular probability measures in $\Pc(t,x)$. We also recall that in general $\Pc(t,x)$ is not finite since it is a convex set, see {\rm \citeauthor*{jacod2003limit} \cite[Proposition III.2.8]{jacod2003limit}}.
\end{enumerate}
\end{remark}

\begin{remark}\label{Remark:com}
\begin{enumerate}[label=$(\roman*)$, ref=.$(\roman*)$,wide, labelindent=0pt]
\item Since $b$ is bounded, it follows that given $(\P,\nu)\in \Mf(x)$, if we define $\M:=(\overline \P^{\nu},\nu)$ with
\begin{gather*}
\frac{\mathrm{d}\overline \P^{\nu}}{\mathrm{d}\P}:= \exp \bigg(\int_0^T b_r(X,\nu_r)\cdot \d  W_r- \int_0^T | b_r(X,\nu_r) |^2 \d r \bigg), \;\text{\rm and}\; W^{\M}_t:=W_t-\int_0^t b_r(X,\nu_r) \d r,\; t\in[0,T],
\end{gather*}
we have that $W^\M$ is a $\overline \P^\nu$--Brownian motion, and
\begin{align*}
 X_t= x_0+\int_0^t \sigma_r(X,\nu_r)  \big(b_r(X, \nu_r)\d r+\d W_r^\M \big), \; \overline \P^\nu\text{--}\as
 \end{align*}
\item  We will exploit the previous fact and often work under the drift-less dynamics \eqref{Eq:driftlessSDE}. We stress that in contrast to the strong formulation setting, in the weak formulation the state process $X$ is fixed and the action process $\nu$ allows to control the distribution of $X$ via $\overline \P^\nu$.
\end{enumerate}
\end{remark}

In light of the previous discussion, we define the collection of \textit{admissible models} with initial conditions $(t,x)\in [0,T]\times \Xc$ 
\[
\Mc(t,x):=\big\{(\overline \P^\nu,\nu): (\P,\nu)\in \Mf(t,x) \big\},
\]
and we set $\Mc(x):=\Mc(0,x)$. To ease notations we set $\Tc_{t,T}:=\Tc_{t,T}(\F)$, for any $\tau\in \Tc_{0,T}$, $\P^{\tau}_\cdot:=(\P^{\tau}_\omega)_{\omega\in \Omega}$, and for any $(\tau,\omega)\in \Tc_{0,T}\times \Omega$, $ \Pc(\tau,x):=\Pc(\tau(\omega),x)$, $ \Mf(\tau,x):=\Mf(\tau(\omega),x)$, and $\Mc(\tau,x):=\Mc(\tau(\omega),x)$.\medskip

We will take advantage in the rest of this paper of the fact that we can move freely from objects in $\Mc$ to their counterparts in $\Mf$, see \Cref{Lemma:CompCond}. Also, we mention that we will make a slight abuse of notation, and denote by $\M$ elements in both $\Mf(t,x)$ and $\Mc(t,x)$. It will be clear from the context whether $\M$ refers to a model for \eqref{Eq:dynamics}, or the drift-less dynamics \eqref{Eq:driftlessSDE}.

\subsection{Objective functional}

Let us introduce the running and terminal cost functionals
\begin{align*}
&f:[0,T]\times [0,T]\times\Xc \times A \longrightarrow \R, \text{ Borel-measurable, with } f_\cdot(s,\cdot,a),\; \F^X\text{-optional, } \text{for any } (s,a) \in [0,T]\times A;\\
&\xi:[0,T]\times\Xc \longrightarrow \R, \text{ Borel-measurable} .
\end{align*}

 We are interested in a generic payoff functional of the form
\begin{align*}
J(t,x,\nu):=\E^{\overline \P^\nu}\bigg[\int_t^T f_r(t,X,\nu_r)\mathrm{d}r + \xi(t,X_{\cdot \wedge T})\bigg], \; (t,x, \nu)\in [0,T]\times \Xc\times  \Ac(t,x).
\end{align*}

\begin{remark}\label{Remark:payoff:omegax}
\begin{enumerate}[label=$(\roman*)$, ref=.$(\roman*)$,wide, labelindent=0pt]
\item As we work on an enlarged probability space, one might wonder whether the reward's values under both formulations coincide. We recall that we chose to enlarge the canonical space, see {\rm\Cref{Section:ControlledDynamics}}, to explicitly account for the randomness driving \eqref{Eq:dynamics}, \emph{i.e.} the process $W$. Nevertheless, as for any $(\omega,\M)\in \Omega\times \Mf(t,x)$, the latter depends only on $x$, see {\rm \Cref{Remark:martingaleproblem}}, we see that given $\M$, $J$ is completely specified by $(t,x)$.

\item Given the form of the payoff functional $J$, the problem of maximising $\Ac(x)\ni \nu\longmapsto J(0,x,\nu)$ has a time-inconsistent nature. More precisely, the dependence of $f$ and $\xi$ on the current time $t$ is the source of inconsistency.
\end{enumerate}
\end{remark}

We study this problem from a game-theoretic perspective and look for \textit{equilibrium laws}. The next section is dedicated to explaining these concepts more thoroughly.

\subsection{Game formulation}\label{Section:gameformulation}

We recall that a strategy profile is \textit{sub-game perfect} if it prescribes a Nash equilibrium in any sub-game.  In our framework, every player together with a past trajectory define a new sub-game. This motivates the idea behind the definition of an equilibrium model, see among others \cite{strotz1955myopia}, \cite{ekeland2006being} and \cite{bjork2010general}.\medskip

Let $\xb \in \Xc$, $\nu^\star \in \Ac(\xb)$ be an action, which is a candidate for an equilibrium, $(t,\omega) \in [0,T]\times \Omega$ an arbitrary initial condition, $\ell\in (0, T-t]$, $\tau \in \Tc_{t,t+\ell}$ and $\nu \in \Ac(\tau,x)$. We define $\nu\otimes_{\tau}\nu^\star:=\nu\1_{[ t, \tau)}+
\nu^\star\1_{[ \tau ,T]}$.

\begin{definition}[\textit{Equilibrium}]\label{Def:equilibrium}
Let $\xb\in \Xc$, $\nu^\star \in \Ac(\xb)$. For $\eps>0$ let
\begin{align*}
\ell_\eps:=\inf \Big\{\ell >0: \exists \P\in \Pc(\xb),\;  \P\big[\big\{x\in \Xc: \exists (t,\nu) \in [0,T] \times  \Ac(t,x),\; J\big(t,t,x,\nu^\star \big)<  J\big(t,t,x,\nu\otimes_{t+\ell}\nu^\star \big) -\eps \ell\big\} \big]>0 \Big\}.
\end{align*}
If for any $\eps>0$, $\ell_\eps>0$ then $\nu^\star$ is an \textbf{equilibrium model}, and we write $\nu^\star \in \Ec(\xb)$.
\end{definition}

\begin{remark}\label{Remark:Def:Eq:Ext} We now make a few remarks regarding our definition.
\begin{enumerate}[label=$(\roman*)$, ref=.$(\roman*)$,wide, labelindent=0pt]
\item The first advantage of {\rm\Cref{Def:equilibrium}} is that $\ell_\eps$ is monotone in $\eps$. Indeed, let $0<\eps^\prime \leq \eps$, then by definition, $\ell_\eps$ satisfies the condition in the definition of $\ell_{\eps^\prime}$, thus $\ell_{\eps^\prime}\leq \ell_\eps$. We will exploit this in the proof of {\rm \Cref{Theorem:DPP:Limit}}.
\item In \emph{\cite{ekeland2008investment}}, the authors chose $\Ac(t,x,t+\ell):=\big\{ \nu\in \Ac(t,x):  \nu\otimes_{t+\ell} \nu^\star \in \Ac(t,x) \big\}$ as the class of actions to be compared to in the equilibrium condition.  Clearly, $\Ac(t,x,t+\ell) \subseteq \Ac(t,x)$ for $\ell>0$. Under the assumption of weak uniqueness, it holds that $\nu\otimes_{t+\ell} \nu^\star \in \Ac(t,x)$ for any $\nu \in \Ac(t,x)$, see {\rm\Cref{Lemma:forwardcondition}}.
\item From the previous definition, given $( \eps, \ell)\in \R_+^\star \times (0,\ell_\eps)$, for $\Pc(\xb)\text{--}\qe\;  x \in \Xc$ and any $(t,\nu)\in [0,T]\times \Ac(t,x)$
\begin{align}\label{Equilibrium:Ell:noifn}
J(t,t,x,\nu^\star)- J(t,t,x,\nu\otimes_{t+\ell}\nu^\star)\geq  -\eps \ell.
\end{align}
{\color{black} There are two distinguishing features in this definition.  The fist one is that {\rm \Cref{Def:equilibrium}} imposes \eqref{Equilibrium:Ell:noifn} for all $\ell<\ell_\eps$ uniformly in $(t,\nu)$. This local feature was not captured by the classical definition of equilibria, has motivated other refinements on the notion of equilibria, and will be key to prove an extended {\rm DPP}, see {\rm \Cref{Section:NotionsofEquilibrium}} for more details. The second one is that \eqref{Equilibrium:Ell:noifn} holds for $\Pc(\xb)\text{--} \qe\; x\in\Xc$, \emph{i.e.} the equilibrium condition holds only for trajectories that are reachable.}

\end{enumerate}
\end{remark}

In the rest of the document we fix $\xb \in \Xc$ and study the problem 
\begin{align*}\label{P}\tag{P}
v(t,x):=J(t,t,x,\nu^\star),\; (t,x)\in[0,T]\times \Xc,\; \nu^\star \in \Ec(\xb).
\end{align*}

Thanks to the weak uniqueness assumption, $v$ is well-defined for all $(t,x)\in [0,T]\times  \Xc$ and measurable. 

\begin{remark}\label{Remark:valuefunction}
\eqref{P} is fundamentally different from the problem of maximising $\Ac(t,x)\ni \nu \longmapsto J(t,t,x,\nu)$. First, in \eqref{P} one finds $\nu^\star\in \Ac(\xb)$ first and then defines the value function. This contrasts with the classical formulation of optimal control problems. Second, the previous maximisation will find player $t$'s so-called pre-committed strategy.
\end{remark}

\section{Related work and our results}\label{Section:OurResults}

As a preliminary to the presentation of our contributions, this section starts by comparing our setting with the ones considered in the existing literature.

\subsection{On the different notions of equilibrium}\label{Section:NotionsofEquilibrium}
We now make a few comments on our definition of equilibria, its relevance and compare it with the ones previously proposed. Let us begin by recalling the set-up adopted by most of the existing literature in continuous time-inconsistent stochastic control.

\medskip 
Given $T>0$, on the time interval $[0,T]$ a fixed filtered probability space $\big(\Omega,\F^W,\Fc_T^W,\P\big)$ supporting a Brownian motion $W$ is given. Here, $\F^W$ denotes the $\P$-augmented Brownian filtration. Let $\Ac$ denote the set of admissible actions and $\G$ a (possibly) smaller filtration than $\F^W$. For a $\G$-adapted process $\nu\in \Ac$, representing an action process, the state process $X$ is given by the unique strong solution to the SDE
\begin{align}\label{Dynamic:StrondFormulation}
 X_t^{0,x_0,\nu}=x_0+\int_0^tb_s(X_s^{0,x_0,\nu},\nu_s)\d s+\int_0^t\sigma_s (X_s^{0,x_0,\nu},\nu_s) \d W_s, \text{ for } t\in [0,T].
\end{align}

As introduced in \cite{ekeland2008investment}, $\nu^\star$ is then said to be an equilibrium if for all $(t,x,\nu)\in [0,T]\times \Xc\times \Ac$
\[\liminf_{\ell \searrow 0} \frac{J(t,t,x,\nu^\star)-J(t,t,x,\nu\otimes_{t+\ell}\nu^\star)}{\ell}\geq 0.\]
If we examine closely the above condition, we obtain that for any $\kappa:=(t,x,\nu )\in [0,T]\times \Xc\times \Ac$, $\eps>0$ and sequence $ (\ell_n)_{n\in\N}\subseteq [0,T]$, $\ell_n\longrightarrow0$,  there is a positive integer $N^\kappa_\eps$ such that
\begin{align}\label{EquilibriumOld}
 \forall n \geq N_{\eps}^{\kappa}, \, J(t,t,x,\nu^\star)-J(t,t,x,\nu\otimes_{t+\ell_n}\nu^\star)\geq -\eps \ell_n.
\end{align}

From \eqref{EquilibriumOld} it is clear that the classical definition of equilibrium is an $\eps$-like notion of equilibrium. Now, ever since its introduction, the distinctive case in which the $\liminf$ is 0 has been noticed. In fact, a situation in which the agent is worse off in a \emph{sequence} of coalitions with future versions of herself but in the limit is indifferent, conforms with this definition. This case is excluded when $\eps=0$, \emph{i.e.} the case of regular equilibria, see \cite[Definition 3]{he2019equilibrium}, another refinement of equilibria which is related to case $\eps=0$ in \eqref{Equilibrium:Ell:noifn}. 
{\color{black} We also remark that {\rm \cite[Section 4]{he2019equilibrium}} presents several examples from the existing literature on time-inconsistent control in which a classical equilibrium fails to be regular. From these examples one can further show that classical equilibria may fail to be equilibria as in \Cref{Def:equilibrium}}. \medskip

{\color{black} At this point we want to emphasise the main idea in our approach to time-inconsistent control for sophisticated agents: \emph{any useful definition of equilibrium for a sophisticated agent ought to lead, from first principles, to a dynamic programming principle}. Indeed, in line with the literature on stochastic control we chose this to be a direct consequence of the notion of equilibrium, and as such we introduced a refinement on the definition. As mentioned in the introduction, \cite[Proposition 8.1]{bjork2016time2} unveils the form of this DPP in a Markovian framework. Yet, it does it without laying proper assumptions or providing a rigorous proof. To be precise, \cite[Proposition 8.1]{bjork2016time2} states that (in the framework of strong formulation with feedback Markovian actions) given an equilibrium, as defined in \cite{bjork2016time2}, for any time-inconsistent stochastic control problem it is possible to find a time-consistent optimal stochastic control problem which attains the same value. Not surprisingly, the DPP satisfied by the associated time-consistent control problem agrees with ours in the Markovian framework. However, the argument laid down in \cite[Proposition 8.1]{bjork2016time2} assumes \emph{a priori} a smooth solution to the PDE system, \emph{i.e.} it presupposes that the value function associated to the equilibrium action and the decoupled pay-off functionals belong to $\Cc_{1,2}([0, T]\times  \R^d)$, and is in the spirit of the Feynman--Kac representation formula. This means that the class of equilibria for which the DPP used in \cite{bjork2016time2} holds is actually a sub-class of the ones given by the classical definition via the liminf. Indeed, these would correspond to regular equilibria as defined in \cite{lindensjo2019regular}, see also \cite[Remark 3.9]{lindensjo2019regular} and the discussion leading to \cite[Assumption 2]{he2021who}. Even more significant in our view, in the context of classic time-consistent control, the argument in \cite[Proposition 8.1]{bjork2016time2} would be equivalent to assuming that the HJB equation has a smooth classical solution to prove the DPP. Overall, we found this line of argument to be quite atypical in the sense that: $(i)$ this is usually done the other way around: the DPP allows one to show that the value function is related to the HJB PDE, usually in the viscosity sense; $(ii)$ quid of the cases where the value function fails to be smooth, which are ubiquitous in the literature? In classical control problems, the DPP holds under mere Borel-measurability of the value function.\medskip

As we will see in the next section, once a DPP is available, all the pieces necessary for a complete theory of time-inconsistent non-Markovian stochastic control will become apparent. }\medskip


We use the rest of this section to address different works in the area.

\begin{enumerate}[label=$(\roman*)$, ref=.$(\roman*)$,wide, labelindent=0pt]
\item The inaugural papers on the game theoretic approach to inconsistent control problems are the sequence of papers by \citeauthor*{ekeland2006being} \cite{ekeland2006being,ekeland2010golden}, and \citeauthor*{ekeland2008investment} \cite{ekeland2008investment}. In their initial work \cite{ekeland2006being}, the authors consider a strong formulation framework, \emph{i.e.} \eqref{Dynamic:StrondFormulation} and $\G=\G^X$ denotes the augmented natural filtration of $X^{0,x_0,\nu}$, and seek for closed-loop, \emph{i.e.} $\G^X$-measurable, action processes defined via spike perturbations of the form
\begin{align}\label{Spike:Admissible}
(a \otimes_{t+\ell} \nu)_r:=a 1_{[t,t+\ell)}(r)+\nu(X_r^{0,x_0,\nu})1_{[t+\ell,T]}(r),
\end{align} 
that maximise the corresponding Hamiltonian. However as already pointed out by \citeauthor*{wei2017time} in \cite{wei2017time}, the local comparison is made between a Markovian feedback control, \emph{i.e.} $\G=\G^X$ and $\nu_r=\nu(X_r^{0,x_0,\nu})$, and an open-loop control value $a$, \emph{i.e.} $\G=\F^W$, and there is no argument as to whether admissibility is preserved. The later two works \cite{ekeland2008investment} and \cite{ekeland2010golden} introduce the definition of equilibrium via \eqref{EquilibriumOld}, in which admissibility is defined by progressively measurable open loop processes with moments of all orders. This last condition guarantees the uniqueness of the strong solution to the controlled state dynamics. In our framework we do not need such condition as the state dynamics hold in the sense of weak solutions.

\item The study in the linear quadratic set-up is carried out by \citeauthor*{hu2012time} \cite{hu2012time,hu2017time}. There, the dynamics are stated in strong formulation. To bypass the admissibility issues in \cite{ekeland2006being}, \cite{ekeland2008investment} and \cite{ekeland2010golden}, the class of admissible actions is open loop, \emph{i.e.} $\G=\F^W$. Equilibria are defined via \eqref{EquilibriumOld} but contrasting against spike perturbation as in \eqref{Spike:Admissible}. In \cite{hu2012time}, the authors obtain a condition which ensures that an action process is an equilibrium, and in \cite{hu2017time} the authors are able to complete their analysis and provide a sufficient and necessary condition, via a flow of forward--backward stochastic differential equations. The approach taken to characterise equilibria in both \cite{hu2012time} and \cite{hu2017time} leverages on the particular structure of the linear quadratic setting and the admissibility class. Moreover, the authors are able to prove uniqueness of the equilibrium in the setting of a mean--variance portfolio selection model in a complete financial market, where the state process is one-dimensional and the coefficients in the formulation are deterministic. Unlike theirs, our results require standard Lipschitz assumptions.

\item Another sequence of works that has received great attention is by \citeauthor*{bjork2014theory} \cite{bjork2014theory} and \citeauthor*{bjork2017time} \cite{bjork2016time2,bjork2017time}. There, the problem is presented in strong formulation, \emph{i.e.} \eqref{Dynamic:StrondFormulation}, and admissibility is defined by Markovian feedback actions, \emph{i.e.} $\G=\G^X$ and $\nu_r=\nu(X_r^{0,x_0,\nu})$. The first of these works deals with the problem in discrete-time, a scenario in which the classic backward induction algorithm is implemented to obtain an equilibrium by seeking for a sign on the difference of the payoffs corresponding to $\nu^\star$ and $\nu\otimes_{t+\ell}\nu^\star$. In their subsequent paper, the results are extended to a continuous-time setting, implementing the definition of equilibrium \eqref{EquilibriumOld}. Their main contribution is to provide a system of PDEs associated to the problem and provide a verification theorem. Nevertheless, the derivation of such system is completely formal and, in addition, there is no rigorous argument about the well-posedness of the system.

\item A different setting is presented in \citeauthor*{wei2017time} \cite{wei2017time} and \citeauthor*{wang2019time} \cite{wang2019time}, see references therein too. Both works study a time-inconsistent recursive optimal control problem in strong formulation setting, \emph{i.e.} \eqref{Dynamic:StrondFormulation}, in which the class of admissible actions are Markovian feedback, \emph{i.e.} $\G=\G^X$ and $\nu_r=\nu_r(X_r^{0,x_0,\nu})$. 
An equilibrium is defined as the unique continuous solution to a forward--backward system that describes the controlled state process and the reward functionals of the players, and satisfy a \emph{local approximate optimality property}. More precisely, $u$ is the limit of a sequence of locally optimal controls at discrete times. 
We highlight that in addition to a verification theorem, the authors are able to prove well-posedness of their system in the uncontrolled volatility scenario. However, we do not think their definition of an equilibrium is much tractable. We believe this is due to the fact that their approach relies heavily on the approximation of the solution to the continuous game by discretised problems, and requires the equilibrium to be continuous in time, as opposed to mere measurability which is the case for us, see {\rm\Cref{Def:admissible}}. 
We also would like to mention \cite{wang2019time} which, building upon the ideas in \cite{wei2017time}, modelled the reward functional by a BSVIE and look for a time-consistent locally near optimal equilibrium strategy. The authors are able to extend the results of the previous paper, but more importantly, they argue that a BSVIE is a more suitable way to represent a recursive reward functional with non-exponential discounting.  We comment of this in \Cref{Section:BSDEassociated}.

\item Other recent works on this subject are \citeauthor*{huang2021strong} \cite{huang2021strong} and \citeauthor*{he2019equilibrium} \cite{he2019equilibrium}. The first article considers an infinite horizon stochastic control problem in which the agent can control the generator of a time-homogeneous, continuous-time, finite-state Markov chain. The authors begin by introducing two variations of the notion of equilibria, referred there as \emph{strong} and \emph{weak} equilibria. Exploiting the structure on their dynamics, 
they derive necessary and sufficient conditions for both notions of equilibria. Moreover, under compactness of the set of admissible actions, existence of an equilibria is proved. In the second work, working in the framework of \cite{bjork2017time}, \emph{i.e.} strong formulation \eqref{Dynamic:StrondFormulation}, 
the authors `perform the analysis of the derivation of the system, \emph{i.e.} lay down sufficient conditions under which the value function satisfies the system.' In addition to their analysis, the authors introduce two new notions of equilibria, \emph{regular} and \emph{strong} equilibria. Regular equilibria compare rewards with feasible actions different from $\nu^\star$, and strong equilibria allow comparisons with any feasible actions. By requiring extra regularity on the actions, the authors provide necessary and sufficient conditions for a strategy to be a regular or a strong equilibria. 
Even though \cite{he2019equilibrium} succeeds in characterising both notion of equilibria, the conditions under which such results hold are, as expected, quite stringent, requiring for instance that the optimal action is differentiable in time and with derivatives of polynomial growth.  {\color{black} Lastly, we mention that the notions of strong and regular equilibria, as introduced in \cite{huang2021strong} and \cite{he2019equilibrium} respectively, are related to the case $\eps=0$ in \eqref{Equilibrium:Ell:noifn}. Indeed, both definitions introduce a parameter analogous to $\ell_\eps$, \emph{i.e.} uniform in $(t,x,\nu)$, and consequently would be consistent with the proof of the extended dynamic programming principle we present in \Cref{Section:DPP}.}

\end{enumerate}

We emphasise that in our setting the dynamics of the state process are non-Markovian, moreover, we have chosen the class of admissible actions to be `closed-loop'\footnote{This is not a closed-loop formulation {\it per se}. Our controls are also adapted to the information generated by the canonical process $\Lambda$.} and non-Markovian, \emph{i.e.} $\F^X$-adapted see \Cref{Section:ProbabilisticFramework}. Additionally, our treatment of time-inconsistent control problems is via weak formulation which is in itself an extension on the previous works in the subject. 
\medskip

The remaining of this section is devoted to present the main results of this paper. 

\subsection{Dynamic programming principle}

Our first main result for the study of time-inconsistent stochastic control problems for sophisticated agents, Theorem \ref{Theorem:DPP:Limit}, concerns the local behaviour of the value function $v$ as defined in \eqref{P}. This result is the first of its kind in a continuous-time setting and it is the major milestone for a complete theory. This result confirms the intuition drawn from the behaviour of a sophisticated agent, in the sense that $v$ does indeed satisfy a dynamic programming principle. In fact, it can be regarded as an extended one, as it does reduce to the classic result known in optimal stochastic control in the case of exponential discounting, see \Cref{Remark:DPP:sanotycheck}. This result requires the following main assumptions.

\begin{assumption}\label{AssumptionA}
\begin{enumerate}[label=$(\roman*)$, ref=.$(\roman*)$,wide, labelindent=0pt]
\item The map $ s \longmapsto \xi(s,x)$ \emph{(}resp. $s\longmapsto f_t(s,x,a))$ is continuously differentiable uniformly in $x$ $($resp. in $(t,x,a))$, and we denote its derivative by $\partial_s \xi (s,x)$ $($resp. $\partial_s f_t(s,x,a))$. \label{AssumptionA:DPP1}

\item  \label{AssumptionA:DPP2} $a \longmapsto  f_t(s,x,a)$ is uniformly Lipschitz continuous, \emph{i.e.}
\begin{align*}
\exists C > 0,\; \forall (s,t,x,a,a^\prime)\in[0,T]^2\times\Xc\times A^2,\; \big|  f_t(s,x,a)-  f_t(s,x,a^\prime)\big| \leq C|a-a^\prime|.
\end{align*}
\item \label{AssumptionA:DPP} \begin{enumerate}[label=$(\alph*)$]
\item $x \longmapsto \xi (t, x)$ is lower-semicontinuous uniformly in $t$, \emph{i.e.}
\begin{align*}
\hspace*{-0.5cm} \forall (\tilde x,\eps)\in \Xc \times \R_+^\star,\; \exists U_{\tilde x}\in \Tf_\infty,\; \tilde x \in U_{\tilde x},\; \forall (t,x) \in [0,T]\times U_{\tilde x},\; \xi(t,x)\geq \xi(t,x_0)-\eps, \text{ when } \xi(t,x_0)>-\infty.
\end{align*}
\item $x\longmapsto b(t,x,a)$ and $x\longmapsto \sigma(t,x,a)$ are uniformly Lipschitz continuous, \emph{i.e.}
\begin{align*}
\exists C > 0,\; \forall (t,x,x^\prime,a)\in[0,T]\times\Xc^2\times A,\; |b_t(x,a)-b_t(x^\prime,a)|+|\sigma_t(x,a)-\sigma_t(x^\prime,a)| \leq C\|x_{\cdot\wedge t}-x_{\cdot\wedge t}^\prime\|_\infty.
\end{align*}
\end{enumerate}
\end{enumerate}
\end{assumption}

\begin{remark}
Let us comment on the above assumptions. As it will be clear from our analysis in {\rm\Cref{Section:DPP}} and {\rm\Cref{Section:necessity}}, to study time-inconsistent stochastic control problems for sophisticated agents under our notion of equilibrium, one needs to make sense of a system. The fact that we get such a system should be compared to the classical stochastic control framework, where only one {\rm BSDE} suffices to characterise the value function and the optimal control, see \emph{\cite[Section 4.5]{zhang2017backward}}.\medskip

Consequently, {\rm\Cref{AssumptionA}\ref{AssumptionA:DPP1}} and {\rm\Cref{AssumptionA}\ref{AssumptionA:DPP2}} are fairly mild requirements in order to understand the behaviour of the reward functionals on the agent's type, which is the source of inconsistency. We also remark that {\rm\Cref{AssumptionA}\ref{AssumptionA:DPP1}} guarantees that the map $ t\longmapsto f_t(t,x,a)$ is continuous, uniformly in $(x,a)$, which ensures extra regularity in type and time for the player's running rewards. Lastly, given our approach and our choice to not impose regularity on the action process, in order to get a rigorous dynamic programming principle, we cannot escape imposing extra assumptions. Namely, if we do not want to assume that $t \longmapsto \nu_t(X)$ is continuous, we have to impose regularity in $x$, which is exactly what {\rm\Cref{AssumptionA}\ref{AssumptionA:DPP}} does. This allows us to use the result in \emph{\cite{karoui2015capacities2}} regarding piece-wise constant approximations of stochastic control problems. We believe our choice is the least stringent and it is clearly much weaker than any regularity assumptions made in the existing literature, see \emph{\cite{bjork2014theory}, \cite{wei2017time}, \cite{he2019equilibrium}}. For details see the discussion after \eqref{DPP:Limit:Aux0} and \emph{\Cref{Remark:assumpDPP}}.
\end{remark}
Our dynamic programming principle takes the following form.
\begin{theorem}\label{Theorem:DPP:Limit}
Let {\rm\Cref{AssumptionA}} hold, and $\nu^\star \in \Ec(\xb)$. For $\{\sigma, \tau\}\subseteq \Tc_{t,T}$, $\sigma\leq  \tau$ and $\Pc(\xb)\text{--}\qe\; x \in \Xc$, we have
\begin{align}\label{Eq:DPP:limit}
v(\sigma,x) =\sup_{\nu \in \Ac(\sigma,x)}  \E^{\overline \P^{\nu}}\bigg[ v(\tau,X)+\int_{\sigma}^\tau \bigg( f_r(r,X,\nu_r)-\E^{{\overline \P}^{\nu^\star}_{r,\cdot}}\bigg[\partial_s \xi(r,X_{\cdot\wedge T}) +\int_r^T  \partial_s f_u(r,X,\nu^\star_u)\d u\bigg] \bigg)\d r\bigg].
\end{align}
Moreover, $\nu^\star$ attains the $\sup$ in \eqref{Eq:DPP:limit}.
\end{theorem}
%
\begin{remark}\label{Remark:ThmDPPlimit}
\begin{enumerate}[label=$(\roman*)$, ref=.$(\roman*)$,wide, labelindent=0pt]
{\color{black}  
\item In light of {\rm \Cref{Theorem:DPP:Limit}} we are led to consider a system consisting of a second order {\rm BSDE} $(${\rm 2BSDE} for short$)$ and an infinite collection of processes in order to solve {\rm  Problem $($\ref{P}$)$}. This will be the object of the next section.  As a by-product, we recover below the connection, already mentioned in {\rm \cite[Proposition 8.1]{bjork2016time2}}, between time-inconsistent control problems and optimal stochastic control problems in a Markovian setting. 

\item We emphasise that {\rm \Cref{Theorem:DPP:Limit}} is a direct consequence of {\rm \Cref{Def:equilibrium}}. Moreover, it differs from {\rm \cite[Proposition 8.1]{bjork2016time2}} in that the latter argues via the {\rm PDE} \eqref{Eq:HJB-BKM}, see below. In fact, the result in {\rm \cite{bjork2016time2}} is obtained assuming that a smooth solution exists and it is in the spirit of the Feynman--Kac representation formula.  This would be analogous to us assuming that we had access to a solution to the {\rm 2BSDE} in the proof of {\rm \Cref{Theorem:DPP:Limit}} which, among other things, would automatically rule out the possibility to prove the necessity of the {\rm 2BSDE}, \emph{i.e.} that any equilibria is associated to a solution to such system. In our probabilistic framework, the proof of {\rm \Cref{Theorem:DPP:Limit}} will ultimately allow us to bypass this and establish the necessity result.

\item Lastly, we point out that even for optimal control in discrete-time, proving a {\rm DPP} in fairly general settings is a very difficult task because of crucial measurability issues, see for instance {\rm \citeauthor*{bertsekas1978stochastic} \cite[Section 1.2]{bertsekas1978stochastic}.}
} 
\end{enumerate}
\end{remark}

\begin{corollary}
Let {\rm\Cref{AssumptionA}} hold, and $\nu^\star\in  \Ec(\xb)$. There exists a time-consistent stochastic control problem with the same value $v,$ for which $\nu^\star$ prescribes an optimal control. Namely let
\[[0,T]\times \Xc\times A \longrightarrow \R,\; (t,x,a)\longmapsto k_t(x,a):=f_t(t,x,a)-\E^{{\overline \P}^{\nu^\star}_{t,x}} \bigg[\partial_s \xi(t,X_{\cdot \wedge T})+\int_t^T \partial_s f_u(t,X,\nu^\star_u)\d u\bigg],\]
Then, for $\Pc(\xb)\text{--}\qe\; x \in \Xc$
\begin{align*}
v(t,x)=\sup_{\nu \in \Ac(t,x)} \E^{\overline \P^\nu} \bigg[\int_t^T k_r(X,\nu_r)\d r+\xi(T,X_{\cdot \wedge T})\bigg].
\end{align*}
\end{corollary}

\begin{remark}\label{Remark:Extention:StrictEq} 
\begin{enumerate}[label=$(\roman*)$, ref=.$(\roman*)$,wide, labelindent=0pt]
\item As mentioned in {\rm \cite[Proposition 8.1]{bjork2016time2}}, this last result is of little practical use as it requires knowing the equilibrium strategy {\it a priori}, since the functional $k$ does depend on $\nu^\star$. 

\item We remark that the results in this section hold if one lets $\ell_{\eps}$ in {\rm \Cref{Def:equilibrium}} depend on $(t,\nu)$, see {\rm\Cref{Remark:DPP:Limit}}. 
\end{enumerate}
\end{remark}

\subsection{BSDE system associated to (\textcolor{red}P)}\label{Section:BSDEassociated}
Let us introduce the functionals needed to state the rest of our results. As in the classical theory of optimal control, we introduce the Hamiltonian operator associated to this problem. For $(s,t,x, z,\gamma,\Sigma,u,v, a)\in [0,T)^2 \times \Xc \times \R^d\times \S_d(\R)\times \S_d(\R) \times \R \times \R^d\times A$, let
\begin{align*}
\mathsf{h} _t(s,x,z,\gamma,a):=&f_t(s,x,a)+ b_t(x,a)\cdot\sigma_t(x,a)^\t z+\frac{1}{2}\Tr\big[(\sigma \sigma^\t)_t(x,a)  \gamma \big] ; \\  
H_t(x,z,\gamma,u)  :=& \sup_{a\in A} \Big\{ \mathsf{h}_t(t,x,z,\gamma,a) \Big\}-u.
\end{align*}

Following the approach of \citeauthor*{soner2012wellposedness} \cite{soner2012wellposedness}, we introduce the range of our squared diffusions coefficient and the inverse map which assigns to every squared diffusion the corresponding set of generating actions
\begin{align*}
\begin{split}
{\bf\Sigma} _t(x)&:=\big\{ \Sigma_t(x,a)\in \S_d(\R): a\in A\big \},\;\text{where } \Sigma_t(x,a):=(\sigma \sigma^\t)_t(x,a),\\
A_t(x,\Sigma)&:=\big\{ a \in A: (\sigma\sigma^\t)_t(x,a)=\Sigma\big\},\, \Sigma \in {\bf\Sigma} _t(x).\end{split}
\end{align*}

The previous definitions allow us to isolate the partial maximisation with respect to the squared diffusion. Let
\begin{align*}
\begin{split}
h_t(s,x,z,a)&:= f_t(s,x,a)+ b_t(x,a)\cdot\sigma_t(x,a)^\t z,\;
 \nabla h_t(s,x,v,a):= \partial_s f_t	(s,x,a)+  b_t(x,a)\cdot\sigma_t(x,a)^\t v.\\
F_t(x,z,\Sigma,u)& :=\sup_{a \in A_t(x,\Sigma)} \big\{ h_t(t,x,z,a)\big \}-u,
\end{split}
\end{align*}
With this, $2H=(-2F)^*$ is the covex conjugate of $-2F$, \emph{i.e.}
\begin{align}\label{Hamiltonia:General}
H_t(x,z,\gamma,u)=\sup_{\Sigma \in {\bf\Sigma}_t(x)}\bigg\{ F_t(x,z,\Sigma,u)+\frac{1}{2}\Tr[\Sigma  \gamma] \bigg\}. 
\end{align}
Moreover, we assume there exists a unique $A$-valued, Borel-measurable map $\Vc^\star(t,x,z)$ satisfying\footnote{{\color{black} The existence of such mapping is guaranteed by \citeauthor*{schal1974selection} \cite[Theorem 3]{schal1974selection} in the case of $A$ bounded and $a\longmapsto h_t(t,x,z,\gamma,a)$ Lipschitz for every $(t,x,z,\gamma)\in [0,T]\times \Xc\times \R^d\times \S_d(\R)$}.}
\begin{align}\label{eq:maxhamiltonian}
[0,T]\times\Xc\times\R^d\ni (t,x,z)\longmapsto \Vc^\star (t,x,z)\in \argmax _{a\in A_t(x,\hat{\sigma}^2_t(x))} h_t(t,x,z,a).
\end{align}

In the most general setting for \eqref{P} considered in this paper, where control on both the drift and the volatility are allowed, to $\xi$, $\partial_s \xi$, $\partial_s f$, and $F$ as above, we associate the system
\begin{align}\label{HJB}\tag{H}
\begin{split}
Y_t=&\ \xi(T,X_{\cdot\wedge T})+\int_t^T   F_r(X,Z_r,\widehat \sigma_r^2,\partial {Y_r^r})\d r- \int_t^T  Z_r \cdot  \d X_r+ K_T^\P-K_t^\P,\;  0\leq t \leq T,\; \Pc(\xb)\text{--}\qs,\\
{\partial Y_t^s}(\omega ):=&\ \E^{\overline \P^{\nu^\star}_{t,x}}\bigg[\partial_s \xi(s,X_{\cdot\wedge T})+\int_t^T\partial_s f_r(s,X, \Vc^\star(r,X,Z_r) )  \d r\bigg],\; (s,t)\in[0,T]^2,\; \omega\in\Omega. 
\end{split}
\end{align}

where $\P^{\nu^\star} \in \Pc(\mathbf{x})$ with $\nu^\star_t:=\Vc^\star(t,X,Z_t) \in \Ac(\xb)$.\medskip

If only drift control is allowed, \emph{i.e.} $\sigma_t(x):=\sigma_t(x,a)$ for all $a\in A$, the weak uniqueness assumption for \eqref{Eq:driftlessSDE} implies $\Pc(\xb)=\{\P\}$. Letting
\begin{align*}
h_t^o(s,x,z,a)&:=f_t(s,x,a)+ b_t(x,a)\cdot\sigma_t(x)^\t z,\; \nabla h_t^o(s,x,z,a):=\partial_s f_t (s,x,a)+ b_t(x,a)\cdot\sigma_t(x)^\t z, \\[0.5em]
H_t^o(x,z,u)  &:=\sup_{a\in A} \{h_t^o(t,x,z,a)\}-u,
\end{align*} 
we show, see Proposition \ref{HJB:to:HJB0}, that \eqref{HJB} reduces to the system
\begin{align}\label{HJB0}\tag{H\textsubscript{o}}
\begin{split}
Y_t&=\xi(T,X_{\cdot\wedge T})+\int_t^T H_r^o(X,Z_r,\partial Y_r^r)\d r-\int_t^T  Z_r \cdot \d X_r,\; 0\leq t \leq T,\; \P\text{--}\as, \\
\partial Y_t^s&=\partial_s \xi(s,X_{\cdot\wedge T})+\int_t^T \nabla h_r^o(s,X,\partial Z_r^s, \Vc^\star(r,X, Z_r)) \d r-\int_t^T\partial Z_r^s \cdot \d X_r,\; 0\leq  t \leq T,\; \P\text{--}\as,\; 0\leq s \leq T.
\end{split}
\end{align}

Though BSDEs have been previously incorporated in the formulation and analysis of time-inconsistent control problems, the kind of systems of BSDEs prescribed by \eqref{HJB} and \eqref{HJB0} are new in the literature. Indeed, in \cite{ekeland2008investment} an auxiliary family of BSDEs was introduced in order to argue, in combination with the stochastic maximum principle, that a given action process complies with the definition of equilibrium \eqref{EquilibriumOld}. This was a reasonable line of arguments as we recall at the time no verification theorem was available. More recently, \cite{wang2019time} studied the case in which the reward functional is represented by a Type-I BSVIEs, a generalisation of the concept of BSDEs. In fact, \cite{wang2019time} argues that when the running cost rate and the terminal cost are time dependent, then the reward functional does satisfy a BSVIE. As it happens, in order to conduct our analysis of \eqref{HJB0}, we also identify a link to such type of equations, namely we obtained that the process $(\partial Y_t^t)_{t\in[0,T]}$ satisfies a Type-I BSVIE, see Lemmata \ref{Lemma:Derivative} and \ref{Lemma:Dynamicsdiagonal}.\medskip

In the general case, the first equation in \eqref{HJB} defines a 2BSDE, whereas the second defines an infinite family of processes, $(\partial Y^s)_{0\leq s\leq T}$, each of which admits a BSDE representation. We chose to introduce such processes via the family $(\P^{\nu^\star}_{t,x})_{(t,x)\in [0,T]\times\Xc}$, since in the first equation one needs an object defined on the support of every $\P\in \Pc(\xb)$. Recall that when the volatility is controlled, the supports of the measures in $\Pc(\xb)$ may be disjoint, whereas in the drift control case, the support is always the same. Had we chosen to introduce the BSDE representation, we would have obtained an object defined only on the support of $\overline\P^{\nu^\star}.$ Moreover, as we work with non-dominated probabilities measures, this last choice would not have been consistent with our extended DPP, nor would have sufficed to obtain the rest of our results.

\medskip
The novel features of system \eqref{HJB} mentioned above raise several theoretical challenges. At the core of such system is the coupling arising from the appearance of $\partial Y_t^t$, the \emph{diagonal} process prescribed by the infinite family $(\partial Y^s)_{0\leq s\leq T}$, in the first equation and of $Z$ in the definition of the family. This makes any standard results available in the BSDEs literature immediately inoperative. Therefore, identifying stating sufficient conditions under which well-posedness holds, \emph{i.e.} existence and uniqueness of a solution, needed to be investigated. Part of these duties consists of determining in what sense such a solution exists, see \Cref{Def:sol:systemH} and \Cref{Def:solution:SystemH0} for the general case and the drift control case, respectively.  The well-posedness of \eqref{HJB0}, \emph{i.e.} in the drift control case,, is part of our, rather technical, \Cref{Appendix:well-posedness}. See \Cref{Section:wellposedness} for precise statements of our results and the necessary assumptions.\medskip

We now introduce the concept of solution of a 2BSDE which we will use to state the definition of a solution to System \eqref{HJB}. To do so, we introduce for $(t,r, x)\in [0,T]^2\times \Xc$, $\P\in \Pc(t,x)$ and a filtration $\G$ 
\begin{align}\label{Eq:measures:uptoF}
 \Pc_{t,x}(r,\P,\G):=\Big\{ \P^\prime\in \Pc(t,x): \P^\prime=\P \text{ on } \Gc_r \Big \}.
\end{align}
We will write $\Pc_x(r,\P,\G)$ for $\Pc_{0,x}(r,\P,\G)$. We postpone to \Cref{Section:Analysis} the definition of the spaces involved in the following definitions. \medskip

\begin{definition}\label{Def:solution2bsde0}
For a given process $(\partial Y_t^t)_{t\in[0.T]}$, consider the equation
\begin{align}\label{Eq:dynamic2bsde0}
Y_t=\xi(T,X_{\cdot\wedge T})+\int_t^T  F_r(X,Z_r,\widehat \sigma_r^2,\partial {Y_r^r})\d r- \int_t^T  Z_r \cdot  \d X_r + K_T^{\P}-K_t^{\P},\;  0 \leq t \leq T.
\end{align}
We say $(Y,Z,(K^{\P})_{\P\in \Pc(\xb)})$ is a solution to {\rm2BSDE} \eqref{Eq:dynamic2bsde0} under $\Pc(\xb)$ if for some $p>1$
\begin{enumerate}[label=$(\roman*)$, ref=.$(\roman*)$,wide, labelindent=0pt]
\item {\rm\Cref{Eq:dynamic2bsde0}} holds $\Pc(\xb)\text{--}\qs;$
\item $(Y,Z,(K^{\P})_{\P\in \Pc(\xb)})\in \S_\xb^p\big(\F^{X,\Pc(\xb)}_+\big)\times \H_\xb^p\big(\F^{X,\Pc(\xb)}_+ \big)\times \I_\xb^p\big( (\F^{X,\P}_+)_{\P\in \Pc(\xb)}\big);$
\item the family $(K^{\P})_{\P\in \Pc(\xb)}$ satisfies the minimality condition
\begin{equation}\label{Min:Condition:NoA}
0=\operatorname*{ess\, inf^\P}_{\P^\prime\in \Pc_\xb(t,\P,\F_+^X)} \E^{\P^\prime}\big[K_T^{\P^\prime}-K^{\P^\prime}_t \big|\Fc_{t+}^{X,\P^\prime}\big],\; 0\leq t\leq T,\; \P\text{--}\as,\; \forall \P\in \Pc(\xb).
\end{equation}
\end{enumerate}
\end{definition}

We now state our definition of a solution to \eqref{HJB}.
\begin{definition}\label{Def:sol:systemH}

We say $(Y,Z,(K^\P)_{\P\in\Pc(\xb)},\partial Y)$ is a solution to \eqref{HJB}, if for some $p>1$
\begin{enumerate}[label=$(\roman*)$, ref=.$(\roman*)$,wide, labelindent=0pt]
\item \label{Def:sol:systemH:i} $\big(Y,Z,(K^\P)_{\P\in\Pc(\xb)}\big)$ is a solution to the {\rm 2BSDE} in \eqref{HJB} under $\Pc(\xb);$
\item\label{Def:sol:systemH:ii} $\partial Y\in \S^{p,2}_{\xb}\big(\F^{X,\Pc(\xb)}_+\big);$
\item \label{Def:sol:systemH:iii} there exists $\nu^\star \in \Ac(\xb)$ such that
\[ 0=\E^{\P^{\nu^\star}}\Big[K_T^{\P^{\nu^\star}}-K_t^{\P^{\nu^\star}}\Big|\Fc_{t+}^{X,\P^{\nu^\star}} \Big],\;  0\leq t\leq T,\; \P^{\nu^\star}\text{\rm--}\as \]
\end{enumerate}
\end{definition}

An immediate result about system \eqref{HJB} is that it is indeed a generalisation of the system of PDEs given in \cite{bjork2014theory} in the Markovian framework. This builds upon the fact second-order, parabolic, fully nonlinear PDEs of HJB type admit a non-linear Feynman--Kac representation formula.

\begin{theorem}\label{Thm:PDEsystem:markovian}
Consider the Markovian setting, \emph{i.e.} $\psi_t(X,\cdot )= \psi_t(X_t,\cdot)$ for $\psi= b, \sigma, f, \partial_s f $, and, $\psi(s,X)=\psi(s,X_T)$ for $\psi= \partial_s \xi, \xi$. Assume that
\begin{enumerate}[label=$(\roman*)$, ref=.$(\roman*)$,wide, labelindent=0pt]
\item there exists a unique $A$-valued Borel-measurable map $\overline \Vc^\star(t,x,z,\gamma)$ satisfying
 \[ [0,T]\times\Xc\times\R^d\times \S_d(\R) \ni (t,x,z,\gamma)\longmapsto \overline \Vc^\star(t,x,z, \gamma)\in \argmax_{a\in A} \big \{ \mathsf{h}_t(t,x,z,\gamma,a)\big\};\]
 
\item for $(s,t,x)\in [0,T)\times[0,T]\times \R^d:=\Oc$, there exists $(v(t,x),J(s,t,x))\in \Cc_{1,2}([0,T]\times \R^d)\times \Cc_{1,1,2}([0,T]^2\times \R^d)$ classical solution to the system\footnote{ Following \cite{bjork2017time} , for a function $(s,\cdot)\longmapsto \varphi(s,\cdot)$, $\varphi^s(\cdot)$, stresses that the $s$ coordinate is fixed.}
\begin{align}\label{Eq:HJB-BKM}
\begin{cases}
  \partial_t V(t,x)+  H(t,x,\partial_{x}V(t,x),\partial_{xx}V(t,x),\partial_s\Jc(t,t,x))=0,\; (s,t,x)\in\Oc,   \\[0.8em]
  \partial_t \Jc^s(t,x)+  \mathsf{h}^s(t,x,  \partial_x \Jc^s(t,x), \partial_{xx} \Jc^s(t,x),\nu^\star(t,x))=0,\; (s,t,x)\in\Oc,\\[0.8em]
V(T,x)=\xi(T,x),\;  \Jc^s(T,x)= \xi(s,x),\; (s,x)\in [0,T]\times \R^d.
\end{cases}
\end{align}
where $\nu^\star(t,x):=\overline \Vc^\star(t,x,\partial_x V(t,x), \partial_{xx} V(t,x))$.

\item $v$, $\partial_x v$, $J$, $\partial_x J$ and $\overline f(s,t,x):= f(s,t,x,\nu^\star(t,x))$ have uniform exponential growth in $x$\footnote{For $x\in \R^d$, $|x|_1=\sum_{i=1}^d |x_i|$.}, \emph{i.e.}
\[\exists C>0, \; \forall (s,t,x)\in[0,T]^2\times\R^d,\; |v(t,x)|+|\partial_x v(t,x)|+|J(s,t,x)|+|\partial_x J(s,t,x)|+ |\overline f(s,t,x)|\leq C \exp(C |x|_1),\]
\end{enumerate}

Then, a solution to the {\rm 2BSDE} in \eqref{HJB} is given by
\begin{align*}
Y_t:=v(t,X_t), \; Z_t:= \partial_x v(t,X_t), \; K_t:=\int_0^t k_r\d r,\; \partial Y_t^s:= \partial_s J^s(t,X_t),
\end{align*}
where $k_t:= H(t,X_t, Z_t,\Gamma_t , \partial Y_t^t) - F_t(X_t,Z_t, \sigmah_t^2, \partial Y_t^t )  -\frac{1}{2} \Tr[ \sigmah_t^2 \Gamma_t]$ and $\Gamma_t:=\partial_{xx} v(t,X_t)$. Moreover, if for $\nu^\star_t:=\overline \Vc^\star(t,X_t,Z_t, \Gamma_t)$ there exists $\P^\star\in \Pc(\xb,\nu^\star)$, then $(Y,Z,K,\partial Y)$ and $\nu^\star$ are a solution to \eqref{HJB}.
\end{theorem}

\begin{remark}
We highlight the assumption on $\Pc(\xb,\nu^\star)$ is satisfied, for instance, if the map $x\longmapsto \sigma \sigma^\t_t(x,\nu^\star(t,x))$ is continuous for every $t$, see \emph{\cite[Theorem 6.1.6]{stroock2007multidimensional}}. Note that the latter is a property of \eqref{Eq:HJB-BKM} itself. In addition, we also remark that under the assumptions in the verification theorem in {\rm\cite[{Theorem 5.2}]{bjork2010general}}, it is immediate that \eqref{HJB} admits a solution. 
\end{remark}

The next sub-sections present the remaining of our results. The first of them is about the necessity of our system. We show that given an equilibrium and the associated game value function, one can construct a solution to \eqref{HJB}. The second result is about the sufficiency of our system, \emph{i.e.} a verification result. In words, it says that from a solution to \eqref{HJB} one can recover an equilibrium. Our last results are about the well-posedness of the system \eqref{HJB} when volatility control is forbidden which ultimately yields the existence and uniqueness of equilibria for \eqref{P}.

\subsection{Necessity of (\textcolor{red}{H})}

{\color{black} The next result, familiar to those acquainted with the literature on optimal stochastic control, is new in the context of time-inconsistent control problems for sophisticated agents. Up until now, the study of such problems, regardless of the notion of equilibrium considered, remained limited to a verification argument, and the study of multiplicity of equilibria. Ever since the work of \cite[Section $5$]{bjork2017time}, it had been conjectured that, in a Markovian setting, given an equilibrium $\nu^\star$ and its value function $v$, the latter would satisfy the associated system of PDEs \eqref{Eq:HJB-BKM}, and $\nu^\star$ would attain the supremum in the associated Hamiltonian.  Nevertheless, according to \cite{bjork2017time} this remained an open and difficult problem.  Fortunately, capitalising on the DPP satisfied by any equilibrium and our probabilistic approach, we are able to present a proof of this claim in a general non-Markovian setting. }\medskip

For our results to hold, we need the following assumption, standard in the context of 2BSDEs, see \citeauthor*{possamai2015stochastic} \cite{possamai2015stochastic}.

\begin{assumption}\label{AssumptionB}
\begin{enumerate}[label=$(\roman*)$, ref=.$(\roman*)$,wide, labelindent=0pt]
\item  \label{AssumptionB:iii} There exists $p>1$ such that for every $(s,t,x)\in [0,T]^2\times \Xc$
\begin{align*}
\sup_{\P\in \Pc(t,x)} \E^\P\bigg[|\xi(T,X)|^p+ |\partial_s \xi(s,X)|^p+\int_t^T |F_r(X,0,\widehat \sigma^2_r ,0)|^p +|\partial_s f_r(s,X,\nu^\star_r)|^p\d r\bigg] < \infty.
\end{align*}
\item \label{AssumptionB:i} $ \R^d\ni z \longmapsto F_t(x,z,\Sigma,u)$ is Lipschitz continuous, uniformly in $(t,x,\Sigma,u)$, \emph{i.e.} there exists $C>0$ s.t.
\begin{align*}
\forall (z,z^\prime)\in   \R^d\times \R^d,\; |F_t(x,z,\Sigma,u)-F_t(x,z',\Sigma,u)|\leq C|\Sigma^{1/2}(z-z')|,\; \forall (t,x,\Sigma,u)\in [0,T]\times \Xc\times \Sigma_t(x)\times \R.
\end{align*}
\item \label{AssumptionB:ii} $ \R^d\ni z \longmapsto \Vc^\star (t,x,z)\in A$ is Lipschitz continuous, uniformly in $(t,x)\in [0,T]\times \Xc$, \emph{i.e.} there exists $C>0$ s.t.
\begin{align*}
\forall (z,z^\prime)\in  \R^d\times \R^d,\; \big|\Vc^\star (t,x,z)- \Vc^\star (t,x,z^\prime)\big|\leq C |z-z^\prime|, \forall (t,x)\in [0,T]\times \Xc.
\end{align*}
\end{enumerate}
\end{assumption}

\begin{theorem}[Necessity]\label{Theorem:Necessity}
Let {\rm\Cref{AssumptionA}} and {\rm\Cref{AssumptionB}} hold. Given $\nu^\star \in \Ec(\xb)$, one can construct $(Y,Z,(K^\P)_{\P\in \Pc(\xb)}, \partial Y)$ solution to \eqref{HJB}, such that for any $t\in [0,T]$ and $\Pc(\xb)\text{\rm--}\qe\; x\in \Xc$
 \[
v(t,x)=\sup_{\P\in\Pc(t,x)} \E^{ \P } \big[Y_t \big].
\] Moreover, $\nu^\star$ satisfies {\rm\Cref{Def:sol:systemH}\ref{Def:sol:systemH:iii}}, \emph{i.e.} $\nu^\star$ is a maximiser of the Hamiltonian.
\end{theorem}

\begin{remark}\begin{enumerate}[label=$(\roman*)$, ref=.$(\roman*)$,wide, labelindent=0pt]
\item {\color{black} We stress that in light of {\rm \Cref{Theorem:Necessity}}, every equilibrium must necessarily maximise the Hamiltonian associated to \eqref{P}. This is, to the best of our knowledge, the first time such a statement is rigorously justified in the framework of time-inconsistent control problems at the level of generality of this paper.

\item 
Even for Markovian time-consistent control problems, we recall that necessity results are quite technical and typically require the theory of viscosity solutions, see {\rm \citeauthor*{fleming2006controlled} \cite{fleming2006controlled}}.  To appreciate the scope of {\rm \Cref{Theorem:Necessity}}, we recall that Markovian {\rm BSDEs} $($resp.  {\rm 2BSDEs)} coincide with viscosity $($resp. Sobolev type$)$ solutions of {\rm PDEs} $($resp. path-dependent {\rm PDEs)}, see {\rm \cite[Theorem 5.5.8]{zhang2017backward}} $($resp. {\rm\cite[Proposition 11.3.8]{zhang2017backward})}.}

\item  {\color{black} We also comment on {\rm \cite[Theorem 3.11]{lindensjo2019regular}} which states that given a regular equilibria, see {\rm \cite[Definition 3.7]{lindensjo2019regular}}, one can define a classic solution the {\rm PDE} system in {\rm \cite{bjork2017time}},  \emph{i.e.} {\rm \eqref{Eq:HJB-BKM}}.  In the Markovian setting of {\rm \cite{lindensjo2019regular}}, regular equilibria render, by definition, smooth classic solutions to the value function and the decoupled payoff functionals. Not surprisingly, one can construct a classic solution to {\rm \eqref{Eq:HJB-BKM}} by means of It\^o's formula.  However, even for time-consistent problems this assumption rarely holds. Moreover, regular equilibria, which are feedback Markovian, are \emph{a priori} required to be continuous. This contrast with our non-Markovian framework and the fact that we take admissible actions and equilibria to only be, in general, measurable.}
\end{enumerate}
\end{remark}

\subsection{Verification}

As is commonplace for control problems, we are able to prove the sufficiency of our system. Indeed, our notion of equilibrium is captured by solutions to System \eqref{HJB}. Our result is not the first of its kind, though our framework allows us to state a fairly simple proof with clear arguments. For instance, 
the proof of \cite[Theorem 6.2]{wei2017time} requires laborious arguments, as a consequence of the notion of equilibrium considered, and relies heavily on PDE arguments. Our theorem requires the following set of assumptions.
\begin{assumption}\label{AssumptionC}
\begin{enumerate}[label=$(\roman*)$, ref=.$(\roman*)$,wide, labelindent=0pt]
\item $x\longmapsto \Phi(r,x):=\partial_s \xi(r,x)+\int_r^T  \partial_s f_u (r,x,a)\d u$ is continuous in $x$. \label{AssumptionC:iib}
\item $ s \longmapsto \xi(s,x)$ \emph{(}resp. $s\longmapsto f_t(s,x,a))$ is continuously differentiable uniformly in $x$ $($resp. in $(t,x,a));$\label{AssumptionC:ii} 
\end{enumerate}
\end{assumption}

\begin{theorem}[Verification]\label{Verification}
Let {\rm\Cref{AssumptionB}} and {\rm\Cref{AssumptionC}} hold. Let $(Y,Z,(K^\P)_{\P\in\Pc(\xb)},\partial Y)$ be a solution to \eqref{HJB} as in {\rm \Cref{Def:sol:systemH}} with $\nu^\star_t:=\Vc^\star(t,X_t,Z_t)$. Then, $\nu^\star \in \Ec(\xb)$ and for $\Pc(\xb)\text{--}\qe\; x\in \Xc$
\[ 
v(t,x)=\sup_{\P\in \Pc(t,x)} \E^{ \P } \big[Y_t \big].
\]
\end{theorem}

We stress that together, \Cref{Theorem:Necessity} and \Cref{Verification} imply that System \eqref{HJB} is fundamental for the study of time-consistent stochastic control problems for sophisticated agents.

\subsection{Well-posedness}\label{Section:wellposedness}

Our analysis would not be complete without a well-posedness result. The result we present is limited to the drift control case, see  \Cref{Appendix:well-posedness}. In this framework, there is a unique weak solution to \eqref{Eq:driftlessSDE} which we will denote by $\P$. In the context of PDEs, under a different and stronger notion of equilibrium, a well-posedness result for the corresponding system of PDEs was given in \cite{wei2017time}. Nonetheless, as we present a probabilistic argument as opposed to an analytic one, our proof makes substantial improvements in both weakening the assumptions as well as the presentation and readability of the arguments. We state next our definition of a solution to \eqref{HJB0}.\medskip

\begin{definition}\label{Def:solution:SystemH0}
The pairs $(Y,Z),$ and $(\partial Y,\partial Z)$ are a \emph{solution} to the system \eqref{HJB0} if for some $p>1$
\begin{enumerate}[label=$(\roman*)$, ref=.$(\roman*)$,wide, labelindent=0pt]
\item $(Y,Z)\in \S^{p}_\xb(\F^{X,\P}_+) \times \H^{p}_\xb(\F^{X,\P}_+)$ satisfy the first equation in \eqref{HJB0} $\P$--$\as;$
\item $(\partial Y, \partial Z)\in \S^{p,2}_\xb(\F^{X,\P}_+) \times \H^{p,2}_\xb(\F^{X,\P}_+)$ and $(\partial Y^s, \partial Z^s)$ satisfies the second equation in \eqref{HJB0} $\P$--$\as, $ for any $s\in [0,T];$

\item there exists $\nu^{\star}\in \Ac(\xb)$ such that 
\begin{align*} H_t^o(X,Z_t,\partial Y_t^t)=h_t^o(t,X,Z_t,\partial Y_t^t,\nu^{\star}_t) \; \d t \otimes \d\P \text{--} \ae \text{ \rm on } [0,T]\times \Xc.
\end{align*}
\end{enumerate}

\end{definition}

Our well-posedness result in the uncontrolled volatility case is subject to the following assumption.

\begin{assumption}\label{AssumptionD}
\begin{enumerate}[label=$(\roman*)$, ref=.$(\roman*)$,wide, labelindent=0pt]
\item $\sigma_t(x):=\sigma_t(x,  a)$ for any $a\in A$, \emph{i.e.} the volatility is not controlled$;$
\item $ s \longmapsto \xi(s,x)$ \emph{(}resp. $s\longmapsto f_t(s,x,a))$ is continuously differentiable uniformly in $x$ $($resp. in $(t,x,a));$
\item  $z\longmapsto H_t^o(x,z,u)$ is Lipschitz uniformly in $(t,x,u) $, \emph{i.e.}
\begin{align*}
\exists C>0,\; \forall (z,z')\in\R^d\times\R^d,\; |H_t^o(x,z,u)-H_t^o(x,z',u)|\leq C|\sigma_t(x)^{1/2}(z-z^\prime)|,\; \forall (t,x,u)\in [0,T]\times \Xc \times \R;
\end{align*}

\item \label{AssumptionD:ii} $ \R^d\ni z \longmapsto \Vc^\star (t,x,z)\in A$ is Lipschitz continuous, uniformly in $(t,x)\in [0,T]\times \Xc$, \emph{i.e.} there exists $C>0$ s.t.
\begin{align*}
\forall (z,z^\prime)\in  \R^d\times \R^d,\; \big|\Vc^\star (t,x,z)- \Vc^\star (t,x,z^\prime)\big|\leq C |\sigma_t(x)^{1/2}(z-z^\prime)|, \forall (t,x)\in [0,T]\times \Xc.
\end{align*}

\item  $(z,a) \longmapsto \nabla h_t^o(s,x,a)$ is Lipschitz continuous, uniformly in $(s,t,x)$, \emph{i.e.} there exists $C>0$ such that for all $(s,t,x)\in [0,T]^2\times \Xc$
\begin{align*}
\forall (z,z',a,a')\in\R^d\times\R^d\times A^2,\; |\nabla h_t^o(s,x,z,a)-\nabla h_t^o(s,x,z',a')|\leq C(|\sigma_t(x)^{1/2}(z-z')|+|a-a'|);
\end{align*}
\item  there exists $p>1$ such that for any $(s,t,x)\in [0,T]^2\times\Xc$
\begin{align*}
\E^{\P}\bigg[\big|\xi(T,X)\big|^p+ \int_t^T \big| H_t^o(x,0,0)\big|^p \d r\bigg] +  \E^{\P}\bigg[ \big| \partial_s \xi (s,X) \big|^p+  \int_t^T \big|\partial_s f_t(s,x,0,\Vc^\star(t,x,0))\big|^p \d r\bigg] < \infty
\end{align*}
\end{enumerate}
\end{assumption}

\begin{remark}
We would like to comment on the previous set of assumptions. We will follow a classic fix point argument to get the well-posedness of {\rm System} \eqref{HJB0}, thus {\rm \Cref{AssumptionD}} consists of a tailor-made version of the classic requirements to get a contraction in a Lipschitz context. Conditions $(iii)$, $(iv)$ and $(v)$ will guarantee the Lipschitz property of the drivers. Condition $(ii)$ will be exploited to control the coupling between the two {\rm BSDEs}.
\end{remark}

The technical but simple results regarding well-posedness are deferred to \Cref{Appendix:well-posedness}. In fact, we are able to establish a well-posedness result for a more general class of systems, see System \eqref{HJB:Abstract}, for which we allow for orthogonal martingales. Moreover, we stress that coupled systems as the ones considered in this work, where the coupling is via an uncountable family of BSDEs, have not been considered before in the literature.

\begin{theorem}[Wellposedness drift control]\label{Thm:wellposedness:driftcontrol}
Let {\rm \Cref{AssumptionD}} hold with $p=2$. There exists a unique solution, in the sense of {\rm\Cref{Def:solution:SystemH0}}, to \eqref{HJB0} with $p=2$. 
\end{theorem}

{\color{black} Lastly, an immediate consequence of the previous results, \Cref{Theorem:Necessity} and \Cref{Verification}, we obtain the existence and uniqueness of equilibrium actions for \eqref{P} in the drift control case.

\begin{theorem}[Uniqueness of equilibria]\label{Thm:uniqueness of equiilibria}
Let {\rm \Cref{AssumptionD}} hold with $p=2$. There exists a unique equilibria for \eqref{P} given by $\nu^\star_\cdot=\Vc^\star(\cdot ,X, Z_\cdot)$, where $\Vc^\star$ denotes the unique Borel-measurable map that maximises $h^o$.
\end{theorem}}

We would like to mention here that the assumption $p=2$ is by no means crucial in our analysis and our results hold in the general case $p>1$, a fact that should be clear to our readers familiar with the theory of BSDEs. Nevertheless, hoping to keep our arguments simple and\ to not drown them in additional unnecessary technicalities, we have opted to present the case $p=2$ only. In this case, it is easier to distinguish between the essential ideas behind our assumptions, and how they work into the probabilistic framework we propose for the study of \eqref{P} and \eqref{HJB0}.

\section{Example: optimal investment with non-exponential discounting}\label{Section:Examples}

We consider the following non-exponential discounting framework 
\begin{align*}
\xi(s,x)=\varphi(T-s) \tilde \xi(x), \; f_t(s,x,a)=\varphi(t-s)\tilde f_t(x,a),\; (t,s,x,a)\in[0,T]^2\times\Xc\times A,
\end{align*}
where $\tilde \xi$ is a Borel-measurable map, and
\begin{align*}
&\tilde f:[0,T]\times\Xc \times A \longrightarrow \R, \text{ Borel-measurable, with } \tilde f_\cdot(\cdot,a)\; \F^X\text{-optional } \text{for any } a \in  A,\\
&\varphi:[0,T]\longrightarrow \R, \text{ non-negative, differentiable, with}\; \varphi(0)=1.
\end{align*}
Let us assume in addition $d=m=1$, $A= \R_+^\star\times \R_+^\star$ and a Markovian framework. We define an action process $\nu$ as a $2$-dimensional $\F^X$-adapted process $(\alpha_t(X_t),c_t(X_t))_{t\in[0,T]}$ with exponential moments of all orders bounded by some arbitrary large constant $M$. We let
$\sigma_t(X_t,\nu_t):=\alpha_t(X_t)$, $ b_t(X_t,\nu_t):=\beta +\alpha_t^{-1}(X_t)(rX_t-c_t(X_t))$. Consequently, for $\M=(\P,\nu)\in \Mf(s,x)$, with $\Mf(s,x)=\{ (\P,\nu) \in \Prob(\Omega)\times \Ac(s,x,\P)\}$
\begin{align*}
\d X_t^{s,x,\nu} =\alpha_t\big(\beta +\alpha^{-1}_t(rX_t-c_t)\d t+\d W_t^\M\big),\; \overline \P^\nu\text{\rm--}\as
\end{align*}

Here we will study the case where the utility function is given, for all $x\in\R$ by
\begin{align*}
U(x):=\frac{x^{1-\eta}-1}{1-\eta}\mathbf{1}_{\{ \eta\in (0,1)\}}+\log (x)\mathbf{1}_{\{ \eta=1\}},
\end{align*}
so that
\begin{align*}
J(s,t,x,\nu)=\E^{\overline \P^\nu}\bigg[ \varphi(T-s)  U (X_T)+\int_t^T\varphi(r-s)U(c_r) \d r\bigg].
\end{align*}

This model, studied for specific choices of $U$ and $\varphi$ in \cite{ekeland2008investment} and \cite{bjork2016time2}, represents an agent who is seeking to find investment and consumption plans, in cash value, $\alpha$ and $c$ respectively, which determine the wealth process $X$. She derives utility only from consumption. At time $s$, the present utility from consumption at time $r$ is discounted according to $\varphi(r-s)$. We present a solution, via verification, based on \Cref{Verification} . The system $\eqref{HJB}$ is given by
\begin{align*}
Y_t=&\ U(X_T)+\int_t^T F_r(X_r,Z_r,\widehat \sigma_r^2,\partial Y_r^r)    \d r -\int_t^T Z_r\cdot \d X_r+K^\P_T-K^\P_t ,\; 0\leq t \leq T,\; \Pc(\xb)\text{\rm--}\qs,\\
{\partial Y_t^s}(\omega ):=&\ \E^{\overline \P^{\nu^\star}_{t,x}}\bigg[\partial_s \varphi(T-s)U(X_T)+\int_t^T\partial_s \varphi(r-s) U(c^\star(r,X_r,Z_r))\d r\bigg],\; (s,t)\in[0,T]^2,\; \omega\in\Omega. 
\end{align*}
where for $(t,x,z,\gamma,\Sigma)\in [0,T]\times \Xc \times \R \times \R \times \R_+^\star$
\begin{align*}
F_t(x,z,\Sigma,u) =\sup_{(c,\alpha)\in \R_+^\star\times \{\alpha^2=\Sigma\}} \big\{ (rx +\beta \alpha - c)z +U(c) \big\}-u;\;
H_t(x,z,\gamma)=\sup_{\Sigma \in \R_+^\star}\bigg\{ F_t(x,z,\Sigma,u)+\frac{1}{2}\Sigma \gamma \bigg\}.
\end{align*}

\begin{proposition}
Let $(Y,Z,K)$ and $\nu^\star$ be given by
\[Y_t:=a(t)U(X_t)+b(t),\; Z_t:=a(t)X_t^{-\eta},\; K_t:=\int_0^t k_r\d r,\;  \nu^\star_t:=(\beta \eta^{-1} X_t, a(t)^{-\frac{1}{\eta}}X_t),\]
where 
\[
k_t:= H(t,X_t, Z_t,\Gamma_t , \partial Y_t^t) - F_t(X_t,Z_t, \sigmah_t^2, \partial Y_t^t )  -\frac{1}{2} \Tr[ \sigmah_t^2 \Gamma_t], \; \Gamma_t:=-\eta a(t)X_t^{-(1+\eta)},\; t\in[0,T],
\] $\partial Y^s$ is defined as above, and $a(t)$ and $b(t)$ as in {\rm \Cref{Eq:examplea}}, which we assume has a unique and continuous solution. Then, there exists $\P^\star\in \Pc(\xb,\nu^\star)$, $(Y,Z,K,\partial Y)$ define a solution to \eqref{HJB} and $(\P^\star,\nu^\star)$ is an equilibrium model. Moreover
\begin{align*}
\d X_t= X_t\Big(\big( r+\beta^2\eta^{-1} + a(r)^{-\frac1{\eta}}\big)\d t+ \beta \eta^{-1}  \d W_t\Big) ,\; 0\leq t \leq T,\; \overline \P^{\nu^\star}\text{\rm--}\as
\end{align*}
\end{proposition}
\begin{proof}
By computing the Hamiltonian, we obtain $(c^\star(t,x,z),\alpha^\star(t,x,z,\gamma)):=(z^{-\frac{1}\eta}, | \beta z \gamma^{-1}|)$ define the maximisers in $F$ and $H$ respectively. Therefore, in this setting, the 2BDE in \eqref{HJB} can be rewritten as for any $\P\in \Pc(\xb)$
\begin{align*}
Y_t=U(X_T)+\int_t^T \Big(Z_r(rX_r+ \beta \sigmah_r -Z_r^{-\frac{1}\eta})+ U(Z_r^{-\frac{1}\eta})-\partial Y_r^r\Big)    \d r -\int_t^T Z_r\cdot \d X_r+K_T^\P-K_t^\P, \; 0\leq t \leq T,\; \P\text{\rm--}\as
\end{align*}

Moreover, taking $Y, Z$ and $K$ as in the statement we obtain $\nu^\star_t:=(\alpha_t^\star,c_t^\star)=(\beta \eta^{-1} X_t, a(t)^{-\frac{1}{\eta}}X_t)$, $t\in[0,T]$. Note that the map $(t,x,z,u)\longmapsto F_t(x,z,\Sigma,u)$ is clearly continuous for fixed $\Sigma$ and that the processes $X, Y,Z$ and $\Gamma$ are continuous in time for fixed $\omega$. This yields, as in the proof of \Cref{Thm:PDEsystem:markovian}, that the process $K$ satisfies the minimality condition under every $\P\in \Pc(\xb)$. Moreover, note that $x\longmapsto \alpha^\star_t(x)$ is continuous for all $t\in[0,T]$. Therefore, there exists $\P^{\nu^\star}\in \Ac(\xb,\nu^\star)$ such that
\begin{align*}
\d X_t= X_t\Big(\big( r+\beta^2\eta^{-1} + a(r)^{-\frac1{\eta}}\big)\d t+ \beta \eta^{-1}  \d W_t\Big) ,\; 0\leq t \leq T,\; \overline \P^{\nu^\star}\text{\rm--}\as
\end{align*}
Moreover, we may find $a(t)$ and $b(t)$ given by \Cref{Eq:examplea} so that for any $\P\in \Pc(\xb)$
\begin{align*}
Y_t&=U(X_T)+\int_t^T h_r(r,X_r,a(r)X_r^{-\eta},a(t)^{-\frac{1}{\eta}}X_t)-\partial Y_r^r \d r -\int_t^T a(r)X_r^{-\eta} \cdot \d X_r+K_T-K_t,\\
&=U(X_T)+\int_t^T F_r(X_r,X_r,a(r)X_r^{-\eta},\sigmah^2_r,\partial Y_r^r) \d r -\int_t^T a(r)X_r^{-\eta}\cdot \d X_r+K_T-K_t, \; 0\leq t \leq T,\; \P\text{\rm--}\as,
\end{align*}
where we exploited the fact $\partial Y$ satisfies \eqref{Eq:BSVEverification}. Also note that given our choice of $\nu^\star$, $K^{\P^{\nu^\star}}=0$. We are left to argue the integrability of $Y,Z,K$. This follows as in the proof of \Cref{Thm:PDEsystem:markovian} as the uniform exponential growth assumption is satisfied, since $a(t)$ and $b(t)$ are by assumption continuous on $[0,T]$. The integrability follows as the action processes are assumed to have exponential moments of all orders bounded by $M$. With this we obtained that $(Y,Z,K)$ is a solution to the 2BSDE in \eqref{HJB}. The integrability of $\partial Y$ is argued as in the proof of \eqref{Theorem:Necessity}.
\end{proof}

\section{Dynamic programming principle}\label{Section:DPP}

This section is devoted to the proof of Theorem \ref{Theorem:DPP:Limit}, namely we wish to obtain the corresponding \emph{extended} version of the \emph{dynamic programming principle}. We begin with a sequence of lemmata that will allow us to study the local behaviour of the value of the game. These results are true in great generality and require mere extra regularity of the running and terminal rewards in the type variable.\medskip

Throughout this section we assume there exists $ \nu^\star\in \Ec(\xb)$. We stress that no assumption about uniqueness of the equilibrium will be imposed. Therefore, in the spirit of keeping track of the notation, for $(t,x)\in [0,T]\times\Xc $ we recall
\[v(t,x)=J(t,t,x,\nu^{\star})=\E^{\overline \P^{\nu^\star}_{t,x}}\bigg[\int_t^T f_r(t,X,\nu_r^\star)\d r+ \xi(t,X_{\cdot \wedge T}) \bigg],\]
and  $(\P^{\nu^\star}_{t,x})_{(t,x)\in [0,T]\times \Xc}$ denotes the unique solution to the martingale problem for \eqref{Eq:driftlessSDE}, with initial condition $(t,x)$ and fixed action process $\nu^\star$. Similarly, for $\omega=(x,\text{w},q) \in \Omega$, $\{\sigma, \tau\} \subset \Tc_{0,T}$, with $\sigma\leq \tau$, and $\nu\in \Ac(\sigma,x)$ we also set
\[
v(\sigma,X)(\omega):=v(\sigma(\omega),x_{\cdot \wedge \sigma(\omega)} ),\;\text{and } J(\sigma, \tau, X,\nu)(\omega):=J(\sigma(\omega),\tau(\omega),x_{\cdot \wedge \sigma(\omega)} ,\nu ).
\]

Our first result consists of a one step iteration of our equilibrium definition.

\begin{lemma}\label{Lemma:DPP}
Let $\nu^\star \in \Ec(\xb)$ and $v$ the value associated to $\nu^\star$ as in \eqref{P}. Then, for any $(\eps,\ell,t,\sigma,\tau)\in \R_+^\star \times (0,\ell_\eps)\times [0,T]\times \Tc_{t,t+ \ell}\times \Tc_{t,t+ \ell}$ with $\sigma
\leq \tau,$ and  $\Pc(\xb)\text{--}\qe\; x \in \Xc$
\begin{align}\label{DPP:Inequalities}
\begin{split}
v(\sigma,x)&\leq  \sup_{\nu \in \Ac(\sigma,x)} \ J(\sigma,\sigma, x, \nu \otimes_{\tau}\nu^\star ), \\
v(\sigma,x)&\geq   \sup_{\nu \in \Ac(\sigma,x)}  \E^{\overline \P^\nu}\bigg[v(\tau,X)+ \int_{\sigma}^{\tau} f_r(\sigma,X,\nu_r)\d r+ J(\sigma,\tau,X,\nu^\star ) - J(\tau,\tau, X,\nu^\star)\bigg]-\eps \ell.
\end{split}
\end{align}
\end{lemma}
\begin{proof}
The first inequality is clear. Indeed for $\Pc(\xb)\text{--}\qe\; x\in \Xc$, $\overline \P^{\nu^\star}_{\sigma(\omega),x}\in \Pc(\sigma(\omega),x)$ and therefore $\nu^{\star} \in \Ac(\sigma, x)$.

\medskip
To get the second inequality note that for $(\eps,\ell,t,\sigma,\tau)\in \R_+^\star \times (0,\ell_\eps)\times [0,T]\times \Tc_{t,t+ \ell}\times \Tc_{t,t+ \ell}$ with $\sigma \leq \tau,$ $\Pc(\xb)\text{--}\qe\; x\in \Xc$ and $\nu \in \Ac(\sigma,x)$
\begin{align*}
v(\sigma,x)&=J(\sigma,\sigma,x, \nu^\star)\geq J(\sigma,\sigma,x, \nu\otimes_{\tau}\nu^\star )-\eps \ell\\
&=\ \E^{\overline \P^{\nu\otimes_{\tau}\nu^\star}} \bigg[ \int_{\sigma}^{\tau} f_r(\sigma,X,( \nu\otimes_\tau \nu^\star)_r)\d r +\int_{\tau}^T  f_r(\sigma,X, ( \nu\otimes_\tau \nu^\star)_r)\d r+ \xi(\sigma,X_{\cdot \wedge T})\bigg]-\eps \ell\\
&=\ \E^{\overline \P^\nu}\bigg[ v(\tau,X)+ \int_{\sigma}^{\tau} f_r(\sigma,X,\nu_r)\d r+ J(\sigma,\tau,X,\nu^{\star} ) - J(\tau,\tau, X,\nu^{\star})\bigg]-\eps \ell,
\end{align*} 
where the last equality follows by conditioning and the $\Fc_\tau$ measurability of all the terms. Indeed, in light of \Cref{Lemma:forwardcondition}, an r.c.p.d. of $\overline \P^{\nu\otimes_\tau \nu^\star}$ with respect to $\Fc_\tau$, evaluated at $x$, agrees with $\overline \P^{\nu^\star}_{\tau(x),x}$, the weak solution to \eqref{Eq:driftlessSDE} with initial condition $(\tau,x)$ and action $\nu^\star$, for $\overline \P^{\nu\otimes_\tau\nu^\star}\text{--}\ae\; x\in \Xc$. As all the terms inside the expectation are $\Fc_\tau$-measurable, the previous holds for $\overline \P^{\nu}\text{--}\ae\; x\in \Xc$.
\end{proof}

From the previous result, we know that equilibrium models satisfy a form of $\eps$-optimality in a sufficiently small window of time. We now seek to gain more insight from iterating the previous result. This will allow us to move forward the time window into consideration.\medskip

In the following, given $(\sigma, \tau) \in \Tc_{t,T}\times \Tc_{t,t+\ell}$, with $\sigma \leq \tau$, we denote by $\Pi^\ell:=(\tau_i^\ell)_{i\in\{1,\dots,n_\ell\}}\subseteq \Tc_{t,T}$ a generic partition of $[ \sigma, \tau]$ with mesh smaller than $\ell$, \emph{i.e.} for $n_{\ell} :=\ceil[\big]{(\tau-\sigma)/\ell}$, $\sigma=:\tau^\ell_0\leq \cdots\leq  \tau^\ell_{n^\ell}:=\tau,$ $\forall \ell$, and $\sup_{ i\in\{1,\dots,  n_\ell\}} |\tau^\ell_i-\tau^\ell_{i-1}|\leq \ell$. We also let $\Delta \tau_i^\ell:= \tau_{i}^\ell-\tau_{i-1}^\ell$. The previous definitions hold $\omega$-by-$\omega$. \medskip

\begin{proposition}\label{DPP:Iterated}
Let $\nu^\star\in \Ec(\xb)$ and $\{\sigma, \tau\}\subset \Tc_{t,T}$, with $\sigma\leq \tau$. Fix $\eps>0$ and some partition $\Pi^{\ell}$ with $\ell<\ell_\eps$. Then for $\Pc(\xb)\text{--}\qe\; x \in \Xc$
\begin{align*}
v(\sigma,x)\geq \sup_{\nu \in \Ac(\sigma,x)}   \E^{\overline \P^\nu}\bigg[ v(\tau,X)+\sum_{i=0}^{n_{\ell}-1} \int_{\tau_i^\ell}^{\tau_{i+1}^\ell} f_r(\tau_i^\ell,X,\nu_r)\d r&+ J(\tau_i^\ell,\tau_{i+1}^\ell, X,\nu^\star) - J(\tau_{i+1}^\ell,\tau_{i+1}^\ell, X,\nu^\star) -n_\ell \eps \ell\bigg].
\end{align*}
\end{proposition}
\begin{proof}
A straightforward iteration of \Cref{Lemma:DPP} yields that for $\Pc(\xb)\text{--}\qe\; x \in \Xc$
\begin{align*}
v(\sigma,x) \geq  \sup_{\nu \in \Ac(\sigma,x)}  &   \E^{\overline \P^\nu}\bigg[ \int_{\sigma}^{\tau_1^\ell} f_r(\sigma,X,\nu_r)\d r+ J(\sigma,\tau_1^\ell, X, \nu^\star) - J(\tau_1^\ell,\tau_1^\ell, X,\nu^\star)+v(\tau_1^\ell,X) - \eps \ell \bigg]\\
\geq \sup_{\nu \in \Ac(\sigma,x)}  & \int_{\Omega}  \bigg( \int_{\sigma}^{\tau_1^\ell} f_r(\sigma,X,\nu_r)\d r+ J(\tau_0^\ell,\tau_0^\ell,X, \nu^\star)- J(\tau_1^\ell,\tau_1^\ell, X,\nu^\star)\\
& +\sup_{\tilde \nu \in \Ac(\tau_1^\ell,\tilde x)}   \E^{{\widetilde \P}^{\tilde \nu}}\bigg[v(\tau_2^\ell, X )+ \int_{\tau_1^\ell}^{\tau_2^\ell} f_r(\tau_1^\ell,X,\tilde \nu_r)\d r+ J(\tau_1^\ell,\tau_2^\ell, X,\nu^\star) - J(\tau_2^\ell,\tau_2^\ell,X,\nu^\star) -2 \eps \ell \bigg]\bigg)\overline \P^\nu(\d \tilde \omega) \\
=\sup_{\nu \in \Ac(\sigma,x)}  &  \E^{\overline \P^\nu}\bigg[ \int_{\sigma}^{\tau_1^\ell} f_r(\sigma,X,\nu_r)\d r+ J(\tau_0^\ell,\tau_1^\ell,X, \nu^\star)- J(\tau_1^\ell,\tau_1^\ell,X,\nu^\star) \\
&+v(\tau_2^\ell,X)+\int_{\tau_1^\ell}^{\tau_2^\ell} f_r(\tau_1^\ell,X,\nu_r)\d r+ J(\tau_1^\ell,\tau_2^\ell, X,\nu^\star) - J(\tau_2^\ell,\tau_2^\ell, X,\nu^\star)-2 \eps \ell \bigg],
\end{align*}
where the second inequality follows by applying the definition of an equilibrium at $(\tau_1^\ell,X)$. Now, the last step follows from \cite[Theorem 4.6.]{karoui2013capacities}, see also \cite[Theorem 2.3.]{nutz2013constructing}, which holds thanks to \cite{karoui2015capacities2}. Indeed, as $\F$ is countably generated and $\Pc(t,x)\neq \varnothing$, for all $(t,x)\in [0,T]\times \Xc$, \cite[Lemmata 3.2 and 3.3]{karoui2015capacities2} hold. The general result follows directly by iterating and the fact the iteration is finite.
\end{proof}

In the same spirit as in the classic theory of stochastic control, a natural question at this point is whether there is, if any, an infinitesimal limit of the previous iteration and what kind of insights on the value function we can draw from it. The next theorem shows than under a mild extra regularity assumption on the running cost, namely \Cref{AssumptionA}\ref{AssumptionA:DPP1}, we can indeed pass to the limit.\medskip

To ease the readability of our main theorem, for $x\in \Xc$, $\{\sigma, \zeta, \tau\} \subset \Tc_{0,T}$ with $\sigma\leq \zeta \leq \tau$, $\nu^\star \in \Ac(\xb)$, $\nu\in \Ac(\sigma,x)$, and any $\Fc_T^X$-measurable random variable $\xi$, we introduce the notation 
\[
\E_\zeta^{\overline \P^\nu, \overline \P^{\nu^\star}}[\xi]:=\E^{\overline \P^{\nu}\otimes_\zeta {\overline \P}^{\nu^\star}_{\zeta,X}} [\xi],
\]
where $\P^{\nu^\star}_{\zeta,X}$ is given by $\omega \longmapsto  \P^{\nu^\star}_{ (\zeta(\omega),x_{\cdot \wedge \zeta(\omega)})}$ and denotes the $\Fc_\zeta$-kernel prescribed by the family of solutions to the martingale problem associated with $\nu^\star$, see \cite[Theorem 6.2.2]{stroock2007multidimensional}. Note in particular $\E_\sigma^{ \overline \P^\nu, \overline \P^{\nu^\star} }[\xi]=\E^{ \overline \P^{\nu^\star}_{\sigma, X} }[\xi]$.

\begin{proof}[Proof of {\rm\Cref{Theorem:DPP:Limit}}]
Let $\eps>0$, $0<\ell<\ell_\eps$ and $\Pi^{\ell}$ be as in the statement of \Cref{DPP:Iterated}. From Proposition \ref{DPP:Iterated} we know that for $\Pc(\xb)\text{--}\qe\; x \in \Xc$
\begin{align}\label{DPP:Limit:Aux1}
v(\sigma,x) \geq &\sup_{\nu \in \Ac(\sigma,x)}  \bigg\{ \E^{\overline \P^\nu}\bigg[v(\tau,X)+\sum_{i=0}^{n_{\ell}-1} \int_{\tau_i^\ell}^{\tau_{i+1}^\ell} f_r(\tau_i^\ell,X,\nu_r)\d r +J(\tau_i^\ell,\tau_{i+1}^\ell,X,\nu^{\star}) - J(\tau_{i+1}^\ell,\tau_{i+1}^\ell, X,\nu^{\star}) -n_\ell \eps \ell  \bigg]\bigg\}\nonumber \\
=&\sup_{\nu \in \Ac(\sigma,x)}  \bigg\{ \E^{\overline \P^\nu}\bigg[ v(\tau,X)+\sum_{i=0}^{n_{\ell}-1} \int_{\tau_i^\ell}^{\tau_{i+1}^\ell} f_r(\tau_i^\ell,X,\nu_r)\d r -n_\ell \eps {\ell}  \bigg]\nonumber \\
&\hspace*{1.5em} +\sum_{i=0}^{n_{\ell}-1} \int_\Omega  \E^{\overline \P^{\nu^\star}_{\tau^\ell_{i+1}(\tilde \omega),X(\tilde \omega)}}\bigg[ \int_{\tau_{i+1}^\ell}^T \big(f_r(\tau_i^\ell,X,\nu_r^\star)-f_r(\tau_{i+1}^\ell,X,\nu_r^\star)\big)\d r +\xi(\tau_{i}^\ell,X)-\xi(\tau_{i+1}^\ell,X)\bigg] \overline \P^\nu(\d \tilde \omega)\bigg\}\nonumber \\
\begin{split}=&\sup_{\nu \in \Ac(\sigma,x)}   \bigg\{ \E^{\overline \P^\nu}\bigg[v(\tau , X)+\sum_{i=0}^{n_{\ell}-1} \int_{\tau_i^\ell}^{\tau_{i+1}^\ell} f_r(\tau_i^\ell,X,\nu_r)\d r -n_\ell \eps \ell \bigg] \\
&\hspace*{1.5em} +\sum_{i=0}^{n_{\ell}-1} \E^{\overline \P^\nu,\overline \P^{\nu^\star}}_{\tau^\ell_{i+1}} \bigg[ \int_{\tau_{i+1}^\ell}^T\big( f_r(\tau_i^\ell,X,\nu_r^\star)-f_r(\tau_{i+1}^\ell,X,\nu_r^\star)\big)\d r +\xi(\tau_{i}^\ell,X_{\cdot\wedge T})-\xi(\tau_{i+1}^\ell,X_{\cdot\wedge T})\bigg]\bigg\},
\end{split}
\end{align}
where we used the definition of $J$ and conditioned. For $(t,x,\nu)\in [0,T]\times \Xc\times \Ac(t,x)$, let $G^t(s):=\int_t^s f_r(s,x,\nu_r)\d r$, $s\in[t,T]$. We set $G(s):=G^\sigma(s)$, so that
\begin{align}\label{DPP:Limit:Aux0}
 \sum_{i=0}^{n_{\ell}-1} \E^{\overline \P^\nu} \bigg[ \int_{\tau_i^\ell}^{\tau_{i+1}^\ell}   f_r(\tau_i^\ell,X,\nu_r)\d r\bigg]&= \sum_{i=0}^{n_{\ell}-1}   \E^{\overline \P^\nu}\bigg[G(\tau_{i+1}^\ell)-G(\tau_{i}^\ell)+\int_{\sigma}^{\tau_{i+1}^\ell}   \big( f_r(\tau_i^\ell,X,\nu_r)-f_r(\tau_{i+1}^\ell,X,\nu_r)\big)\d r\bigg]\nonumber \\
 &=  \sum_{i=0}^{n_{\ell}-1}  \E^{\overline \P^\nu} \Big[G(\tau_{i+1}^\ell)-G(\tau_{i}^\ell)\Big] + \E_{\tau_{i+1}^\ell}^{\overline \P^\nu,\overline \P^{\nu^\star}}  \bigg[ \int_{\sigma}^{\tau_{i+1}^\ell}   \big(  f_r(\tau_i^\ell,X,\nu_r) - f_r(\tau_{i+1}^\ell,X,\nu_r)\big)\d r\bigg].
\end{align}
where the last equality follows from the $\Fc_{\tau_{i+1}^\ell}$-measurability of the integral and \Cref{Thm:Concatenated:M}. Now we observe that we can add the integral terms in \eqref{DPP:Limit:Aux1} and \eqref{DPP:Limit:Aux0}, \emph{i.e.}
\begin{align*}
&\sum_{i=0}^{n_{\ell}-1}\E_{\tau_{i+1}^\ell}^{\overline \P^\nu,\overline \P^{\nu^\star}} \bigg[ \int_{\tau_{i+1}^\ell}^T \big(f_r(\tau_i^\ell,X,\nu_r^\star)-f_r(\tau_{i+1}^\ell,X,\nu_r^\star)\big)\d r \bigg]+\E_{\tau_{i+1}^\ell}^{\overline \P^\nu,\overline \P^{\nu^\star}}\bigg[ \int_{\sigma}^{\tau_{i+1}^\ell}\big(  f_r(\tau_i^\ell,X,\nu_r)-f_r(\tau_{i+1}^\ell,X,\nu_r)\big)\d r\bigg]\\
&=\sum_{i=0}^{n_{\ell}-1}\E_{\tau_{i+1}^\ell}^{\overline \P^\nu,\overline \P^{\nu^\star}}\bigg[ \int_{\sigma}^{T} \big( f_r(\tau_i^\ell,X,(\nu\otimes_{\tau_{i+1}^\ell}\nu^\star)_r)-f_r(\tau_{i+1}^\ell,X,(\nu\otimes_{\tau_{i+1}^\ell}\nu^\star)_r)\big)\d r\bigg].
\end{align*}
Consequently, for $\Pc(\xb)\text{--}\qe\; x \in \Xc$
\begin{align}\label{DPP:Limit:Aux2}
\begin{split}
v(\sigma,x)\geq \sup_{\nu \in \Ac(\sigma,x)}\bigg\{ &\E^{\overline \P^\nu}\bigg[v(\tau, X)+ \sum_{i=0}^{n_{\ell}-1} G(\tau_{i+1}^\ell)-G(\tau_{i}^\ell)-n_\ell \eps \ell \bigg]\\
&+ \sum_{i=0}^{n_{\ell}-1}\E_{\tau_{i+1}^\ell}^{\overline \P^\nu,\overline \P^{\nu^\star}}\bigg[ \int_{\sigma}^T  f_r(\tau_i^\ell,X,(\nu\otimes_{\tau_{i+1}^\ell}\nu^\star)_r)-f_r(\tau_{i+1}^\ell,X,(\nu\otimes_{\tau_{i+1}^\ell}\nu^\star)_r)\d r\bigg]\\
&+ \sum_{i=0}^{n_{\ell}-1}\E_{\tau_{i+1}^\ell}^{\overline \P^\nu,\overline \P^{\nu^\star}}\bigg[\xi(\tau_{i}^\ell,X_{\cdot\wedge T})-\xi(\tau_{i+1}^\ell,X_{\cdot\wedge T})\bigg]\bigg\}.
\end{split}
\end{align}

The idea in the rest of the proof is to take the limit $\ell\longrightarrow 0 $ on both sides of \eqref{DPP:Limit:Aux2}. As $v$ is finite we can exchange the limit with the $\sup$ and study the limit inside. The analysis of all the above terms, except the error term $\ceil[\big]{\tau-\sigma/\ell}\eps {\ell}$, is carried out below. Regarding the error term, we would like to make the following remarks as it is clear that simply letting $\eps$ go to zero will not suffice for our purpose. As $\ell_\eps$ is bounded and monotone in $\eps$, see Remark \ref{Remark:Def:Eq:Ext}, we consider $\ell_0$ given by $\ell_\eps \longrightarrow \ell_0$ as $\eps \longrightarrow 0$. We must consider two cases for $\ell_0$: when $\ell_0=0$ the analysis in the next paragraph suffices to obtain the result; in the case $\ell_0>0$, we can then take at the beginning of this proof $\ell < \ell_0\leq \ell_\eps$, in which case all the sums in \eqref{DPP:Limit:Aux2} are independent of $\eps$, we then first let $\eps$ go to zero so that $\ceil[\big]{\tau-\sigma/\ell}\eps {\ell}\longrightarrow0$ as $\eps \longrightarrow 0$, and then study the limit $\ell\longrightarrow 0$ as in the following. In both scenarii \eqref{Eq:DPP:limit} holds.\medskip

We now carry out the analysis of the remaining terms. To this end, and in order to prevent enforcing unnecessary time regularity on the action process, we will restrict our class of actions to piece-wise constant actions, \emph{i.e.} $\nu_t:=\sum_{ k \in\N^\star} \nu_k \mathbf{1}_{(\varrho_{k-1},\varrho_k]}(t)$ for a sequence of non-decreasing $\F$--stopping times $(\varrho_k)_{k\in\N}$, and random variables $(\nu_k)_{k\in\N^\star}$, such that for any $k\geq1$, $\nu_k$ is $\Fc^X_{\varrho_{k-1}}$-measurable. We will denote by $\Ac^{\text{\rm pw}}(t,x)$ the corresponding subclass of actions. By \cite{karoui2015capacities2} the supremum over $\Ac(t,x)$ and $\Ac^{\text{pw}}(t,x)$ coincide. Indeed, under \Cref{AssumptionA}\ref{AssumptionA:DPP2} and \ref{AssumptionA}\ref{AssumptionA:DPP}, we can apply \cite[Theorem 4.5]{karoui2015capacities2}. Assumption \ref{AssumptionA}\ref{AssumptionA:DPP2}, \emph{i.e.} the Lipschitz-continuity of $a\longmapsto f_t(t,x,a)$, ensures the continuity of the drift coefficient when the space is extended to include the running reward, see \cite[Remark 3.8]{karoui2015capacities2}. Without loss of generality we assume $(\varrho_k)_{k\in\N}\subseteq \Pi^\ell$, as we can always refine $\Pi^\ell$ so that $\nu_{r}=\nu_i$ for $\tau_i^\ell\leq r \leq \tau_{i+1}^\ell$.
\medskip

In the following, we fix $\omega\in \Omega$. A first-order Taylor expansion of the first summation term in \eqref{DPP:Limit:Aux2} guarantees the existence of $\gamma_i^\ell\in (\tau_i^\ell,\tau_{i+1}^\ell),$ $i\in\{0,\dots,n_\ell\}$ such that
\begin{align}\label{DPP:Limit:Aux3:1}
&\left| \sum_{i=0}^{ n_\ell-1} G(\tau_{i+1}^\ell)-G(\tau_{i}^\ell) -  \Delta \tau_{i+1}^\ell \bigg( f_{\tau_i^\ell}(\tau_{i}^\ell,X,\nu_i)+\int_{\sigma}^{\tau_{i+1}^\ell} \partial_s f_r(\tau_{i+1}^\ell,X,\nu_r)\d r\bigg)\right|\\
&=\left| \sum_{i=0}^{ n_\ell -1} \Delta \tau_{i+1}^\ell \bigg( f_{\gamma_i^\ell}(\gamma_i^\ell,X,\nu_i)-  f_{\tau_i^\ell}(\tau_{i}^\ell,X,\nu_i)+\sum_{k=0}^{i} \int_{\tau_k^\ell}^{\tau_{k+1}^\ell\wedge \gamma_i^\ell} \partial_s f_r(\gamma_i^\ell,X,\nu_k)\d r-\int_{\tau_k^\ell}^{\tau_{k+1}^\ell} \partial_s f_r(\tau_{i+1}^\ell,X,\nu_k)\d r\bigg) \right|\nonumber \\
&\leq \sum_{i=0}^{ n_\ell -1}|\Delta \tau_{i+1}^\ell | \bigg( \rho_f(|\Delta \tau_{i+1}^\ell |)+\sum_{k=0}^{i} \int_{\tau_k^\ell}^{\tau_{k+1}^\ell} \rho_{\partial_s f} (|\Delta \tau_{i+1}^\ell |)  \d r  \bigg)\leq 2 T\big( \rho_f(\ell)+\rho_{\partial_s f}(\ell)\big)\xrightarrow{\ell \to 0}0. \nonumber
\end{align}
The equality follows by replacing the expansion of the terms $G(\tau_{i+1}^\ell)$ and the fact $\nu$ is constant between any two terms of the partition. The first inequality follows from \Cref{AssumptionA}\ref{AssumptionA:DPP1}, where $\rho$ and $\rho_{\partial_s f}$ are the modulus of continuity of the maps $t\longmapsto f_t(t,x,a)$ and $s\longmapsto\partial_s f_r(s,x,a)$, for $a$ constant. The limits follows by bounded convergence as the last term is independent of $\omega$. Thus, both expressions on the first line have the same limit for every $\omega \in \Omega$. We claim that for a well chosen sequence of partitions of the interval $[\sigma ,\tau]$
\begin{align}\label{DPP:Limit:Aux3}
\E^{\overline \P^\nu}\bigg[\sum_{i=0}^{ n_\ell-1} G(\tau_{i+1}^\ell)-G(\tau_{i}^\ell)\bigg]
 \xrightarrow{\ell \to 0} \E^{ \overline \P^\nu}\bigg[\int_{\sigma}^\tau f_r(r,X,\nu_r)\d r+\int_{\sigma}^{\tau} \E^{ { \overline \P}^{\nu^\star}_{r,X} }\bigg[\int_{\sigma}^{r} \partial_s f_u(r,X,\nu_u)\d u \bigg]\d r \bigg],
\end{align}
where the integrals on the right-hand side are w.r.t the Lebesgue measure on $[0,T]$, and we recall the term inside $ \E^{{\overline \P}^{\nu^\star}_{r,X }}$ is $\Fc^X_r$-measurable. Indeed, following \citeauthor*{mcshane1986unified} \cite{mcshane1986unified}, for $\ell>0$ fixed there exists, $\omega$-by-$\omega$,  $\widehat \Pi^\ell:=(\hat \tau_i^\ell)_{i\in \{ 1, \dots, n_\ell\}}$ a partition of $[\sigma,\tau]$ such that the Riemann sum in \eqref{DPP:Limit:Aux3:1} evaluated at $\widehat \Pi^\ell$ converges to the Lebesgue integral $\omega$-by-$\omega$. With this, we are left to argue \eqref{DPP:Limit:Aux3}. Recall that so far, our analysis was for $\omega\in \Omega$ fixed, therefore one has to be careful about, for instance, the measurability of the partition $\widehat \Pi^\ell$. An application of Galmarino's test, see \citeauthor*{dellacherie1978probabilities} \cite[Ch. IV. 99--101]{dellacherie1978probabilities}, guarantees that $\hat \tau_i^\ell \in \Tc_{0,T}$ for all $i\in \{ 1, \dots, n_\ell\}$, \emph{i.e.} the random times $\hat \tau_i^\ell$ are in fact stopping times. See \Cref{Lemma:galmarinos} for details. Finally, \eqref{DPP:Limit:Aux3} follows by the bounded convergence theorem. \medskip

Similarly, a first-order expansion of the second term in \eqref{DPP:Limit:Aux2} yields $( \gamma_i^\ell)_{ i \in\{0,\dots, n_\ell\}}$ such that 
\begin{align*}
 &\Bigg|\sum_{i=0}^{ n_\ell-1} \E_{\tau_{i+1}^\ell}^{\overline \P^\nu,\overline\P^{\nu^\star}} \bigg[ \int_{\sigma}^{T}  \Big( f_r\big(\tau_i^\ell,X,(\nu\otimes_{\tau_{i+1}^\ell}  \nu^\star)_r\big)-f_r\big(\tau_{i+1}^\ell,X,(\nu\otimes_{\tau_{i+1}^\ell} \nu^\star)_r\big)+\Delta \tau_{i+1}^\ell \partial_s f_r\big(\tau_{i+1}^\ell,X,(\nu\otimes_{\tau_{i+1}^\ell}  \nu^\star)_r\big) \Big) \d r\bigg] \Bigg| \\ 
&=\Bigg|\sum_{i=0}^{ n_\ell -1} \E_{\tau_{i+1}^\ell}^{\overline \P^\nu,\overline \P^{\nu^\star}}\bigg[ \Delta \tau_{i+1}^\ell \int_{\sigma}^{T} \Big( \partial_s f_r\big(\gamma_{i}^\ell,X,(\nu\otimes_{\tau_{i+1}^\ell} \nu^\star)_r\big)-\partial_s f_r\big(\tau_{i+1}^\ell,X,(\nu\otimes_{\tau_{i+1}^\ell} \nu^\star)_r\big)\Big)\d r\bigg]\Bigg|\leq   T  \rho_{\partial_s f}(\ell)\xrightarrow{\ell \to 0}0. \nonumber
\end{align*}
Since the limits agree, we obtain that for an appropriate choice of $\Pi^\ell$ this term converges to
\[
\E^{\overline \P^\nu}  \bigg[ \sum_{i=0}^{ n_\ell -1}  \Delta \tau_{i+1}^\ell   \E^{{\overline \P }^{\nu^\star}_{\tau_{i+1}^\ell, X}} \bigg[  \int_{\sigma}^T     \partial_s f_u (\tau_{i+1}^\ell,X,(\nu\otimes_{\tau_{i+1}^\ell} \nu^\star)_u)\d u \bigg] \bigg] \xrightarrow{\ell \rightarrow 0}  \E^{\overline \P^\nu}  \bigg[ \int_{\sigma}^\tau \!  \E^{\overline \P^{\nu^\star}_{r,X} } \bigg[  \int_{\sigma}^T  \partial_s f_u(r,X,(\nu\otimes_{r} \nu^\star)_u) \d u\bigg]  \d r \bigg].
\]
Combining the double integrals in \eqref{DPP:Limit:Aux3} and the previous expression we obtain back in \eqref{DPP:Limit:Aux2} that for $\Pc(\xb)\text{--}\qe\; x \in \Xc$
\begin{align*}
v(\sigma,x) \geq \sup_{\nu \in \Ac(\sigma,x)} \E^{\overline \P^\nu}\bigg[ v(\tau, X)+\int_{\sigma}^\tau \bigg[ f_r(r,X,\nu_r)-\E^{{\overline \P}^{\nu^\star}_{r,X} } \bigg[\partial_s \xi(r,X_{\cdot\wedge T}) +\int_r^T  \partial_s f_u(r,X,\nu^\star_u)\d u\bigg]\bigg]\d r\bigg].
\end{align*}
Now for the reverse inequality, note that for $\Pc(\xb)-\qe x\in \Xc$, $\overline \P^{\nu^\star}_{\sigma(\omega),x} \in \Pc(\sigma(\omega),x)$, \emph{i.e.} $\nu^\star \in \Ac(\sigma,x)$. Second, by definition 
\begin{align*}
v(\sigma,x)=\E^{\overline \P^{\nu^\star}_{\sigma,x}} \bigg[ v(\tau,X)+\int_\sigma^T f_r(\sigma, X,\nu^\star)\d r - \int_\tau^T f_r(\tau, X,\nu^\star)\d r +\xi(\sigma,X_{\cdot \wedge T}) - \xi(\tau, X_{\cdot \wedge T})\bigg].
\end{align*}
In light of the regularity of $s\longmapsto f_t(s,x,a)$ and the measurability of $\nu^\star$, Fubini's theorem yield
\begin{align*}
\E^{\overline \P^{\nu^\star}_{\sigma,x}} \bigg[ \int_\sigma^T f_r(\sigma, X,\nu^\star_r)\d r - \int_\tau^T f_r(\tau, X,\nu^\star_r)\d r \bigg] =\E^{\overline \P^{\nu^\star}_{\sigma,x}} \bigg[  \int_\sigma^\tau  \bigg( f_r(r,X,\nu^\star_r) - \E^{\overline \P^{\nu^\star}_{r,X}} \bigg[ \int_u^T \partial_s f_u (r,X,\nu^\star_u)\d u \bigg] \bigg)\d r\bigg],
\end{align*}
where we also use the tower property. Proceeding similarly for $s\longmapsto \xi(s,x)$, we conclude that for $\Pc(\xb)-\qe x\in \Xc$
\begin{align*}
v(\sigma, x)=\E^{\overline \P^{\nu^\star}_{\sigma,x}}\bigg[ v(\tau,X)+\int_{\sigma}^\tau \bigg( f_r(r,X,\nu_r^\star)-\E^{{\overline \P}^{\nu^\star}_{r,X}} \bigg[\partial_s \xi(r,X_{\cdot\wedge T}) +\int_r^T  \partial_s f_u(r,X,\nu^\star_u)\d u\bigg] \bigg)\d r\bigg],
\end{align*}
which gives us the desired equality and the fact that $\nu^\star$ does attain the supremum.
\end{proof}

\begin{remark}\label{Remark:assumpDPP}
Let us comment on the necessity of {\rm\Cref{AssumptionA}}\ref{AssumptionA:DPP} for our result to hold. As commented in the proof, a crucial step in our approach is that the $\sup$ in \eqref{DPP:Limit:Aux2} attains the same value over $\Ac(t,x)$ and $\Ac^{\rm{pw}}(t,x)$. For this we used {\rm \cite[Theorem 4.5]{karoui2015capacities2}} which holds in light of {\rm\Cref{AssumptionA}}\ref{AssumptionA:DPP}. Indeed, after inspecting the proof of {\rm \cite[Theorem 4.5]{karoui2015capacities2}}, one sees that {\rm \cite[Assumption 1.1]{karoui2015capacities2}} guarantees pathwise uniqueness of the solution to an auxiliary {\rm SDE}. 
However, as pointed out also in {\rm\citeauthor*{claisse2016pseudo} {\cite[Section 2.1]{claisse2016pseudo}}}, the previous condition can be relaxed to weaker conditions which imply weak uniqueness but are beyond the scope of the current paper. 
\end{remark}

\begin{remark}\label{Remark:DPP:Limit}
A close look at our arguments in the above proof, right after {\rm\Cref{DPP:Limit:Aux2}}, brings to light how to obtain {\rm\Cref{Theorem:DPP:Limit}} in the case one introduces $\ell_{\eps,t,\nu }$ in the definition of equilibria. Indeed, we need to control $\ell_{0,t,\nu }$, the limit $\eps\longrightarrow 0$ of $\ell_{\eps,t,\nu }$. In the case of equilibria, no extra condition was necessary as $\ell_\eps$ is uniform in $(t,\nu )$. However, when this is not the case one could add, for instance, the condition that for $\Pc(\xb)$--$\qe x \in \Xc$  
\[
\inf_{(t,\nu)\in [0,T]\times \Ac(t,x)} \ell_{0,t,\nu}>0.
\]
\end{remark}

\begin{remark}[Reduction in the exponential case]\label{Remark:DPP:sanotycheck}
As a sanity check at this point, we can see what {\rm\Cref{Theorem:DPP:Limit}} yields in the case of exponential discounting. Let, for any $(t,s,x,a)\in[0,T]^2\times \Xc\times A$
\begin{align*}
&f(s,t,x,a)=\mathrm{e}^{-\theta (t-s)}\tilde f(t,x,a),\; \xi(s,x)=\mathrm{e}^{-\theta (T-s)}\tilde \xi(x),\\
&J(t,x,\nu)=\E^{\overline \P^\nu} \bigg[\int_t^T e^{-\theta (r-t)}\tilde f(r,X,\nu_r)\d r+e^{-\theta (T-t)}\tilde \xi(X_{\cdot \wedge T})\bigg].
\end{align*}
Notice that
\begin{align*}
\int_t^\tau \bigg(\partial_s \xi (r,X_{\cdot \wedge T}) +\int_r^T  \partial_s f_u(r,X, \nu_u^\star) \d u \bigg) \d r=&\ \big(\mathrm{e}^{-\theta(T-\tau)}-  \mathrm{e}^{\theta(T-t)} \big) \tilde \xi (X_{\cdot \wedge T})+\int_t^\tau  \big(1-\mathrm{e}^{-\theta (r-t)}\big) \tilde f(r,X,\nu_r^\star ) \d r\\
&+\int_\tau^T \big (\mathrm{e}^{-\theta(r-\tau)}-\mathrm{e}^{-\theta(r-t)}\big)\tilde f(r,X,\nu_r^\star)\d r.
\end{align*}
Now, replacing on the right side of \eqref{Eq:DPP:limit} and cancelling terms we obtain that for $\Pc(\xb)$--$\qe x \in \Xc$
\begin{align*}
v(\sigma,x)=\sup_{\nu \in \Ac(\sigma,x)} \E^{\overline \P^\nu}\bigg[ \int_\sigma^\tau e^{-\theta(r-\sigma)} \tilde f(r,X,\nu_r)\d r+ v(\tau,X_{\cdot \wedge T}) \bigg],
\end{align*}
which is the classic dynamic programming principle, see \emph{\cite[Theorem 3.5]{karoui2015capacities2}}.
\end{remark}

Finally, we mention that representations in the spirit of \eqref{Eq:DPP:limit} have been obtained, see \cite[Proposition 3.2]{ekeland2008investment}. Nevertheless, this is an \emph{a posteriori} result, which follows from a direct application of Feynman--Kac's formula.

\section{Analysis of the BSDE system}\label{Section:Analysis}

We begin this section introducing the spaces necessary to carry out our analysis of \eqref{P}.
\subsection{Functional spaces and norms}
Let $(\Pc(t,x))_{(t,x)\in [0,T]\times \Xc}$ be given family of sets of probability measures on $(\Omega,\Fc)$ solutions to the corresponding martingale problems with initial condition $(t,x)\in [0,T]\times \Omega$.  Fix $(t,x)\in [0,T]\times \Xc$ and let $\Gc$ be an arbitrary $\sigma$-algebra on $\Omega$, $\G:=(\Gc_r)_{s\leq r\leq T}$ be an arbitrary filtration on $\Omega$, $\text{X}$ be an arbitrary $\G$-adapted process, $\P$ an arbitrary element in $\Pc(t,x)$. For any $p,q \geq 1$ we introduce the space

\begin{enumerate}[label=$\bullet$, ref=.$(\roman*)$,wide, labelindent=0pt]
\item $\Lc^p_{t,x}(\Gc)$ $($resp. $\Lc^p_{t,x}(\Gc,\P))$ of $\Gc$-measurable $\R$-valued random variables $\xi$ with 
\[ \| \xi\|_{\Lc^p_{t,x}}^p:=\sup_{\P\in\Pc({t,x})}\E^\P[ |\xi|^p]<\infty,\; \bigg(\text{resp. }  \| \xi\|_{\Lc^p_{t,x}(\P)}^p:= \E^\P[ |\xi|^p]<\infty \bigg). \]
\item  $\S^p_{t,x}(\G)$ $($resp. $\S^p_{t,x}(\G,\P))$ of $Y\in \Pc_{\text{prog}}(\R,\G)$, with $\Pc({t,x})\text{\rm--}\qs$ $($resp. $\P\text{\rm--}\as)$ c\`adl\`ag paths on $[t,T]$, with 
\[ \|Y\|_{\S^p_{t,x}}^p:=\sup_{\P\in\Pc({t,x})} \E^\P\bigg[ \sup_{r\in[t,T]} |Y_r|^p\bigg]<\infty, \; \bigg(\text{resp. } \|Y\|_{\S^p_{t,x}(\P)}^p:=\E^\P\bigg[ \sup_{r\in[t,T]} |Y_r|^p\bigg]<\infty \bigg) .\]

\item  $\L^{q,p}_{t,x}(\G)$ $($resp. $\L^{q,p}_{t,x}(\G,\P))$ of $Y\in \Pc_{\text{prog}}(\R,\G)$, with 
\[ \|Y\|_{\L^{q,p}_{t,x}}^p:=\sup_{\P\in\Pc({t,x})} \E^\P\bigg[\bigg( \int_0^T |Y_r|^q\d r\bigg)^{\frac{p}q} \bigg]<\infty, \; \bigg(\text{resp. } \|Y\|_{\L^{q,p}_{t,x}(\P)}^p:=\E^\P\bigg[ \bigg( \int_0^T |Y_r|^q\d r \bigg)^{\frac{p}q}\bigg]<\infty \bigg) .\]

\item $\H^{p}_{t,x}(\G)$ $($resp. $\H^p_{t,x}(\G,\P))$ of $Z \in  \Pc_{\text{\rm pred}}(\R^d,\G)$, which are defined $\sigmah_t^2 \d t\text{\rm--}\ae$, with 
\[  \|Z\|_{\H^{p}_{t,x}}^p :=\sup_{\P\in\Pc({t,x})} \E^\P\bigg[\bigg( \int_0^T |\sigmah_r Z_r|^2\d r\bigg)^{\frac{p}2} \bigg] <\infty,\; \bigg(\text{resp. }  \|Z\|_{\H^{p}_{t,x}(\P)}^p := \E^\P\bigg[\bigg( \int_0^T |\sigmah_r Z_r|^2\d r\bigg)^{\frac{p}2} \bigg] <\infty\bigg). \]

\item $\I^p_{t,x}(\G,\P)$ of $K \in \Pc_{\text{\rm pred}}(\R,\G)$, with $\P-\as$  c\`adl\`ag, non-decreasing paths with $K_t=0$, $\P$--$\as$, and such that
\[\|K\|_{\I^p_{t,x}(\P)}^p:=\E^\P\big[  |K_T|^p\big]<\infty.\]

We will say a family $(K^\P)_{\P\in \Pc({t,x})}$ belongs to $\I^p_{t,x}((\G_\P)_{\P\in\Pc({t,x})})$, if for any $\P\in \Pc({t,x})$, $K^\P\in \I^p_{t,x}(\G_\P,\P)$, and
 \[
 \|K\|_{\I^{p}_{t,x}}:= \sup_{\P\in\Pc({t,x})}  \|K\|_{\I^p_{t,x}(\P)}^p<\infty.
 \]
 
\item $\M^p_{t,x}(\G,\P)$ of martingales $M\in \Pc_{\text{\rm opt}}(\R,\G)$ which are $\P$-orthogonal to $X$ (that is the product $XM$ is a $(\G,\P)$-martingale), with $\P$--$\as$ c\`adl\`ag paths, $M_0=0$ and
\[\|M\|^p_{\M^p_{t,x}(\P)}:=\E^\P\Big[ [ M]^{\frac{p}{2}}_T\Big]<\infty.\]

%
Due to the time-inconsistent nature of the problem, for a metric space $E$ we let $\Pc^2_{\text{meas}}(E,\Gc)$ be the space of two parameter processes $(U_\uptau)_{\uptau \in [0,T]^2 }$ $:([0,T]^2\times \Omega, \Bc([0,T]^2)\otimes \Gc)  \longrightarrow (\Bc(E),E)$ measurable. 
\medskip

\item $\Lc^{p,2}_{t,x}(\Gc)$ $($resp. $\Lc^{p,2}_{t,x}(\Gc,\P))$ denotes the space of collections $(\xi(s))_{s\in [0,T]}$ of $\Gc$-measurable $\R$-valued random variables such that the mapping $([0,T]\times \Omega, \Bc([0,T])\otimes \Fc^X_T)\longrightarrow (\Lc^p_{t,x}(\Gc),\| \cdot \|_{\Lc^p_{t,x}})$ $($resp. $(\Lc^{p,2}_{t,x}(\Gc,\P)),\| \cdot \|_{\Lc^p_{t,x}(\P)})):s\longmapsto \xi(s)$ is continuous and
\[   \| \xi\|_{\Lc^{p,2}_{t,x}}^p:= \sup_{s\in[0,T]} \|\xi(s)\|^p_{\Lc^{p}_{t,x}}<\infty     ,\, \bigg(\text{resp. } \| \xi\|_{\Lc^{p,2}_{t,x}(\P)}^p:= \sup_{s\in[0,T]} \|\xi(s)\|^p_{\Lc^{p}_{t,x}(\P)}<\infty\bigg). \]

Finally, given a generic integrability space $(\I^p,\|\cdot\|_{\I})$ we introduce the space

\item $\I^{p,2}$ of $(U_\uptau)_{\uptau \in [0,T]^2 }\in \Pc^2_{\text{meas}}(\R,\Gc_T)$ such that the mapping $([0,T],\Bc([0,T])) \longrightarrow (\I^{p},\|\cdot \|_{ \I^{p}}): s \longmapsto U^s $ and 
\begin{align*}
\| U\|_{\I^{p,2}}^p:= \sup_{s\in[0,T]} \| U^s\|_{\I^{p}}^p <\infty.
\end{align*}

\end{enumerate}

\begin{remark}
To ease the notation, when $p=q$ we will write $\L^{p}_{t,x}(\G)$ $\big($resp. $\L^{p,2}_{t,x}(\G)\big)$ for $\L^{q,p}_{t,x}(\G)$ $\big($resp. $\L^{q,p,2}_{t,x}(\G)\big)$. With this convention, $\L^{2}_{t,x}(\G)$ $\big($resp. $\L^{2,2}_{t,x}(\G)\big)$ will always mean $\L^{2,2}_{t,x}(\G)$ $\big($resp. $\L^{2,2,2}_{t,x}(\G)\big)$. The spaces are $\L^{q,p,2}_{t,x}(\G)$ and $\H^{p,2}_{t,x}(\G,X)$ are Hilbert spaces. For $U\in \S^{p,2}_{t,x}(\G)$ we highlight the diagonal process $(U_t^t)_{t\in [0,T]}$ is well defined. Indeed, the path continuity of $U^s$ for all $s\in [0,T]$ together with the $($uniform$)$ continuity of $s\longmapsto \|U^s\|_{\S^2}$ allows us to define a $\Bc[0,T]\otimes \Fc$-measurable version. Finally we will suppress the dependence on $(0,\xb)$ and write $\S_{\xb}(\G)$ for $\S_{0,\xb}(\G)$ and similarly for the other spaces.
\end{remark}

\subsection{The BSDE system}
We now begin our study of the system
\begin{align}\tag{H}
\begin{split}
Y_t=&\ \xi(T,X_{\cdot\wedge T})+\int_t^T   F_r(X,Z_r,\widehat \sigma_r^2,\partial {Y_r^r})\d r- \int_t^T  Z_r \cdot  \d X_r+ K_T^\P-K_t^\P,\;  0\leq t \leq T,\; \Pc(\xb)\text{--}\qs,\\
{\partial Y_t^s}(\omega ):=&\ \E^{\overline \P^{\nu^\star}_{t,x}}\bigg[\partial_s \xi(s,X_{\cdot\wedge T})+\int_t^T\partial_s f_r(s,X, \Vc^\star(r,X,Z_r) )  \d r\bigg],\; (s,t)\in[0,T]^2,\; \omega\in\Omega. 
\end{split}
\end{align}

As a motivation of the notion of solution to \eqref{HJB}, let us note that the first equation is a 2BSDE under the set $\Pc(\xb)$, \emph{i.e.} the dynamics holds $\Pc(\xb)$--$\qs$ A closer examination of \Cref{Def:equilibrium} reveals that, unlike in the classical stochastic control scenario, one needs to be able to make sense of a solution under any $\Pc(s,x)$ for $s\in [0,T]$ and $x$ outside a $\Pc(\xb)$-polar set. Fortunately, the results in \citeauthor*{possamai2015stochastic} \cite{possamai2015stochastic} allow us to verify that constructing the initial solution suffices, see \Cref{Lemma:shifted2bsde}.

\begin{definition}\label{Def:2bsde}
Let $(s,x)\in [0,T]\times \Xc$, $\partial Y_r^r$ be a given process and consider the equation
\begin{align}\label{Eq:dynamic2bsde}
Y_t^{s,x}=\xi(T,X_{\cdot\wedge T})+\int_t^T   F_r(X,Z_r^{s,x},\widehat \sigma_r^2,\partial {Y_r^r})\d r- \int_t^T  Z_r^{s,x} \cdot  \d X_r\ + K_T^{s,x,\P}-K_t^{s,x,\P},\;  s\leq t \leq T.
\end{align}
We say $(Y^{s,x},Z^{s,x},(K^{s,x,\P})_{\P\in \Pc(s,x)})$ is a solution to {\rm 2BSDE} \eqref{Eq:dynamic2bsde} under $\Pc(s,x)$ if for some $p>1$,
\begin{enumerate}[label=$(\roman*)$, ref=.$(\roman*)$,wide, labelindent=0pt]
\item {\rm\Cref{Eq:dynamic2bsde}} holds $\Pc(s,x)$--$\qs$
\item $(Y^{s,x},Z^{s,x},(K^{s,x,\P})_{\P\in \Pc(s,x)})\in \S_{s,x}^p(\F^{X,\Pc(s,x)}_+)\times \H_{s,x}^p(\F^{X,\Pc(s,x)}_+)\times (\I_{s,x}^p(\F^{X,\P}_+,\P))_{\P\in \Pc(s,x)}$.
\item The family $(K^{s,x,\P})_{\P\in \Pc(s,x)}$ satisfies the minimality condition
\[0=\operatorname*{ess\, inf^\P}_{\P'\in \Pc_{s,x}(t,\P,\F_+^X)} \E^{\P'}\big[K_T^{s,x,\P'}-K^{s,x,\P'}_t \big|\Fc_{t+}^{X,\P'}\big],\; s\leq t\leq T,\; \P\text{\rm--}\as,\; \forall \P\in \Pc(s,x).\]
\end{enumerate}
\end{definition}

Consistent with \Cref{Def:solution2bsde0}, we set $(Y,Z,(K^\P)_{\P\in\Pc(\xb)})=(Y^{0,\xb},Z^{0,\xb},(K^{0,\xb,\P})_{\P\in\Pc(\xb)})$. We use the rest of this section to prove \Cref{Thm:PDEsystem:markovian}, justifying that in the setting of this paper our approach encompasses that of \cite{bjork2014theory}.

\begin{proof}[Proof of \emph{Theorem \ref{Thm:PDEsystem:markovian}}]
Let $\P\in \Pc(\xb)$ and consider $(\Omega,\F^{X,\Pc(\xb)}_+,\P)$. We first verify that $( Y, Z, K)$ satisfies first equation in System \eqref{HJB}. A direct application of It\^o's formula to $ Y_t=v(t,X_t)$ with $X$ given by the SDE \eqref{Eq:driftlessSDE} yields that $\P$--$\as$
\begin{align*}
 Y_t&= Y_T -\int_t^T \Big(\partial_t v(r,X_r) +\frac{1}{2} \Tr[\d \langle X\rangle_r \partial_{xx}v(r,X_r)] \Big) \d r - \int_t^T  \partial_x v(r,X_r) \cdot  \d X_r\\
&=  Y_T+ \int_t^T  \sup_{\Sigma \in \Sigma_r(X_r)}\Big\{ F_r(X_r,  Z_r,\Sigma,\partial  \Yc_r^r)+\frac{1}{2} \Tr[ \Sigma \,  \Gamma_r] \Big\}-  \frac{1}{2} \Tr\big[ \d \langle X\rangle_r  \Gamma_r  \big] \d r -\int_t^T  Z_r  \cdot \d X_r,
\end{align*}
where we used \eqref{Eq:HJB-BKM} and the definition of $H$ in terms of $F$ as in \eqref{Hamiltonia:General}. Next, by definition of $\sigmah^2_t$ and with $K_t$ as in the statement we obtain
\begin{align*}
& Y_t=  Y_T+ \int_t^T F_r(X_r,  Z_r,\widehat \sigma_r^2 ,\partial \Yc_r^r)-\int_t^T Z_r  \cdot \d X_r+ K_T- K_t, \; \P\text{\rm--}\as
\end{align*}

Next, we verify the integrability conditions in Definitions \ref{Def:sol:systemH}\ref{Def:sol:systemH:i} and \ref{Def:sol:systemH}\ref{Def:sol:systemH:ii}. As $\sigma$ is bounded, it follows that for any $\P\in \Pc(\xb)$, $X_t$ has exponential moments of any order which are bounded on $[0,T]$, \emph{i.e.} $\exists C$, $\sup_{t\in[0,T]} \E^\P[\exp(c |X_t|_1)]\leq C<\infty, \forall \P\in \Pc(\xb), \forall c>0$ ,where $C$ depends on $T$ and the bound on $\sigma$. 

\medskip
The exponential grown assumption on $v$ and de la Vallée-Poussin's theorem yield $ Y\in\S^p_\xb(\F^{X,\Pc(\xb)}_+,\Pc(\xb))$ for $p>1$. Similarly, we obtain $ Z\in \H^p_\xb(\F^{X,\Pc(\xb)}_+,\Pc(\xb),X)$ and $\partial  Y \in \S^{p,2}_\xb(\F^{X,\Pc(\xb)}_+,\Pc(\xb))$. To derive the integrability of $ K$, let $\P \in \Pc(\xb)$ and note that
\begin{align*}
\E^\P[K_T^p]\leq & C_p\bigg( \| Y \|^p_{\S^p_\xb}+\sup_{\P\in \Pc(\xb)}\E^\P\bigg[\bigg(\int_0^{T} |F_r(X_r, Z_r,\widehat \sigma_r^2 ,0)| \d r\bigg)^p \bigg]+\| \partial  Y \|^p_{\S^{p,2}_\xb}+ \| Z\|_{\H_\xb^p}^p \bigg)<\infty,
\end{align*}
where the inequality follows from the fact $(t,x,z,a)\longmapsto F(t,x,z,a,0)$ is Lipschitz in $z$ which follows from the exponential growth assumption on $f^0$ and the boundedness of the coefficients $b$ and $\sigma$. The constant $C_p$ depends on the Lipschitz constant and the value of $p$ as in \citeauthor*{bouchard2018unified} \cite[Lemma 2.1]{bouchard2018unified}. As the term on the right does not depend on $\P$, we conclude $ K \in  \I^{p}_\xb((\F^{X,\P}_+)_{\P\in \Pc(\xb)})$. The previous estimate shows as a by-product that the 2BSDE in \eqref{HJB} is well-posed, see \cite[Theorem 4.1]{possamai2015stochastic}. Therefore, provided $K$ satisfies the minimality condition by \cite[Theorem 4.2]{possamai2015stochastic}, we conclude \Cref{Def:sol:systemH}\ref{Def:sol:systemH:i}, \emph{i.e.} $(Y,Z,K)$ is the solution to the 2BSDE in \eqref{HJB}.  \medskip

We now argue that $K$ satisfies \eqref{Min:Condition:NoA}. Following \cite[Theorem 5.2]{soner2012wellposedness}, we can exploit the fact the $\sigma$ is bounded and the continuity in time of $X$, $Z$ and $\Gamma$ for fixed $x\in \Xc$, to show that for $\eps>0$, $(t,\nu)\in [0,T]\times \Ac(\xb,\P)$ and $\tau^{\eps,t} :=T\wedge \inf \{r>t : K_r \geq K_t +\eps\},$ there exists $\P^{\nu^\eps}\in \Pc_{\xb}(t,\P,\F_+^X)$ such that $k_t \leq \eps, \; \d t\otimes \d \P^{\nu^\eps} \text{ on } [\tau^{\eps,t},T]\times \Omega $. From this the minimality condition follows. Moreover, by assumption, we know there exists $\P^{\nu^\star}\in \Pc^0(0,\xb,\nu^\star)$ where $\nu^\star$ maximises the Hamiltonian, \emph{i.e.} $\sigmah^2_r=(\sigma\sigma^\t)_t(X_t,\nu^\star_t), \d t\otimes \d \P^\star\text{\rm--}\ae$ on $[0,T]\times \Omega$, and $k_t=0$, $\P^{\nu^\star}\text{\rm--}\as$ for all $t\in[0,T]$. Thus the minimality condition is attained under $\P^{\nu^\star}$, \emph{i.e.} \ref{Def:sol:systemH}\ref{Def:sol:systemH:iii} holds. Moreover, note that this implies
\begin{align}\label{Eq:equalitymaximisers}
\overline \Vc^\star(t,X_t,Z_t,\Gamma_t)= \Vc^\star(t,X_t,Z_t),\; \P^{\nu^\star}\text{\rm--}\as
\end{align}

We are left to argue $\partial  Y$ satisfies the second equation in \eqref{HJB}. Given the regularity of $s\longmapsto J(s,t,x)$, we can differentiate the second equation in \eqref{Eq:HJB-BKM}. Using this, for $s\in [0,T]$ fixed and $\omega=(x,\textbf{w},q) \in \Omega$ we can apply It\^o's formula to $\partial Y^s_t=\partial_s J(s,t,X_t)$ under $\overline \P^{\nu^\star}_{t,x}$. This yields, 
\[ \partial Y^s_t(\omega)=\E^{\overline \P^{\nu^\star}_{t,x}}\bigg[ \partial_s \xi(s,X_T)+\int_t^T \partial_s f_r(s,X_r, \Vc^\star (r,X,Z_r))\d r\bigg], \]
where the stochastic integral term vanished in light of the growth assumption on $\partial_x J(s,t,x)$ and we used \eqref{Eq:equalitymaximisers}.
\end{proof}

\subsection{Necessity of (\textcolor{red}{H}) for equilibria}\label{Section:necessity}

We recall that throughout this section, we let the Assumptions \ref{AssumptionA} and \ref{AssumptionB} hold. To begin with, from the definition of the set $\Ac(t,x)$, we can restate the result of our dynamic programming principle \Cref{Theorem:DPP:Limit}, by decomposing a control $\nu\in \Ac(t,x)$ into a pair $(\P,\nu)\in\Pc(t,x)\times \Ac(t,x,\P)$, where $\P$ is the unique weak solution to \eqref{Eq:driftlessSDE} and $\nu$. We remark that for any given $\P\in \Pc(t,x)$, there could be in general several admissible controls $\nu$. With this we state \Cref{Theorem:DPP:Limit} as, for $\nu^\star \in \Ec(\xb)$, $\sigma, \tau\in \Tc_{t,T}$, $\sigma\leq \tau$ and $\Pc(\xb)$--$\qe x\in \Xc$
\begin{align}\label{Eq:V:sup:sup}
v(\sigma,x)=\sup_{\P \in \Pc(\sigma,x)}\sup_{\nu \in \Ac(\sigma,x,\P)} \E^{\overline \P^\nu} \bigg[ v(\tau, X)+ \int_{\sigma}^\tau  f_r(r,X,\nu_r)- \E^{\overline \P^{\star,\nu^\star}_{r,\cdot} }\bigg[\partial_s \xi(r,X_{\cdot\wedge T}) + \int_r^T  \partial_s f_u(r,X,\nu^\star_u)\d u\bigg]\d r\bigg].
\end{align}

The goal of this section is to show that given $\nu^\star$, $\P^\star$, unique solution to the martingale problem asociated with $\nu^\star$, and $v(t,x)$, one can construct a solution to \eqref{HJB}. To do so, we recall that given a family of BSDEs $(\Yc^\P)_{\P\in\Pc(t,\omega)}$ indexed by $\Pc(t,\omega)\subseteq \Prob(\Omega)$ with $(t,\omega)\in [0,T]\times \Omega$, a 2BSDE is the supremum over $\Pc(t,\omega)$ of the $\P$-expectation of the afore mentioned family, see \cite{soner2012wellposedness}, \cite{possamai2015stochastic}. This together with equation \eqref{Eq:V:sup:sup} reveals the road map we should take.\medskip

Let us begin by fixing an equilibrium $\nu^\star\in \Ec(\xb)$. For $(s,t,\omega,\P)\in [0,T]\times [0,T]\times \Omega\times \Pc(t,x)$ we consider the $\F^X$-adapted processes
\begin{align}\label{BSDE:supactions}
\begin{split}
\widetilde \Yc_t^\P(\omega)&:=\sup_{\nu \in \Ac(t,x,\P)} \E^{\overline \P^\nu}\bigg[ \xi(T,X)+\int_{t}^T \big( f_r(r,X,\nu_r)- \partial Y_r^r\big)\d r\bigg],\\
 \partial Y_t^s( \omega )&:=\E^{{\overline \P}^{\nu^\star}_{t,x} }\bigg[\partial_s \xi(s,X_{\cdot\wedge T}) +\int_t^T  \partial_s f_r(s,X,\nu^\star_r)\d r\bigg].
\end{split}
\end{align}
and on $(\Omega, \F^{X,\P}_+, \P)$ the BSDE
\begin{align}\label{BSDE:Nusup}
\Yc_t^\P=\xi(T,X_{\cdot \wedge T})+\int_t^T  F_r(X,\Zc_r^\P,\widehat \sigma _r^2, {\partial Y_r^r})\d r-\bigg(\int_t^T  \Zc_r^\P \cdot \d X_r\bigg)^\P,\; 0\leq s\leq t\leq T.
\end{align}
Note we specify the stochastic integral w.r.t $X$ is under the probability $\P$. Our first step is to relate $\widetilde \Yc^\P$ with the solution to the BSDE \eqref{BSDE:Nusup}. Namely, \Cref{Lemma:BSDE:Nusup} says that $\widetilde \Yc^\P$ corresponds to the first component of the solution to \eqref{BSDE:Nusup}.

\begin{lemma}\label{Lemma:BSDE:Nusup}
Let {\rm\Cref{AssumptionB}} hold, $(t,\omega , \P)\in [0,T]\times \Omega\times \Pc(t,x)$ and $(\Yc^{\P},\Zc^\P)$ be the solution to the {\rm BSDE} \eqref{BSDE:Nusup}, as in {\rm\citeauthor*{papapantoleon2018existence} \cite[Definition 3.2]{papapantoleon2018existence}}, and $\tilde \nu^{\star}_t:=\Vc^\star(t,X,\Zc_t^{\P})$. Then
\begin{align}\label{Eq:Lemma:BSDE:Nusup}
\widetilde \Yc_t^\P(\omega )=\E^{\P^{\tilde \nu^\star}}[\Yc^\P_t].
\end{align}
\end{lemma}
\begin{proof}
Let us consider on $\big(\Omega, \F^{X,\P}_+, \P\big)$, for $t\leq u\leq T$ and $\nu \in \Ac(t,x,\P)$ the BSDE
\begin{align}\label{Necessity:Aux0}
\Yc_u^{\P,\nu}= \xi(T,X_{\cdot \wedge T})+\int_u^T \big( h_r(r,X,\Zc_r^{\P,\nu},\nu_r)-\partial Y_r^r \big)\d r-\bigg(\int_u^T \Zc_r^{\P,\nu}\cdot \d X_r\bigg)^\P,\; \P\text{\rm--}\as
\end{align}
Under Assumptions \ref{AssumptionB}\ref{AssumptionB:i} and \ref{AssumptionB}\ref{AssumptionB:iii}, we know that $z\longmapsto h_t(t,x,z,a)$ is Lipschitz-continuous, uniformly in $(t,x,a)$, that there exists $p>1$ such that $([0,T]\times \Omega,\F^X)\ni (t,\omega)\longmapsto h_t(t,x,0,0)\in\H^{p,2}_{s,\omega}(\R^m,\F^{X,\P}_+,\P)$ is well defined, and that $\partial Y\in\H^p_{s,x}(\R,\F_+^{X,\P},\P)$. Moreover, as for any $\nu$, and $\tilde\nu$ in $\Ac(s,x,\P),$ $(\sigma\sigma^\t)_t(X,\nu_t)=\sigmah^2_t=(\sigma\sigma^\t)_t(X,\tilde \nu_t), \d t\otimes \d \P$--a.e., $\P$ is the unique solution to an uncontrolled martingale problem where $(\sigma \sigma^\t)_t(X):=(\sigma\sigma^\t)_t(X,\nu_t)$. Consequently, the martingale representation property holds for any local martingale in $(\Omega, \F^\P, \P)$ relative to $X$, see \citeauthor*{jacod2003limit} \cite[Theorem 4.29]{jacod2003limit}. Therefore, conditions (H1)--(H6) in \cite[Theorem 3.23]{papapantoleon2018existence} hold and the above BSDE is well defined. Its solution consists of a tuple $(\Yc^{\P,\nu},\Zc^{\P,\nu})\in \D^p_{s,x}(\F^{X,\P}_+,\P) \times \H^p_{s,x}(\R^m,\F^{X,\P}_+,\P,X)$ and for every $\nu \in \Ac(t,x,\P)$ and $\overline \P^\nu$ is as in \Cref{Remark:com} we have
\begin{align*}
\E^{\overline \P^\nu}[\Yc_t^{\P,\nu}]=\E^{\overline \P^\nu}\bigg[ \xi(T,X)+\int_{t}^T \big( f_r(r,X,\nu_r)- \partial Y_r^r\big)\d r\bigg],\ t\in[0,T].
\end{align*}

In addition, Assumption \ref{AssumptionB}\ref{AssumptionB:ii} guarantees the solution $(\Yc^\P,\Zc^\P)\in \D^p_{s,\omega}(\F^{X,\P}_+,\P) \times \H^p_{s,\omega}(\R^m,\F^{X,\P}_+,\P)$ to BSDE \eqref{BSDE:Nusup} is well defined. Furthermore, conditions (Comp1)--(Comp3) in \cite[Theorem 3.25]{papapantoleon2018existence} are fulfilled, ensuring a comparison theorem holds. Indeed, as $X$ is continuous (Comp1) is immediate, while (Comp2) and (Comp3) correspond in our setting to \ref{AssumptionB}\ref{AssumptionB:i} and \ref{AssumptionB}\ref{AssumptionB:iii}, respectively. By definition, $\tilde \nu^\star_t$ satisfies $(\sigma \sigma^\t)_t(X, \tilde \nu^\star_t)=\sigmah_t^2, \d t\otimes \d \P$--a.e. on $[t,T]\times \Xc$, and so, $\tilde \nu^\star \in \Ac(t,x,\P)$. By comparison, we obtain
\begin{align*}
\widetilde \Yc^\P_t(x)= \sup_{\nu\in \Ac(t,x\P)}\E^{\overline \P^\nu}[ \Yc^{\P,\nu}_t ] \leq \E^{\overline \P^{\tilde \nu^\star}} [\Yc^{\P}_t ].
\end{align*}\end{proof}

\begin{remark}
\begin{enumerate}[label=$(\roman*)$, ref=.$(\roman*)$,wide, labelindent=0pt]
\item In the literature on {\rm BSDEs} one might find the additional term $ \int_t^T \d \Nc_r^\P$ in \eqref{BSDE:Nusup}, where $\Nc^\P$ is a $\P$-martingale $\P$-orthogonal to $X$. Yet, as noticed in the proof once $(\sigma\sigma^\t)_t(X)$ is fixed, uniqueness of the associated martingale problem guarantees the representation property relative to $X$.

\item We remark that an alternative constructive approach to relate $\widetilde \Yc$ to a {\rm BSDE} is to consider for any $\nu \in \Ac(s,x,\P)$ and $t\geq s$ the process
\begin{align*}
N_t^{\P,\nu}:=J(t,t, X,\nu^{\star}(\P)) +\int_0^t \big( f_r(r,X,\nu_r)-\partial Y_r^r \big)\d r.
\end{align*}
However, as the careful reader might have noticed, this requires that for a given $\P\in \Pc(s,x)$ we introduce $\nu^{\star}(\P)$ the action process attaining the {\rm sup} in \eqref{BSDE:supactions}. However, the existence of an action with such property is not necessarily guaranteed for all $\P\in \Pc(s,x)$, at least without further assumptions, which we do not want to impose here.
\end{enumerate}
\end{remark}

\begin{remark}
At this point we are halfway from our goal in this section. For $(t,\omega)\in [0,T]\times \Omega$, the previous lemma defines a family $(\Yc^{\P})_{\P\in \Pc(t,\omega)}$ of $\F^{X,\P}_+$-adapted processes. Recalling our discussion at the beginning of this section, all we are left to do is to take $\sup$ over $(\Pc(t,x)
)_{(t,x)\in [0,T]\times \Xc}$, \emph{i.e.} putting together \eqref{Eq:Lemma:BSDE:Nusup} and \eqref{Eq:V:sup:sup}, we now know
\begin{align}\label{Eq:v:Psup}
v(t,x)=\sup_{\P\in \Pc(t,x)} \E^{\overline \P^{\tilde \nu^\star}}\big[ \Yc^\P_t\big].
\end{align}
\end{remark}

In light of the previous remark and the characterisation in \cite{possamai2015stochastic}, we consider the following 2BSDE
\begin{align}\label{2BSDE:Necessity}
Y_t=\xi(T,X_{\cdot \wedge T})+\int_t^T  F_r(X,Z_r,\widehat \sigma_r^2, \partial Y_r^r)\d r-\int_t^T  Z_r \cdot \d X_r+K_T^\P-K_t^\P,\; 0\leq t \leq T ,\ \Pc(\xb)\text{\rm--}\qs
\end{align}

With this we are ready to prove the necessity of System \eqref{HJB}.

\begin{proof}[Proof of {\rm \Cref{Theorem:Necessity}}]
We begin by verifying the integrability of $\partial Y$, defined as in \eqref{BSDE:supactions}. From \Cref{AssumptionB}\ref{AssumptionB:iii} we have that for any $s\in [0,T]$
\[
\|\partial Y^s\|_{\S_{\xb}^p} \leq  \sup_{\P\in \Pc(\xb)} \E^\P\bigg[ \big | \partial_s \xi(s,X_{\cdot \wedge T}) \big |^p+ \int_0^T \big |  \partial_s f_r (s,X,\nu^\star_r)\big|^p  \d r\bigg] < \infty.
\]
Therefore, as \Cref{AssumptionA}\ref{AssumptionA:DPP1} guarantees the continuity of the map $s\longmapsto \|\partial Y^s\|_{\S_{0,\xb}^p(\F^X,\Pc(\xb))}$ the result follows.  \medskip

Let us construct such a solution from $\nu^\star \in \Ec(\xb)$. Under \Cref{AssumptionB}\ref{AssumptionB:i}, it follows from \eqref{Eq:v:Psup} and \cite[Lemma 3.2]{possamai2015stochastic} that $v$ is l\`adl\`ag outside a $\Pc(\xb)$-polar set. Therefore the process $v^+$ given by
\[
v^+_t(x) :=\lim_{r\in \Q\cap (t,T], r\downarrow t} v(t,x),
\]
is well defined in the $\Pc(\xb)$--$\qs$ sense. Clearly $v^+$ is c\`adl\`ag, $\F^{X,\Pc(\xb)}_+$-adapted, and in light of \cite[Lemmata 2.2 and 3.6]{possamai2015stochastic}, which hold under \Cref{AssumptionB}, we have that for any $\P \in \Pc(\xb)$, there exist $(\Zc^\P,\Kc^\P)\in \H^p_{0,\xb}(\F^{X,\P}_+,\P)\times  \I^p_{0,\xb}(\F^{X,\P}_+,\P)$ such that 
\begin{align*}
v^+_t=\xi(T,X_{\cdot \wedge T})+\int_t^T  F_r(X,\Zc_r^\P,\widehat \sigma _r^2, {\partial Y_r^r})\d r-\bigg(\int_t^T  \Zc_r^\P \cdot \d X_r\bigg)^\P+\Kc^\P_T-\Kc^\P_t,\; 0\leq t\leq T,\; \P\text{\rm--}\as
\end{align*}
Moreover, the process $ Z_t:=(\sigmah^{2}_t)^{\oplus} \frac{\d [ v^+,X]_t }{\d t},$ where $(\sigmah^{2}_t)^{\oplus}$ denotes the Moore--Penrose pseudo-inverse of $\sigmah^{2}_t$, aggregates the family $(Z^\P)_{\P\in \Pc(\xb)}$. The proof that $(K^\P)_{\P\in \Pc(\xb)}$ satisfies the minimality condition \eqref{Min:Condition:NoA} is argued as in \cite[Section 4.4]{possamai2015stochastic}. The integrability follows from \Cref{AssumptionB}\ref{AssumptionB:iii}.\medskip

Arguing as in \cite[Lemma 3.5]{possamai2015stochastic}, we may obtain that for any $\Pc(\xb)$--$\qe x\in \Xc$ and $\P\in \Pc(t,x)$
\[ v_t^+=\es_{\tilde \P \in \Pc(t,\P,\F_+^X)} \Yc_t^{\tilde \P}
\]
with $\Yc^{\tilde \P}$ as in \eqref{BSDE:Nusup}. Consequently
\[
v(t,x)=\sup_{\P\in \Pc(t,x)} \E^{\P}[v_t^+].
\]
Moreover, as for any $t\in [0,T]$ and $\Pc(\xb)$--$\qe x\in \Xc$, $\overline \P^{\nu^\star}_{t,x}$ attains equality in \eqref{Eq:v:Psup}, see \Cref{Theorem:DPP:Limit}, we deduce, in light of \Cref{Remark:com}, $\P^{\nu^\star}$ attains the minimality condition. This is, under $\P^{\nu^\star}$, the process $K^{\P^{\nu^\star}}$ equals $0$. With this, we obtain $(v^+,Z,(K)^\P_{\P\in \Pc(\xb)})$ and $\partial Y$ are a solution to \eqref{HJB}. Moreover, \Cref{Lemma:BSDE:Nusup} implies 
\[
\nu^\star_t\in \argmax_{a\in A_t(x,\hat\sigma_t^2(x))} h_t(t,X,Z_t,a),\; \d t \otimes \d \P^{\nu^\star}\text{\rm--}{\rm a.e.},\; \text{on}\; [0,T]\times \Xc.\]
 \end{proof}

\subsection{Verification}

This section is devoted to prove the verification \Cref{Verification}. To do so we will need to obtain a rigorous statement of how the processes defined by \eqref{HJB} relate. This is carried in a series of lemmata available in \Cref{Appendix:verification}.

\begin{proof}[Proof of {\rm \Cref{Verification}}]
We will first show that with $\nu^\star$ as in the statement of the theorem $Y_t(x)=J(t,t,x,\nu^{\star})$ for all $t\in [0,T]$ and $\Pc(\xb)$--$\qe x \in \Xc$. To do so, let $(s,x)\in [0,T]\times \Xc$ and note that \Cref{AssumptionB} guarantees that the corresponding 2BSDE under $\Pc(s,x)$ is well-posed. Indeed, it follows from \citeauthor*{soner2011martingale} \cite[Lemma 6.2]{soner2011martingale} that for any $p>p'>\kappa>1$,
\begin{align*}
\sup_{\P \in \Pc(s,x)} \E^{\P}\bigg[\es_{s\leq t \leq T}\bigg( \es_{\P'\in \Pc_{s,x}(t,\P,\F^+)}  \E^{\P'}\bigg[|\xi(T,X_{ \cdot \wedge T })|^\kappa +	\int_s^T|F_r(X,0,\sigmah^2_r,\partial Y_r^r)|^\kappa \d r\bigg| \Fc_t^+\bigg]\bigg)^{\frac{p'}{\kappa}}\bigg] <\infty.
\end{align*}
The well-posedness follows by \cite[Theorem 4.1]{possamai2015stochastic}. Now, in light of \cref{Lemma:shifted2bsde}, for any $s \in [0,T]$
\begin{align}\label{Eq:consistencyshiftedbsde}
\begin{split}
Y_t&=Y_t^{s,x},\; s\leq  t\leq T, \; \Pc(s,x)\text{\rm--}\qs, \text{ for } \Pc(\xb)\text{\rm--}\qe\; x\in \Xc,\\
Z_t&=Z_t^{s,x}, \; \sigmah_t^2 \d t\otimes \d \Pc(s,x)\text{\rm--}\qe \text{ on } [s,T]\times \Xc, \text{ for } \Pc(\xb)\text{\rm--}\qe\; x\in \Xc,\\
K_t^{\P }&=K_t^{s,x,\P_{s,x} },\; s\leq  t\leq T,\;  \P_{s,x}\text{\rm--}\as,\; \text{for}\; \P\text{\rm--}\ae\; x\in \Xc,\; \forall \P\in \Pc(\xb).
\end{split}
\end{align}

We first claim that given a solution to \eqref{HJB} for any $(s,x)\in (0,T]\times \Xc$, $\P^{\nu^\star}_{s,x}$ attains the minimality condition for the 2BSDE under $\Pc(s,x)$, see \Cref{Def:2bsde}. Indeed, by \Cref{Def:sol:systemH}\ref{Def:sol:systemH:iii}
\[\E^{\P^{\nu^\star}_{0,\xb}} \Big[K_T^{\P^{\nu^\star}_{0,\xb}}-K_t^{\P^{\nu^\star}_{0,\xb}} \Big]=0,\;  0\leq t\leq T.\]
As $K^{\P^{\nu^\star}_{0,\xb }}$ is an increasing process, this implies $K^{\P^{\nu^\star}_{0,\xb}}=0$ and therefore $\P^{\nu^\star}_{0,\xb }$ attains the minimality condition for the 2BSDE in \eqref{HJB} under $\Pc(\xb)$. This implies, together with \eqref{Eq:consistencyshiftedbsde}, that for $\Pc(\xb)\text{--}\qe x\in \Xc$ and $s\in [0,T]$
\begin{align*}
\E^{\P^{\nu^\star}_{s,x}}\Big[ \E^{\P^{\nu^\star}_{s,x}}\Big[ K_T^{s,x,\P^{\nu^\star}_{s,x}}- K_s^{s,x,\P^{\nu^\star}_{s,x}} \Big|\Fc_{s+}^{X} \Big]\Big]=\E^{\P^{\nu^\star}_{s,x}}\Big[ K_T^{s,x,\P^{\nu^\star}_{s,x}}- K_s^{s,x,\P^{\nu^\star}_{s,x}} \Big]=0,\; \P^{\nu^\star}_{s,x}\text{--a.s.},
\end{align*}
which proves the claim. Consequently, for $\Pc(\xb)$--$\qe x\in \Xc$ and $s\in [0,T]$
\begin{align*}
Y_t=Y_t^{s,x}=\xi(T,X_{\cdot \wedge
 T})+\int_t^T F_r(X,Z_r^{s,x},\widehat{\sigma}^2_r, \partial Y_r^r) \d r -\int_t^T Z_r^{s,x} \cdot \d X_r,\; s\leq t\leq T,\; \P^{\nu^\star}_{s,x}\text{\rm--}\as
\end{align*}
We note the equation on the right side prescribes a BSDE under $ \P^{\nu^\star}_{s,x}\in \Pc(s,x)$. Moreover, given that for $\Pc(\xb)$--$\qe x\in \Xc$ and $s\in [0,T]$, $\Vc^\star_t(X,Z^{s,x}_t)=\Vc^\star_t(X,Z_t), \d t \otimes \d \P^{\nu^\star}_{s,x}\text{\rm--}\ae$ on $[s,T]\times \Xc$, we obtain
\begin{align*}
Y_t^{s,x}=\xi(T,X_{\cdot \wedge T})+\int_t^T \big(h_r(r,X,Z_r,\nu^{\star}_r)- \partial Y_r^r\big) \d r -\int_t^T Z_r^{s,x} \cdot \d X_r,\; s\leq t\leq T, \;  \P^{\nu^\star}_{s,x}\text{\rm--}\as
\end{align*}
In particular, at $t=s$ we have that for $\Pc(\xb)$--$\qe x\in \Xc$ and $t\in [0,T]$
\begin{align}\label{Eq:veri1}
\E^{\P^{\nu^\star}_{t,x}}\big[Y_t^{t,x}\big]=\E^{\P^{\nu^\star}_{t,x}}\bigg[ \xi(T,X_{\cdot \wedge  T})+\int_t^T\big( h_r(r,X, Z_r,\nu^\star_r)-\partial Y_r^r\big)\d r \bigg].
\end{align}

Now, in light of \Cref{AssumptionC}\ref{AssumptionC:ii}, there exists $(\partial \Yc,\partial \Zc)\in \S^{2,2}_{t,x}(\F^{X,\P^{\nu^\star}_{t,x}}_+)\times \H^{2,2}_{t,x}(\F^{X,\P^{\nu^\star}_{t,x}}_+,X)$ such that for any $s\in [t,T]$
\begin{align*}
\partial \Yc_r^{s}&=\partial_s \xi(s, X_{\cdot \wedge T})+\int_r^T \nabla h_u(s,X,\partial \Zc_u^{s}, \nu_u^{\star}) \d u- \int_r^T\partial \Zc_u^{s} \cdot \d X_u ,\; s\leq r \leq T, \; \P^{\nu^\star}_{t,x}\text{\rm--}\as
\end{align*}
In addition, \Cref{Lemma:Derivative} yields
\begin{align*}
\E^{\P^{\nu^\star}_{t,x}}\bigg[ \int_t^T \partial Y_r^r\d r\bigg]=\E^{\P^{\nu^\star}_{t,x}}\bigg[ \int_t^T \E^{\P^{\nu^\star}_{t,\cdot}}\bigg[ \partial_s \xi(r, X_{\cdot \wedge T})+\int_r^T \partial_s f_u (r,X,\nu^\star_u)\d u\bigg] \d r\bigg]=\E^{\P^{\nu^\star}_{t,x}}\bigg[ \int_t^T \partial \Yc_r^{r}\d r\bigg].
\end{align*}
Therefore, from \Cref{Lemma:Dynamicsdiagonal} and \eqref{Eq:veri1} we have that for $\Pc(\xb)$--$\qe x\in X$ and $t\in [0,T]$
\begin{align}\label{Eq:veri2}
\E^{\P^{\nu^\star}_{t,x}}\big[ Y_t^{t,x}\big]=\E^{\P^{\nu^\star}_{t,x}}\bigg[ \xi(T,X_{\cdot \wedge  T})+\int_t^T\big( h_r(r,X, Z_r^{t,x},\nu^\star_r)-\partial \Yc_r^{r}\big) \d r\bigg]=J(t,t,x,\nu^{\star}).
\end{align}
Finally, arguing as in \cite[Proposition 4.6]{cvitanic2015dynamic}, \eqref{Eq:veri2} yields that for $\Pc(\xb)-\qe x\in \Xc$ and $t\in [0,T]$
\begin{align*}
v(t,x)=\sup_{\P\in \Pc(t,x)} \E^\P\big[Y_t^{t,x}\big].
\end{align*}

It remains to show $\nu^\star \in \Ec(\xb)$. Let $(\eps, \ell, t,x,\nu)\in \R_+^\star \times (0, \ell_\eps)\times  [0,T]\times \Omega\times  \Ac(t,x)$, $\ell_\eps$ to be chosen, and $\nu \otimes_{t+\ell} \nu^\star$. Recall we established that for $\Pc(\xb)$--$\qe x\in \Xc$, $t\in [0,T]$
\begin{align*}
J(t,t,x,\nu^{\star})=\xi(T,X)+\int_t^T F_r(X,Z_r^{t,x},\sigmah_r^2,\partial Y_r^r)\d r-\int_t^T Z_r^{t,x} \cdot \d X_r +K_T^{t,x,\P}-K_t^{t,x,\P}, \; \Pc(t,x)\text{\rm--}\qs
\end{align*}
By computing the expectation of the stochastic integral under $\overline \P^{\nu\otimes_{t+\ell}\nu^\star}$ we obtain
\begin{align*}
&J(t,t,x,\nu^{\star})-J(t,t,x,\nu \otimes_{t+\ell} \nu^\star)\\
\geq&\ \E^{\overline \P^{\nu\otimes_{t+\ell}\nu^\star}}\bigg[\xi(T,X)+\int_t^{T} \big(h_r(r,X,Z_r^{t,x},\nu^\star_r)-b_r(X,(\nu\otimes_{t+\ell}\nu^\star)_r) \cdot\sigmah_r^\t Z_r^{t,x}\big)\d r\nonumber\\ 
& - \xi (t,X) - \int_t^{T}\big(f_r(t,X,(\nu\otimes_{t+\ell}\nu^\star)_r)+\partial Y_r^r \big)\d r \bigg]\nonumber \\
 =&\ \E^{\overline \P^{\nu\otimes_{t+\ell}\nu^\star}}  \bigg[\xi(T,X)-\xi(t,X)+ \int_{t}^T h_r(r,X,Z_r^{t,x},\nu^\star_r) - h_r(t,X,Z_r^{t,x},\nu^\star_r)-\partial Y_r^r\d r \bigg] \nonumber \\
&+ \int_t^{t+\ell}  h_r(r,X,Z_r^{t,x},\nu^\star_r )-h_r(t,X,Z_r^{t,x},\nu_r)+  h_r(t,X,Z_r^{t,x},\nu^\star_r )-h_r(r,X,Z_r^{t,x},\nu_r)\d r\bigg],\nonumber
\end{align*}
where the inequality follows from dropping the $K$ term. Now, since for $\Pc(\xb)$--$\qe x\in \Xc$, $\Vc^\star_t(X,Z^{s,x}_t)=\Vc^\star_t(X,Z_t), \d t \otimes \d \P\text{\rm--}\ae$ on $[s,T]\times \Xc$ for all $\P \in \Pc(s,x)$, we have the previous expression is greater or equal than the sum of
\begin{align*}
I_1&:=\E^{\overline \P^{\nu\otimes_{t+\ell}\nu^\star}}  \bigg[ \xi(T,X)-\xi(t+\ell,X) +\int_{t+\ell}^T\big( h_r(r,X,Z_r^{t,x},\nu^\star_r) - h_r(t + \ell,X,Z_r^{t,x},\nu^\star_r)-\partial Y_r^r\big)\d r\bigg], \\
I_2&:= \E^{\overline \P^{\nu\otimes_{t+\ell}\nu^\star}}  \bigg[\xi(t+ \ell,X)-\xi(t,X) - \int_{t}^T f_r(t,X,\nu^\star_r)\d r+ \int_{t+\ell}^T f_r(t+\ell ,X,\nu^\star_r)\d r+\int_t^{t+\ell}\big(f_r(r,X,\nu^\star_r)- \partial Y_r^r\big) \d r \bigg],\\
I_3&:= \E^{\overline \P^{\nu\otimes_{t+\ell}\nu^\star}}  \bigg[\int_t^{t+\ell} \big( f_r(r,X,\nu_r )-f_r(t,X,\nu_r) + f_r(t,X,\nu_r^\star )-f_r(r,X,\nu_r^\star)\big)\d r\bigg].
\end{align*}

We now study each remaining terms separately. First, regarding $I_1$, by conditioning we can see this term equals 0. Indeed, this follows analogously to \eqref{Eq:veri1}, by using the fact that $\big(\delta_{\omega}\otimes_{t+\ell}\overline \P^{\nu^\star,t+\ell,x}\big )_{\omega \in \Omega}$ is an r.c.p.d. of $\overline \P^{\nu\otimes_{t+\ell}\nu^\star} |\Fc_{t+\ell}^X$, see \Cref{Lemma:forwardcondition}, together with \Cref{Lemma:Dynamicsdiagonal}.\medskip

We can next use Fubini's theorem and $\Phi$ as in \Cref{AssumptionC} to express the term $I_2$ as
\begin{align*}
I_2&= \E^{\overline \P^{\nu\otimes_{t+\ell}\nu^\star}}  \bigg[ \int_t^{t+\ell} \partial_s \xi(r,X) \d r +\int_{t}^T \int_r^T \partial_s f_u (r,X,\nu^\star_u)\d u \d r- \int_{t+\ell}^T \int_{r}^T \partial_s f_u (r,X,\nu^\star_u)\d u \d r -\int_t^{t+\ell} \partial Y_r^r \d r  \bigg]\\
&=  \E^{\overline \P^{\nu}}  \bigg[   \int_t^{t+\ell} \E^{\overline \P^{\nu^\star}_{t+\ell,\cdot}}  \big[\Phi(r,X)]-  \E^{\overline \P^{\nu^\star}_{r,\cdot}}  \big[\Phi(r,X)  \big]\d r\bigg],
\end{align*}
where the second equality follows by conditioning, see \Cref{Lemma:forwardcondition}. Now, arguing as in the proof of \cite[Corollary 6.3.3]{stroock2007multidimensional}, under the weak uniqueness assumption for fixed actions ,$\P^{\nu^\star}_{t_n,x_n}\longrightarrow \P^{\nu^\star}_{t,x}$ weakly, whenever $(t_n,x_n)\longrightarrow (t,x)$.\medskip

By \Cref{AssumptionC}\ref{AssumptionC:iib}, for every $(r,t,x)\in [0,T]^2\times \Xc$, $ \E^{\overline \P^{\nu^\star}_{t+\ell,x}}  \big[\Phi(r,X)]\longrightarrow \E^{\overline \P^{\nu^\star}_{t,x}}  \big[\Phi(r,X)]$, $\ell \longrightarrow 0$. Moreover, as $t\longmapsto \Phi(t,x)$ is clearly continuous and $(r,t)\in [0,T]^2$, the previous convergence holds uniformly in $(r,t)$. Together with bounded convergence we obtain that for $\Pc(\xb)$--$\qe x\in \Xc$ and $\nu\in \Ac(t,x)$, $\E^{\overline \P^{\nu\otimes_{t+\ell}\nu^\star}}  \big[\Phi(r,X)\big] \longrightarrow  \E^{\overline \P^{\nu\otimes_{t}\nu^\star}}  \big[\Phi(r,X)\big]$, $\ell \longrightarrow 0$, uniformly in $(t,r)$.\medskip

We now argue that the above convergence holds uniformly in $\nu$. Indeed, \Cref{AssumptionB}\ref{AssumptionB:iii} guarantees that for $\Pc(\xb)$--$\qe x\in \Xc$ the family $ \Mc_{t+\ell}^x(\tilde x):=\E^{\overline \P^{\nu^\star}_{t+\ell,\tilde x}}[\Phi(r,X)]-\E^{\overline \P^{\nu^\star}_{t,x}}[\Phi(r,X)]$, is $\Pc(t,x)$--uniformly integrable. Thus, provided $\Pc(t,x)$ is weakly compact, an application of the non-dominated monotone convergence theorem, see \cite[Theorem 31]{denis2011function}, yields the result. In order to bypass the compactness assumption on $\Pc(t,x)$, we consider the compact set $\Ac^{\text{rel}}(t,x)$, see \cite[Theorem 4.1]{karoui2015capacities2}, of solutions to the martingale problem for which relaxed action processes are allowed, \emph{i.e.} ignoring condition $(iii)$ in the definition of $\Pc(t,x)$. By \cite[Theorem 4.5]{karoui2015capacities2}, the supremum over the two families coincide. With this we can find $\ell_\eps$ such that for $\ell<\ell_\eps$
\begin{align*}
\int_t^{t+\ell} \sup_{\nu \in \Ac(t,x)} \Big| \E^{\overline \P^{\nu\otimes_{t+\ell}\nu^\star}}  \big[\Phi(r,X)\big] -  \E^{\overline \P^{\nu\otimes_{t}\nu^\star}}  \big[\Phi(r,X)  \big] \Big| \d r \leq  \eps \ell.
\end{align*}
Finally, to control $I_3$, we see that \Cref{AssumptionC}\ref{AssumptionC:ii} guarantees there is $\ell_\eps$ such that for all $(r,x,a)\in[0,T]\times \Xc \times A,\ |f_r(s,x,a)-f_r(t,x,a)|<\eps/2$ whenever $|s-t|<\ell_\eps$, so that
\begin{align*}
\E^{\overline \P^{\nu \otimes_{t+\ell}\nu^\star}} \bigg[ \int_t^{t+\ell} |f_r(r,X,\nu_r)dr-f_r(t,X,\nu_r)|+|f_r(t,X,\nu_r^\star)-f_r(r,X,\nu_r^\star) |dr\bigg] \leq \eps \ell.
\end{align*}
Combining the previous arguments, we obtain that for $0<\ell<\ell_\eps$, $\Pc(\xb)$--$\qe x \in \Xc$ and $(t,\nu)\in [0,T]\times \Ac(t,x)$ 
\[
J(t,t,x,\nu^\star)-J(t,t,x,\nu \otimes_{t+\ell} \nu^\star)\geq-\eps \ell.
\]
\end{proof}

\subsection{Well-posedness: the uncontrolled volatility case}

We start this section studying how System \eqref{HJB} reduces when no control on the volatility is allowed. Intuitively speaking the first equation should reduce to a standard BSDE and under our assumption of weak uniqueness for \eqref{Eq:driftlessSDE} we end up with only one probability measure which allows a probabilistic representation of the second element in the system. We first study the reduction in the next proposition.

\begin{proposition}\label{HJB:to:HJB0}
Suppose $\sigma_t(x,a)=\sigma_t(x,\tilde a)=:\sigma_t (x)$ for all $a\in A$, \emph{i.e.} the volatility is not controlled, then {\rm System} \eqref{HJB} reduces to
\begin{align}\tag{H\textsubscript{o}}
\begin{split}
Y_t&=\xi(T,X_{\cdot\wedge T})+\int_t^T H_r^o(X,Z_r,\partial Y_r^r)\d r-\int_t^T  Z_r \cdot \d X_r,\; 0\leq t \leq T,\; \P\text{--}\as, \\
\partial Y_t^s&=\partial_s \xi(s,X_{\cdot\wedge T})+\int_t^T \partial h_r^o(s,X,\partial Z_r^s, \Vc^\star(r,X, Z_r)) \d r-\int_t^T\partial Z_r^s \cdot \d X_r,\; 0\leq  t \leq T,\; \P\text{--}\as,\; 0\leq s \leq T.
\end{split}
\end{align}
\begin{proof}
As the volatility is not controlled there is a unique solution to the martingale problem \eqref{Eq:MartingaleProblem}, \emph{i.e.} $\Pc(\xb)=\{\P\}$. In addition, since $(\sigma \sigma^\t)_t(x,a)=(\sigma \sigma^\t)_t(x)$ for all $t \in [0,T]$, then
\begin{align*}
\Sigma_t(x)=\{(\sigma \sigma^\t)_t(x)\} \in \S_n^+(\R),\; A_t(x, \sigma_t(x))=A.
\end{align*}

Let $(Y,Z,K)$ be a solution to the 2BSDE in \eqref{HJB}. As $\Pc(\xb)=\{\P\}$, the minimality condition implies that the process $K^\P$ vanishes in the dynamics, thus $(Y,Z)$ is a solution to the first BSDE in \eqref{HJB0}. Now as the family $\partial Y$ is defined $\P$--$\as$, $Y$ is well-defined in the $\P$--$\as$ sense too. Finally, $\Pc(\xb)=\{\P\}$ guarantees that the predictable martingale representation holds for $(\F^{X,\P}_+,\P)$-martingales. With this, it follows that for $s\in [0,T]$, $\partial Y^s$ in \eqref{HJB} admits the representation in \eqref{HJB0}, which holds up to a $\P$-null set.
\end{proof}
\end{proposition}

\begin{remark}[HJB system exponential case]
As a sanity check at this point we can check what the above system leads to in the case of exponential discounting in a non-Markovian framework. Defining $f(s,t,x,a)$ and $F(s,x)$ as in {\rm\Cref{Remark:DPP:sanotycheck}}, note that
\[
J(t,x,\nu)=\E^{\P^\nu}\bigg[\int_t^T \mathrm{e}^{-\theta (r-t)}\tilde f(r,X,\nu_r)\d r+\mathrm{e}^{-\theta (T-t)}\tilde F(X_{\cdot \wedge T})\bigg]=Y_t^t.
\]
Notice that
\begin{align*}
&\partial Y_t^s=\theta \mathrm{e}^{-\theta (T-s)} +\int_t^T\big( \theta \mathrm{e}^{-\theta (r-s)}\tilde f(r,X,\nu_r^{\star})+b(r,X,\nu_r^{\star})\cdot \sigma(r,X)^\t Z_r^s \big) \d r-\int_t^T\partial Z_r^s \cdot \d X_r,
\end{align*}
and as it does turn out that $Z_r^r=Z_r$, see {\rm\Cref{Lemma:Derivative}} and {\rm\Cref{Verification}}, we get
\begin{align*}
\partial Y_t^t&= \theta \mathrm{e}^{-\theta (T-t)} +\int_t^T \big( \theta \mathrm{e}^{-\theta (r-t)}\tilde f(r,X,\nu^{\star}_r)+b(r,X,\nu_r^{\star})\cdot \sigma(r,X)^\t Z_r^t \big)\d r-\int_t^T\partial Z_r^t \cdot \d X_r\\
&=\theta \E^{\P^{\nu^{\star}}}\bigg[\int_t^T \mathrm{e}^{-\theta (r-t)}\tilde f(r,X,\nu_r^{\star})\d r+e^{-\theta (T-t)}\tilde F(X_{\cdot \wedge T})\Big| \Fc_{t+}^{X,\P} \bigg]=\theta Y_t^t.
\end{align*}
Thus
\begin{align*}
& Y_t =\xi(T)+\int_t^T H_r^o(X_r,Z_r^{\nu^{\star}},\theta Y_r)\d r - \int_t^T Z_r^{\nu^{\star}}\cdot \mathrm{d}X_r,\; \nu^{\star}_t(x, z,u):=\argmax_{a\in A}  \big \{ h_t^o(t,x,z,a)\big \}.
\end{align*}
In the classic Brownian filtration set up, \emph{e.g.} assuming $\sigma$ is non-degenerate and $n=d$, the above {\rm BSDE} corresponds to the well-known solution to an optimal stochastic control problem with exponential discounting, see \emph{\cite{zhang2017backward}}.
\end{remark}

The general treatment of systems as \eqref{HJB0} is carried out in the appendix.

\begin{proof}[Proof of {\rm \Cref{Thm:wellposedness:driftcontrol}}]
The result is immediate from \Cref{Existent:Abstract}, see the appendix.
\end{proof}

\begin{remark}\label{Remark:wp}
The general well-posedness result, \emph{i.e.} in which both the drift and the volatility are controlled remains open. In fact, as this requires to be able to guarantee the existence of a probability measure $\P^\star$ under which the minimality condition \eqref{Min:Condition:NoA} is attained, we believe a feasible direction to attain this result is to go one level beyond the weak formulation, and work in the relaxed framework, see for example {\rm\cite{karoui2015capacities2}}.

\end{remark}

\section{Extensions of our results}\label{Section:Extensions}

We now present an extension of our results to more general classes of pay-off functionals, as in \cite[Section 7.4]{bjork2016time2}. The dynamics of the controlled process $X$ remains as in \Cref{Section:ControlledDynamics}. {\color{black} We only present the corresponding results, the proofs are analogous to those presented in this document and are available in \citeauthor*{hernandez2021me} \cite{hernandez2021me}}. We will consider
\begin{align*}
 &\f:[0,T]\times  \Xc\longrightarrow \R, \text{ Borel-measurable, with } \f_\cdot(\cdot)\; \F^X\text{-optional};\;  g:\Xc\longrightarrow \R, \text{ Borel-measurable};\\
& \xi:[0,T]\times \Xc \longrightarrow \R, \text{ Borel-measurable} ;\;
 G:[0,T] \times  \R\longrightarrow \R,  \text{ Borel-measurable},\\
&f:[0,T]^2\times \Xc \times \R \times  A \longrightarrow \R, \text{ Borel-measurable, with } f_\cdot(s,{\rm n},\cdot,a)\; \F^X\text{-optional, }
\end{align*}
$ \text{for any } (s,{\rm n},a) \in [0,T] \times \R \times  A$,
and define for $(s,t,x,\nu)\in [0,T]^2\times \Xc\times \Ac(t,x)$
\[ J(s,t,x,\nu)=\E^{\overline \P^{\nu}_{t,x}}\bigg[\int_t^T f_r\Big(s,X,\E^{\overline \P^\nu_{s,x}}[\f_r(X)],\nu_r\Big)\d r+ \xi(s,X_{\cdot \wedge T}) \bigg]+ G\Big(s,\E^{\overline\P^{\nu}_{s,x}}\big[g(X_{\cdot\wedge T})]\Big) .
\]
As a motivation for the consideration for this kind of pay-off functionals, notice that the presence of the term $G(s,\E^{\P^\nu_{t,x}}[g(x)])$ allows, for example, to include classic mean--variance models into the analysis. In order to present the corresponding DPP in this framework we need to adapt the notation and assumptions that led to it.\medskip

Recall the convention $\partial^2_{{\rm nn}} f_t(s,x,{\rm n},a):=\frac{\partial^2}{\partial {\rm n}^2} f_t(s,x,{\rm n},a)$ denotes the respective derivatives. Let $\nu^\star\in \Ec(\xb)$ and define
\begin{align*}
M_t^\star(x):=\E^{\overline \P^{\nu^\star}_{t,x}}\big[\g(X_{\cdot\wedge T})\big],\; N_t^{s,\star}(x):= \E^{\overline \P^{\nu^\star}_{s,x}}[\f_t(X_{\cdot\wedge t})],\; (s,t,x)\in  [0,T]^2\times \Xc.
\end{align*}

We emphasise that $N$ defines an infinite family of processes, when considered as functions of $s$, one for every $t\in [0,T]$. We also recall that under the weak uniqueness assumption, both processes are well--defined. Moreover, we know that the application $(t,x)\longmapsto \P_{t,x}^{\nu^\star}$ is measurable and continuous for the weak topology, see \cite[Corollary 6.3.3]{stroock2007multidimensional} and the preceding comments.\medskip

To be able to extend our results, we work under the following set of assumptions.

\begin{assumption}\label{AssumptionDppExt}
{\rm\Cref{AssumptionA}\ref{AssumptionA:DPP2} and \ref{AssumptionA}\ref{AssumptionA:DPP}} together with
\begin{enumerate}[label=$(\roman*)$, ref=.$(\roman*)$,wide, labelindent=0pt]

\item $ s \longmapsto \xi(s,x)$ is continuously differentiable uniformly in $x$. $(s,{\sf m})\longmapsto G(s,{\sf m})$ belongs to $\Cc_{1,2}([0,T]\times \R,\R)$ with spatial derivatives Lipschitz-continuous uniformly in $s$. $(s,{\rm n})\longmapsto f_t(s,x,{\rm n},a)$ belongs to $ \Cc_{1,2}([0,T]\times \R,\R)$ with spatial derivatives Lipschitz-continuous uniformly in $(s,t,x,a)$.  \label{AssumptionDppExt1}

\item \label{AssumptionDppExt2} $f_t(s,x,{\rm n},a)= \tilde f_t(s,x,a)+\hat f_t(s,{\rm n})$ for an $\F^X$-optional $($resp. deterministic$)$ mapping $\tilde f$ $($resp. $\hat f)$. The mappings $x \longmapsto \f_t(x)$, $t\longmapsto \f_t(x)$, and $x\longmapsto g(x)$ are continuous uniformly in the other variables. 

\item \label{AssumptionDppExt2.5} $\exists C>0$, $\rho:(0,\infty)\longrightarrow [0,\infty)$, $\rho(|\ell|) \longrightarrow 0$, $\ell\longrightarrow 0$, such that for $\Pc(\xb)\qe \; x\in \Xc,\nu\in \Ac(t, x)$, $t\leq t^\prime \leq r \leq T$, 
\begin{align*}
\E^{\overline \P^{\nu}}\bigg[  \Big| \E^{\overline \P^{\nu}_{t ,\cdot }} [N_r^{t^\prime,\star}]- N_r^{t,\star}\Big|^2 +  \Big| \E^{\overline \P^{\nu}_{t ,\cdot }} [M_{t^\prime}^{\star}]- M_t^{\star}\Big|^2\bigg]  \leq C|t^\prime-t|\rho(|t^\prime-t|).
\end{align*}
\end{enumerate}
\end{assumption}

\begin{remark}\label{remark:ext1}
We would like to comment on the previous set of assumptions. The extensions of our previous set of assumptions correspond to $(i)$, $(iv)$ and $(v)$. In addition, the reader might have noticed the assumptions imposed on $f$, $\f$ and $g$ in {\rm \Cref{AssumptionDppExt}\ref{AssumptionDppExt2}} and {\rm \Cref{AssumptionDppExt}\ref{AssumptionDppExt2.5}}. The condition on $f$ basically disentangles the randomness coming from $X_{\cdot\wedge r}$ and $\E^{\overline \P^\nu_{s,x}}[f_r(X_{\cdot\wedge r})]$. This helps us bypass some measurability issues arising from the interaction between these two variables. Given the non-linear dependence of the reward, when passing to the limit in the proof of {\rm \Cref{Theorem:DPP:Ext}} below, one should expect that first order, \emph{i.e.} linear, approximations, would not suffice to rigorously obtain the limit. Not surprisingly, it is necessary to have access to the quadratic variations of the previously introduced processes. This is usually carried out by having a pathwise construction of the stochastic integral. 
For this, a viable way is to follow the approach in {\rm \cite{karandikar1995pathwise}}. It is thus necessary to guarantee that $M^\star$ and $N^{\star,\cdot}_r$ are left limits of c\`adl\`ag processes. In light of the continuity of the map $(t,x)\longmapsto \P_{t,x}^{\nu^\star}$, {\rm \Cref{AssumptionDppExt}\ref{AssumptionDppExt2}} ensures that these processes are continuous. Hence, there exists a process $[ M^\star ]$ $($resp. $[ N^{\star,\cdot}_r ]$ for any $r\in[0,T])$ which coincides with the quadratic variation of $M^\star$ $($resp. $N^{\star,\cdot }_r$ for any $r\in[0,T])$ under $ \P^{\nu^\star}$.

\medskip
Moreover, ${\rm f}$ and $g$ can be understood as \emph{changes of variables} from the canonical process $X$. As such, it is expected to require some control on the quadratic difference under the laws induced by an arbitrary action $\nu$ and the equilibrium $\nu^\star$. This is precisely the goal of {\rm \Cref{AssumptionDppExt}\ref{AssumptionDppExt2.5}}. We highlight that both processes appearing in {\rm \Cref{AssumptionDppExt}\ref{AssumptionDppExt2.5}} are $\Fc_t$-measurable and differ only, from $\nu$ to $\nu^\star$, on the action performed over the interval $[t,t^\prime]$. In fact, when $\nu=\nu^\star$ this condition holds trivially as both expression are equal to zero. In fact, when $\nu=\nu^\star$ this condition holds trivially as both expression are equal to zero. It is also possible to verify this condition in the case of uncontrolled volatility if ${\rm f}$ and $g$ are regular in the sense of {\rm \citeauthor*{cont2010change} \cite{cont2010change}} so that the functional It{\^ o} formula holds, see below.

\end{remark}

\begin{lemma}\label{lemma:assumpextendeddpp}
{\rm \Cref{AssumptionDppExt}\ref{AssumptionDppExt2.5}} holds if either
\begin{enumerate}[label=$(\roman*)$, ref=.$(\roman*)$,wide, labelindent=0pt]
\item the volatility is uncontrolled, $A$ is bounded, $\f$ has bounded horizontal, first and second order vertical derivatives,  $\Dc_t f$, $\nabla_xf$ and $\nabla^2_{x} $, respectively. Moreover, the process $\Af(t,X_{\cdot\wedge t},\nu_t)$ is square integrable for any $\nu\in \Ac$, where $\Af(t,x,a):= \Dc_t f_t(x)+\sigma_t(x)b_u(x,a) \nabla_x f_u(x)+ \frac{1}{2}(\sigma\sigma^\t)_u(x)\nabla^2_{x} f_u(x);$ 

\item the problem is in strong formulation with state dependent coefficients, unique strong solution. This is, there is a probability space $(\Omega, \Fc,\F,\P)$ and a $\P$--Brownian motion $W$ such that for any $\nu \in \Ac(\xb)$ there is a unique process $X^{x,\nu}$ that satisfies
\[ X_t^{x,\nu}=x_0+\int_0^t b_r(X_r,\nu_r)\d r +\int_0^t \sigma_r(X_r,\nu_r)\d W_r, \; t\in [0,T], \; \P\text{\rm--}\as\]

Moreover, $(\f,\g)\in \Cc_{1,2}([0,T]\times \R,\R)\times \Cc_2(\R,\R)$ and $ \Ac(t,X_t,\nu_t)$ is square integrable for any $\nu\in \Ac$, where $\Ac(t,x,a):= \partial_t \f_t(x)+\sigma_t(x,a)b_u(x,a) \partial_x \f_u(x)+ \frac{1}{2}(\sigma\sigma^\t)_u(a,x)\partial^2_{xx} \f_u(x)$.
\end{enumerate}
\end{lemma}

In light of \Cref{remark:ext1} we define $\mh$ and $\nh$ square roots of the processes
\begin{align*}
\mh^{2}_t :=\limsup_{\eps \searrow 0} \frac{[ M^\star ]_t-[ M^\star ]_{t-\eps}}{\eps},\;  \nh_r^{t\, 2} :=\limsup_{\eps \searrow 0} \frac{[ N^{\cdot,\star}_r ]_t-[ N^{\cdot,\star}_r ]_{t-\eps}}{\eps}, \; (r,t)\in [0,T]^2.
\end{align*}

Building upon our previous analysis, we can profit from the recent results available in \citeauthor*{djete2019mckean} \cite{djete2019mckean} to obtain the next DPP. We also remark that time-inconsistent McKean--Vlasov problems have been recently studied by \citeauthor*{mei2020closed} \cite{mei2020closed}.\medskip

\begin{theorem}\label{Theorem:DPP:Ext}
Let {\rm\Cref{AssumptionDppExt}} hold, and $\nu^\star \in \Ec(\xb)$. For $\{\sigma, \tau\}\subseteq \Tc_{t,T}$, $\sigma\leq  \tau$ and $\Pc(\xb)\qe\; x \in \Xc$, we have
\begin{align*}
\begin{split}
v(\sigma,x) = \! \! \sup_{\nu \in \Ac(\sigma,x)} & \E^{\overline \P^\nu}\bigg[ v(\tau,X)+ \!  \int_{\sigma}^\tau\!    \bigg(  f_r(r,X,\f_r(X),\nu_r)- \E^{{\overline \P}^{\nu^\star}_{r,\cdot}} \bigg[ \partial_s G(r,M_r^\star) - \frac{1}2  \partial_{{\sf mm}}^2  G(r,M_r^\star) \mh^2_r\bigg]  \bigg)\mathrm{d}r \\
&\!\!\!\!\!\! - \int_{\sigma}^\tau\E^{{\overline \P}^{\nu^\star}_{r,\cdot}} \bigg[\partial_s\xi(r,X)+ \int_r^T   \Big( \partial_s f_u(r,X,N^{r,\star}_u ,\nu^\star_u)+   \frac{1}2 \partial_{{\rm nn} }^2 f_u(r,X,N^{r,\star}_u,\nu^\star_u) \nh_u^{r\, 2}  \Big) \d u\bigg] \d r\bigg]
\end{split}
\end{align*}
\end{theorem}

Analogously, we can associate a system of BSDEs to the problem. Define for $(s,t,x, z,\gamma,\Sigma,u,v,$ ${\rm n},{\rm z},{\sf m},{\sf z}, a)\in [0,T)^2 \times \Xc \times \R^d\times \S_d(\R)\times \S_d(\R) \times \R \times \R^d\times \R \times \R^d\times  \R \times \R^d\times A$ 
\begin{align*}
h_t(s,x,z,a)&:= f_t(s,x,\f_t(x),a)+ b_t(x,a)\cdot\sigma_t(x,a)^\t z;\\
F_t(x,z,\Sigma,u,{\sf m},{\sf z})& :=\sup_{a \in A_t(x,\Sigma)} \big\{ h_t(t,x,z,a)\big \}-u-\frac{1}{2}\, {\sf z}^\t \Sigma \, {\sf z} \,  \partial_{{\sf mm}}^2 G(t,{\sf m}),
\end{align*}
and $\Vc^\star(t,x,z)$ denotes the unique (for simplicity) $A$-valued Borel-measurable map satisfying
\begin{align*}
[0,T]\times\Xc\times\R^d\ni (t,x,z)\longmapsto \Vc^\star (t,x,z)\in \argmax _{a\in A} h_t(t,x,z,a).
\end{align*}

In light of \Cref{Theorem:DPP:Ext}, we consider for $(s,t)\in [0,T]^2$ and $\omega\in \Omega$ the system 
\begin{align}\label{HJB2}\tag{H\textsubscript{e}}
\begin{split}
& Y_t=\xi(T,X)+G(T,g(X))+\int_t^T F_r(X,Z_r,\sigmah_r^2, \partial Y_r^r,M^\star_r,\mh^2_t)\d r-\int_t^T  Z_r \cdot \d X_r+K^\P_T-K^\P_t,\\
& \partial Y_t^s(\omega)=\E^{{\overline \P}^{\nu^\star}_{t,x}}\bigg[ \partial_s \xi(s,X)+ \int_t^T \Big( \partial_s f_r^\star(s,X, N^{\star,s}_r, Z_r)+ \frac{1}2 \partial_{{\rm nn}}^2 f_r^\star(s,X, N^{s,\star}_r,Z_r) \nh^{\star\, 2}_r\Big) \d r\bigg],\\
&M^{\star}_t(\omega)=\E^{{\overline \P}^{\nu^\star}_{t,\omega}}\big[ g(X) \big],\; N^{\star,s}_t(\omega)=\E^{{\overline\P}^{\nu^\star}_{s,\omega}}\big[\f_t(X) \big],\;   \mh^2_t:=\frac{ \d [ M^\star]_t}{\d t},\;  \nh_r^{t\, 2}:=\frac{ \d [ N^{\star,\cdot}_r ]_t}{\d t}  ,
\end{split}
\end{align}
where  $\partial_s f_t^\star(s,x, n, z):=\partial_s f_t^\star(s,x, n, \Vc^\star(t,x,z))$ and $\partial_{\rm n n}^2 f_t^\star(s,x, n, z):=\partial_{\rm n n}^2  f_t^\star(s,x, n, \Vc^\star(t,x,z))$.\medskip

In the case of drift control only, we define
\begin{align*}
h_t^o(s,x,z,a)&:= f_t(s,x,n,a)+ b_t(x,a)\cdot  \sigma_t(x)^\t z,\\
 \partial h_t^o(s,x,v,{\rm n},{\rm z},a)&:= \partial_s f_t	(s,x,{\rm n},a)+  b_t(x,a)\cdot \sigma_t(x)^\t  v +\frac{1}{2}{\rm z}^\t\sigma_t(x)^\t\sigma_t(x) \, {\rm z}\, \partial_{{\rm nn}}^2 f_t(s,x,{\rm n},a),\\
H_t^o(x,z,u,{\sf m},{\sf z})& :=\sup_{a \in A} \big\{ h_t^o(t,x,z,a)\big \}-u-\partial_s G(t,{\sf m})-\frac{1}{2}\ {\sf z}^\t \sigma_t(x)^\t \sigma_t(x) \, {\sf z} \  \partial_{{\sf mm}}^2 G(t,{\sf m}) ,
\end{align*}
and \eqref{HJB2} reduces to the infinite family of BSDEs which for any $s\in [0,T]$ (recall the notations in \Cref{Section:BSDEassociated}) satisfies
\begin{align}\label{HJB3}\tag{$\text{H}_\e^\o$}
\begin{split}
Y_t&=\xi(T,X)+G(T,g(X))+\int_t^T H_r^o(X, Z_r, \partial Y_r^r,M^\star_r, Z^\star_r)\d r-\int_t^T  Z_r \cdot \d X_r, \, t\in [0,T],\,\P\text{\rm --a.s.},\\
\partial Y_t^s&=\partial_s \xi(s,X)+\int_t^T \partial h_r^o\big(s,X,\partial Z_r^s, N^{s,\star}_r, Z^{s,\star}_r, \Vc^\star(r,X,  Z_r)\big) \d r-\int_t^T\partial Z_r^s \cdot \d X_r,  \, t\in [0,T],\,\P\text{\rm --a.s.},\\
M^{\star}_t&=g(X) +\int_t^T b_r\big(X,\Vc^\star(r,X,  Z_r)\big)\cdot \sigma_r^\t(X) Z^{\star}_r\d r  -\int_t^T  Z^{\star}_r \cdot \d X_r, \, t\in [0,T],\,\P\text{\rm --a.s.},\\
N^{s,\star}_t&=\f_t(X) +\int_s^T b_r\big(X,\Vc^\star(r,X, Z_r)\big)\cdot \sigma_r^\t(X)Z^{r,\star}_t \d r  -\int_s^T  Z^{r,\star}_t \cdot \d X_r,\, t\in [0,T],\,\P\text{\rm --a.s.}
\end{split}
\end{align}

In the same way, a necessity theorem holds. It does require us to introduce the following set of assumptions.

\begin{assumption}\label{AssumptionNecExt} {\rm\Cref{AssumptionB}\ref{AssumptionB:i} and \ref{AssumptionB}\ref{AssumptionB:ii}} together with
\begin{enumerate}[label=$(\roman*)$, ref=.$(\roman*)$,wide, labelindent=0pt]
\item  \label{AssumptionNecExt:iii} there exists $p>1$ such that for every $(s,t,x)\in [0,T]^2\times \Xc$
\begin{align*}
\sup_{\P\in \Pc(t,x)} \E^\P\bigg [ & |\xi(T,X)|^p+ |\partial_s \xi(s,X)|^p+|G(T,g(X))|^p+|g(X)|^p+|f_s(X)|^p\\
&\; +\int_t^T |F_r(X,0,\widehat \sigma^2_r ,0,0,0)|^p +|\partial h_r^o(s,X,0,N_r^{s,\star}, \nh^{s,\star}_r,\nu^\star_r)|^p \d r\bigg] < \infty.
\end{align*}
\end{enumerate}
\end{assumption}

\begin{theorem}[Necessity]\label{Theorem:NecessityExt}
Let {\rm\Cref{AssumptionDppExt}} and {\rm\Cref{AssumptionNecExt}} hold. Given $\nu^\star \in \Ec(\xb)$, one can construct $(Y,Z,(K^\P)_{\P\in \Pc(\xb)}, \partial Y)$ solution to \eqref{HJB}, such that for any $t\in [0,T]$ and $\Pc(\xb)\qe\; x\in \Xc$
 \[
v(t,x)=\sup_{\P\in\Pc(t,x)} \E^{ \P} \big[Y_t \big].
\] Moreover, $\nu^\star$ satisfies {\rm\Cref{Def:sol:systemH}\ref{Def:sol:systemH:iii}}, \emph{i.e.} $\nu^\star$ is a maximiser of the Hamiltonian.
\end{theorem}

\begin{remark}
We would like to comment that the well-posedness of \eqref{HJB3}, \emph{i.e.} the extended system when only drift control is allowed, remains a much harder task. In particular, it is known that the presence of a non-linear functionals of conditional expectations opens the door to scenarii with multiplicity of equilibria with different game values, see {\rm \cite{landriault2018equilibrium}} for an example in a mean--variance investment problem. Consequently and in line with current results available for systems of {\rm BSDEs} with quadratic growth, see {\rm \citeauthor*{frei2011financial} \cite{frei2011financial}}, {\rm \citeauthor*{harter2019stability} \cite{harter2019stability}}, and {\rm\citeauthor*{xing2016class} \cite{xing2016class}}, we expect to be able to obtain existence of a solution, but not necessarily uniquness.
\end{remark}

{\small
\bibliography{bibliography}
}

\begin{appendix}

\section{Appendix}


\subsection{Optimal investment and consumption for log utility}\label{Appendix:examples}

We provide the necessary results for \Cref{Section:Examples}. We start with expressions to determine the functions $a(\cdot)$ and $b(\cdot)$
\[
p:=\frac{1}{\eta}, \; q:=1-\frac{1}{\eta}, \; \alpha_1(t):= r+\frac{1}{2} \beta^2 p + a(t)^{-p},\; t\in[0,T], \; \alpha_2(t):=(1-\eta)(\alpha_1(t)+\frac{\beta^2}{2\eta^2}(1+\eta)),\; t\in[0,T].
\]
Under the optimal policy $(c^\star,\gamma^\star)$ we have that $\P^{\nu^\star}$--$\as$
\begin{align*}
\d X_t= X_t(r+\beta^2\eta^{-1} + a(r)^{-\frac1{\eta}})&\d t+ \beta \eta^{-1} X_t \d W_t,\; \d {X_t}^{1-\eta} =   (1-\eta)X_t^{1-\eta}\big[\alpha_1(t) \d t +p \beta \d W_t\big], \\
\d U(c^\star(t,X_t))&=a(t)^{q} X_t^{1-\eta} \big(\big( \alpha_1(t)-p a'(t) a(t)^{-1}\big)  \d t+p \beta  \d W_t\big), 
\end{align*}
which we can use to obtain that for $\P\in \Pc(t,x)$
\begin{align*}
\E^\P[U(c^\star(r,X_r))]&=U(c^\star(t,x))+ x^{1-\eta} \int_t^r \exp\bigg(\int_t^u \alpha_2(v) \d v\bigg)a(u)^{q} (\alpha_1(u)-p a'(u) a(u)^{-1} )\d u, \\
\E^\P [U(X_T)]&=U(x)\exp\bigg(\int_t^T \alpha_2(u)\d u\bigg).
\end{align*}
By direct computation in \eqref{Eq:BSVEverification} one finds that in general $a$ must satisfy
\begin{align}\label{Eq:examplea}\begin{split}
a'(t)&+(1-\eta)a(t)\alpha_1(t)+a(t)^q \varphi(T-t)+\varphi'(T-t)\exp\bigg(\int_t^T\alpha_2(u)\d u\bigg)\\
&+\int_t^T \varphi'(r-t)\int_t^r\exp\bigg(\int_t^r \alpha_2(v)\d v\bigg)a(u)^q\big(\alpha_1(t)-p a'(u)a(u)^{-1}\big)\d u \d r=0,\; t\in[0,T),\; a(T)=1,
\end{split}
\end{align}
where we recall $\alpha_1$ and $\alpha_2$ are actually functions of $a$. The boundary condition follows as $Y_T(x)=U(x)$ for all $x\in \Xc$. The previous equation is, of course, an implicit formula that reflects the non-linearities inherent to the general case. A general expression for $b$ can be written down too. We have refrained from doing so here. Nevertheless, in the particular case $\eta=1$, which corresponds to the $\log$ utility scenario, the expressions involved simplify considerably, which reflects the fact that all non-linearities vanish. Indeed, in this case $p=1$, $q=0$, $\alpha_1=r+\beta/2-a^{-1}$, $\alpha_2=0$ and one obtains
\begin{align*}
b'(t)+ a(t)\alpha_1(t)-\log a(t) - \int_t^T  \varphi'(s-t)\big(A(s)-A(t)- \log [a(s)]\big)\d s&+\varphi '(T-t)(A(T)-A(t))\\
&+\log(x)[a'(t)+ \varphi(T-t)+ \varphi'(T-t)]=0.
\end{align*}
where $A(t)$ denotes the antiderivative of $\alpha_1$. To find $a$ we set  
\[a'(t)+ \varphi(T-t)+ \varphi'(T-t)=0,\; t\in[0,T),\; a(T)=1\]
This determines both $\alpha$ and $A$. $b$ is then given by setting the first line in the above expression equal to zero together with the boundary condition $b(T)=0$.

\subsection{Verification theorem}\label{Appendix:verification}
Throughout this section we assume \Cref{AssumptionC}. We present next a series of lemmata which shed light on the properties satisfied by the 2BSDE in \eqref{HJB}.

\begin{lemmaA}\label{Lemma:shifted2bsde}
For $(\xb,x,s)\in \Xc\times \Xc\times (0,T]$ consider the {\rm 2BSDEs}
\begin{align}\label{Eq:lemma:shifted2bsde}\begin{split}
Y_t&=\xi(T,X_{\cdot \wedge
 T})+\int_t^T  F_r(X,Z_r,\widehat{\sigma}^2_r, \partial Y_r^r) \d r -\int_t^T Z_r \cdot \d X_r +K_T^{\P}-K_t^{\P},\; 0\leq t\leq T,\; \Pc(\xb)\text{\rm--}\qs\\
Y_t^{s,x}&=\xi(T,X_{\cdot \wedge
 T})+\int_t^T  F_r(X,Z_r^{s,x},\widehat{\sigma}^2_r, \partial Y_r^r) \d r -\int_t^T Z_r^{s,x} \cdot \d X_r+K_T^{s,x,\Q}-K_t^{s,x,\Q},\; s\leq t\leq T,\; \Pc(s,x)\text{\rm--}\qs
 \end{split}
\end{align}
Suppose both {\rm2BSDEs} are well-posed. Then, for any $s\in (0,T]$, \begin{align*}
Y_t&=Y_t^{s,x},\; s\leq  t\leq T,\;  \Pc(s,x)\text{\rm--}\qs,\; \text{\rm for } \Pc(\xb)\text{\rm--}\qe\; x\in \Xc,\\
Z_t&=Z_t^{s,x}, \; \sigmah_t^2 \d t\otimes \d \Pc(s,x)\text{\rm--}\qe \text{\rm on } [s,T]\times \Xc, \text{ \rm for } \Pc(\xb)\text{\rm--}\qe\; x\in \Xc,\\
K_t^{\P }&=K_t^{s,x,\P_{s,x} },\; s\leq  t\leq T,\;  \P_{s,x}\text{\rm--}\as,\; \text{for } \P\text{\rm--}\ae\; x\in \Xc, \forall \P\in \Pc(\xb).
\end{align*}
\end{lemmaA}

\begin{proof}
Following \cite{possamai2015stochastic}, we consider for $(t,x)\in [0,T]\times \Xc$
\begin{align*}
\widehat \Yc_t(x):= \sup_{\P\in \Pc(t,x)} \E^\P\big[\Yc_t^\P\big],
\end{align*}
where for an arbitrary $\P\in \Pc(s,x)$, $\Yc^\P$ corresponds to the first coordinate of the solution to the BSDE
\begin{align*}
\Yc_t^\P&=\xi(T,X_{\cdot \wedge T}) +\int_t^T F_r(X,\Zc_r^\P,\sigmah_r^2,\partial Y_r^r)\d r-\int_t^T \Zc_r^\P\cdot \d X_r,\; s\leq t \leq T,\; \P\text{\rm--}\as
\end{align*}
It then follows by \cite[Lemmata 3.2. and 3.6.]{possamai2015stochastic} that $\widehat \Yc^+$, the right limit of $\widehat \Yc$, is $\F^{X,\Pc(\xb)}_+$-measurable, $\Pc(\xb)$--$\qs$ c\`adl\`ag, and for every $\P\in \Pc(\xb)$, there is $(\widehat \Zc^\P,\widehat \Kc^\P) \in \H^p_{\xb}(\F^{X,\P}_+,\P)\times \I^p_{\xb}(\F^{X,\P}_+,\P)$ such that for every $\P\in \Pc(\xb)$
\begin{align*}
\widehat \Yc_t^+&=\xi(T,X_{\cdot \wedge T}) +\int_t^T F_r(X,\widehat \Zc_r^\P,\sigmah_r^2,\partial Y_r^r)\d r-\int_t^T \widehat  \Zc_r^\P\cdot \d X_r+\int_t^T \d\widehat  \Kc^{\P}_r,\;  0\leq t \leq T,\; \P\text{\rm--}\as
\end{align*}
By \cite{karandikar1995pathwise}, there exists a universal process $[ \widehat \Yc^+,X]$ which coincides with the quadratic co-variation of $\widehat \Yc^+$ and $X$ under each probability measure $\P\in \Pc(\xb)$. Thus, one can define a universal $\F^{X,\Pc(\xb)}_+$-predictable process $Z$ by
\begin{align}\label{Eq:defzlemmaaux}
 \widehat Z_t:=(\sigmah^{2}_t)^\oplus \frac{\d [ \widehat \Yc^+,X]_t }{\d t},
\end{align}
and obtain,
\begin{align*}
\widehat \Yc_t^+&=\xi(T,X_{\cdot \wedge T}) +\int_t^T F_r(X,\widehat Z_r,\sigmah_r^2,\partial Y_r^r)\d r-\int_t^T\widehat Z_r \cdot \d X_r +\widehat \Kc^{\P}_T-\widehat \Kc^{\P}_t, \; 0\leq t \leq T,\; \Pc(\xb)\text{--}\qs
\end{align*}

By well-posedness, we have that 
\begin{align}\label{Eq:equalityatP0}\begin{split}
\widehat \Yc_t^+&=Y_t,\; 0\leq t\leq T,\; \Pc(\xb)\text{\rm--}\qs,\;  \widehat Z_t=Z_t,\; \sigmah^2 \d t\otimes \d \Pc(\xb)\text{\rm--}\qe\; \text{on } [0,T]\times \Xc,\\
 \widehat \Kc^{\P}_t&=K^{\P}_t, 0\leq t \leq T,\; \P\text{\rm--}\as,\; \forall \P\in  \Pc(\xb),\end{split}
\end{align}
where the later denotes the solution to the first 2BSDE in \eqref{Eq:lemma:shifted2bsde}. Thus, as $\Yc^+$ is computed $\omega-$by$-\omega$, we can repeat the previous argument on the time interval $[s,T]$ and $\Omega^\omega_s=\{\tilde \omega\in \Omega: \tilde x_r=x_r, 0\leq r\leq s\}$, \emph{i.e.} fixing an initial trajectory. Reasoning as before, we then find that on $\Omega^\omega_s$, $\widehat \Yc^{+}$ is $\F^{\Pc(s,x)}_+$-measurable and $\Pc(s,x)$--$\qs$ c\`adl\`ag. By well-posedness of the second 2BSDE in \eqref{Eq:lemma:shifted2bsde}, this yields the analogous version of \eqref{Eq:equalityatP0} between $(\widehat \Yc^+,\widehat Z, (\widehat \Kc^\P)_{\P\in \Pc(s,x)})$ and $(Y^{s,x},Z^{s,x},(K^{s,x})_{\P\in \Pc(s,x)})$. It is then clear that
\begin{align*}
Y_t=Y^{s,x}_t, s\leq t\leq T, \Pc(s,x)\text{--}\qs, \text{ for } \Pc(\xb)\text{--}\qe\; x\in \Xc,\; s\in (0,T].
\end{align*}
The corresponding result for $Z$ follows from \eqref{Eq:defzlemmaaux}. The relation for the family $(K^\P)_{\P\in \Pc(\xb)}$ holds $\P$-by-$\P$ for every $\P\in \Pc(\xb)$ in light of the weak uniqueness assumption for the drift-less dynamics \eqref{Eq:driftlessSDE} and \cite[Lemma 4.1]{claisse2016pseudo}, which guarantees that for any $\P\in \Pc(\xb)$ and $\P\text{--}\ae x\in \Xc$
\begin{align*}
\bigg(\int_t^T Z_r\cdot X_r\bigg)^\P=\bigg(\int_t^T Z_r\cdot X_r\bigg)^{\P_{s,x}}, \;  0\leq s\leq t,\; \P_{s,x}\text{--}\as
\end{align*}
\end{proof}

\begin{lemmaA}\label{Lemma:Derivative}
Let $(\P^\nu,\nu)\in \Mf(\xb)$. For $(s,x)\in [0,T]\times \Xc$ consider the system, assumed to hold $\P^{\nu}_{s,x}\text{--}\as$
\begin{align}\label{Eq:lemma:derivative}\tag{D}\begin{split}
\partial \Yc_t^{r}&=\partial_s \xi(r,X_{\cdot \wedge T})+\int_t^T \partial h_u(r,X,\partial \Zc_u^r, \nu_u) \d u-\int_t^T\partial \Zc_u^r \cdot \d X_u ,\;  s\leq t \leq T,\\
\Yc_t^{r}&=\xi(r,X_{\cdot \wedge T})+\int_t^T h_u(r,X,\Zc_u^{r},\nu_u)\d u-\int_t^T\Zc_u^{r} \cdot \d X_u , \;  s \leq t \leq T.
\end{split}
\end{align}
Let $( \partial \Yc, \partial \Zc) \in \S^{p,2}_{s,x}(\F^{X,\P^{\nu}_{ s,x}}_+,\P^{\nu}_{s,x})\times \H^{p,2}_{ s,x}(\F^{X,\P^{\nu}_{ s,x}}_+,\P^{\nu}_{  s,x},X)$ be the solution to the first {\rm BSDE} in \eqref{Eq:lemma:derivative}. Then, the mapping $[0,T]\ni s  \longmapsto (\partial \Yc^s,\partial \Zc^s)\in \S^{p}_{s,x}(\F^{X,\P^{\nu}_{ s,x}}_+,\P^{\nu}_{s,x})\times \H^{p}_{ s,x}(\F^{X,\P^{\nu}_{ s,x}}_+,\P^{\nu}_{  s,x},X)$ is Lebesgue-integrable with antiderivative $(\Yc,\Zc) $, that is to say
\begin{align*}
\bigg(\int_s^T \partial \Yc^r \d r, \int_s^T \partial \Zc^r \d r \bigg) = (\Yc^T-\Yc^s,\Zc^T-\Zc^s ),\;  \P^{\nu}_{ s,x}\text{\rm--}\as
\end{align*}
Furthermore, letting 
\begin{align*}
 \partial Y_t^r(\omega)&:=\E^{{\overline \P}^{\nu}_{t,x} } \bigg[\partial_s \xi(r,X_{\cdot\wedge T}) +\int_t^T  \partial_s f_u(r,X,\nu_u)\d u\bigg],\;  (s,t)\in [0,T]^2,\; \omega \in \Omega,
\end{align*}
it holds that $\partial Y_s^s=\E^{\overline \P_{s,x}^\nu}\big[\partial \Yc_s^s\big]$ and
 $J(s,s,x,\nu)=\E^{\overline \P_{s,x}^\nu}\big[\Yc_s^s\big]$. If in addition $\partial Y\in \S^{p,2}_{\xb}(\F^{X,\P^\nu_{s,x}}_+)$, for any $t\in [s,T]$
\begin{align*}
\E^{\overline \P^{\nu}_{s,x}}\bigg[ \int_t^T \partial Y_r^r \d r\bigg]=\E^{\overline \P^{\nu}_{s,x}}\bigg[ \int_t^T \partial \Yc_r^r \d r\bigg].
\end{align*}
\end{lemmaA}
\begin{proof}
We first prove the second part of the statement. Note that
\begin{align*}
\Yc_t^s&=  \xi(s,X_{\cdot\wedge T})+ \int_t^T \big(f_r(s,X,\nu_r)+ b_r(X,\nu_r)\cdot \sigmah_r^\t Z_r^s\big)\d r-\int_t^T Z_r^s\cdot  \d X_r,\\
&=\xi(s,X_{\cdot\wedge T}) + \int_t^T f_r(s,X,\nu_r)\d r-\int_t^T  Z_r^s\cdot  \d X_r , \; \overline \P^{\nu}_{s,x}\text{\rm--}\as
\end{align*}
This implies
\[
\Yc_t^r(x)=\E^{\overline \P^{\nu}_{ s,x}}\bigg[\xi(r,X_{\cdot\wedge T}) + \int_t^T f_u(r,X,\nu_u)\d u \Big| \Fc_{t+}^{X,\P^{\nu}_{ s,x}} \bigg].
\]
Therefore, by taking expectation
\[
\E^{\overline \P_{s,x}^\nu}\big[\Yc_s^s(\omega)\big]=\E^{\overline \P^{\nu}_{s,x}}\bigg[\xi(s,X_{\cdot\wedge T}) + \int_s^T f_r(s,X,\nu_r)\d r \bigg]=J(s,s,x,\nu).
\]
The equality $\partial Y_s^s=\partial \Yc_s^s,$ is argued identically. Now, to obtain the last equality we use the fact that $\partial \Yc \in \S^{p,2}_{s,x}(\F^{X,\P^{\nu}_{ s,x}}_+,\P^{\nu}_{ s,x})$. Indeed, the continuity of the mapping $s\longmapsto (\partial \Yc^s,\partial \Zc^s)$ guarantees the integral is well-defined. The equality follows from the tower property.\medskip

We now argue the first part of the statement. Again, we know the mapping $[0,T]\ni s\longmapsto (\partial \Yc^s,\partial \Zc^s)$ is continuous, in particular integrable. A formal integration with respect to $s$ to the first equation in \eqref{Eq:lemma:derivative} leads to
\begin{align*}
\int_s^T\partial \Yc_t^s \d s &= \int_s^T \partial_s \xi(s,X_{\cdot\wedge T})\d s +\int_t^T   \int_s^T \frac{\partial f_r }{\partial s}	(s,X,\nu_r)\d s +  b_r(X,\nu_r) \cdot  \sigmah_r^\t \int_s^T\partial \Zc_r^s \d s \d r-\int_t^T  \int_s^T \partial \Zc_r^s\d s\cdot \d X_r.
\end{align*}
Therefore, a natural candidate for solution to the second BSDE in \eqref{Eq:lemma:derivative} is $( \Yc^s,  \Zc^s, \Nc^s)$, solution of the BSDE
\begin{align*}
\Yc_t^s=\xi(s,X_{\cdot\wedge T})+\int_t^T\big( f_r(s,X,\nu_r)+ b_r(X,\nu_r)\cdot \sigmah_r^\t \Zc_r^s\big) \d r-\int_t^T \Zc_r^s\cdot   \d X_r.
\end{align*}

Let $(\Pi^\ell)_{\ell}$ be a properly chosen sequence of partitions of $[s,T]$, as in \cite[Theorem 1]{neerven2002approximating}, $\Pi^\ell=(s_i)_{i\in\{1,\dots,n_\ell\}}$ with $\| \Pi^\ell\|\leq \ell$. Recall $\Delta s_i^\ell=s_i^\ell-s_{i-1}^\ell$. For a generic family process $X\in \H^{p,2}_{s,x}(\G)$, and mappings $s\longmapsto \partial_s \xi(s,x)$, $s\longmapsto \partial_s f(s,x,a)$ for $(s,x,a)\in [0,T]\times \Xc\times A$ we define 
\begin{align*}
I^\ell( X):=\sum_{i=0}^{n_\ell} \Delta s_i^\ell  X^{s_i^\ell}, \; \delta X:=X^T-X^s,\; I^{\ell}(\partial_s \xi(\cdot,x)):=\sum_{i=0}^n \Delta s_i^\ell \partial_s \xi(s_i^\ell,x),\; I^\ell(\partial_s f)_t(\cdot,x,a):=\sum_{i=0}^{n_\ell} \Delta s_i^\ell \partial_s f_t (s_i^\ell,x,a),
\end{align*}
and notice that for any $t\in[0,T]$
\begin{align*}
I^\ell(\partial Y)_t-(\delta Y)_t & =  I^\ell(\partial_s \xi(\cdot,X_{\cdot \wedge T}))-(\xi(T,X_{\cdot \wedge T})-\xi(s,X_{\cdot \wedge T})) -\int_t^T\big(I^\ell(\partial Z)_r-(\delta Z)_r\big)\cdot \d X_r \\
&\int_t^T \big[I^\ell(\partial_s f)_r(\cdot,X,\nu_r))-(f_r(T,X,\nu_r)-f_r(s,X,\nu_r)]+\sigmah_r b_r(X,\nu)[I^\ell(\partial Z)_r-(\delta Z)_r]\d r.
\end{align*}
Thanks to the integrability of $(\partial Y, \partial Z)$ and $(Y, Z)$, it follows that $I^\ell(\partial Y)-(\delta Y)\in \H^{p,2}_{s,x}$ and similarly for $\partial Z$ and $Z$. Therefore, \cite[Theorem 2.2]{bouchard2018unified} yields
\begin{align*}
\|I^\ell(\partial Y)-(\delta Y) \|_{\H^{p,2}_{s,x}}^p + \|I^\ell(\partial Z)-(\delta Z) \|_{\H^{p,2}_{s,x}}^p \leq&\  \E^{\P^\nu_{s,x}} \bigg[ \Big|I^\ell\big(\partial_s \xi(\cdot ,X_{\cdot\wedge T})\big)-\big(\xi(T,X_{\cdot\wedge T})-\xi(s,X_{\cdot\wedge T})\big)\Big|^p\\
&+\int_t^T \Big|I^\ell \big( \partial_s f)_r (\cdot ,X,\nu_r)\big)-\big(f_r(T,X,\nu_r)-f_r(s,X,\nu_r)\big)\Big|^p\d r \bigg].
\end{align*}
The uniform continuity of $s\longmapsto \partial_s \xi(s,x)$ and $s\longmapsto\partial_s f(s,x,a)$, see Assumption \ref{AssumptionC}, justifies, via bounded convergence, the convergence in $\S^{p}_{s,x}(\F^{X,\P^{\nu}_{ s,x}}_+,\P^{\nu}_{ s,x})$ (resp $\H^{p}_{s,x}$) of $I^\ell(\partial Y^s)$ to $Y^T-Y^s$ (resp $I^\ell(\partial Z^s)$ to $Z^T- Z^s$) as $\ell \longrightarrow 0$.
\end{proof}

\begin{lemmaA}\label{Lemma:Dynamicsdiagonal}
Let $(\P^\nu,\nu)\in \Mf(\xb)$, $( s, x)\in [0,T]\times \Xc$ and $(\partial \Yc,\partial \Zc)$ and $(\Yc,\Zc)$ as in \eqref{Eq:lemma:derivative}. Then
\begin{align}\label{Eq:BSVEverification}
\Yc_{t}^{t} =\Yc_{T}^{T}+ \int_t^T  h_r(r,X, \Zc_r^r, \nu_r)-\partial \Yc_t^t\d r -\int_t^T  \Zc_r^r\cdot \d X_r, \; \tilde s \leq t \leq T,\; \P_{s,x}\text{\rm--}\as &
\end{align}
\end{lemmaA}

\begin{proof}
By evaluating $\partial \Yc$ at $r=t$ in \eqref{Eq:lemma:derivative} we get
\begin{align*}
\partial \Yc_t^t=\partial_s \xi(t,X_{\cdot \wedge T})+\int_t^T \partial h_r(t,X,\partial \Zc_r^t, \nu_r) -\int_t^T  \partial \Zc_r^t\cdot \d X_r, \; \P^{ s}_x\text{\rm--}\as 
\end{align*}
We will show that for $ s\leq t \leq T$, $\P_{s,x}$--$\as$
\begin{align*}
\begin{split}
\xi(T,X_{\cdot \wedge T})-\xi(t,X_{\cdot \wedge T})&+\int_t^T  h_r(r,X,\Zc_r^r,\nu_r)\d r -\int_t^T \Zc_r^r \d X_r = \int_t^T h_r(t,X,\Zc_r^t,\nu_r)\d r - \int_t^T \Zc_r^t \cdot \d X_r+ \int_t^T \partial \Yc_r^r \d r.
\end{split}
\end{align*}
Indeed
\begin{align*}
\int_t^T \partial \Yc_r \d r= \int_t^T  \partial_s \xi(r,X_{\cdot \wedge T})\d r +\int_t^T\int_r^T \partial h_u(r,X,\partial \Zc_u^r, \nu_u) \d u\, \d r -\int_t^T\int_r^T\partial \Zc_u^r \cdot \d X_u \d r.
\end{align*}
Now, \Cref{AssumptionC}\ref{AssumptionC:ii} and $\partial \Zc \in\H^{p,2}_{ s, x}(\F^{X,\P^{\nu}_{ s,x}}_+,\P^{\nu}_{ s,x},X)$ yield
\begin{align*}
 \int_t^T\int_r^T \partial h_u(r,X,\partial Z_u^r, \nu_u) \d u\, \d r=&\int_t^T \int_t^u \big(\partial_s f_u(r,X,\nu_u) + \sigmah_u b_u(X,\nu_u)\cdot \partial \Zc_u^r\big)\d r\d u\\
=&\int_t^T \big(h_u(u,X,\Zc^u_u,\nu_u)-h_u(t,X,\Zc_u^t,\nu_u)\big)\d u.
\end{align*}
Moreover, $\|\partial \Zc\|_{\H^{p,2}_{\tilde s,x}(X)}<\infty$ guarantees $\int_0^T \E^{ \P_{s,x}^\nu}\Big[\int_0^T \big|\sigmah_t Z^r_t\big|^2\d t\Big]^{\frac{p}{2}}\d r<\infty$ so a stochastic Fubini's theorem, see \citeauthor*{da2014stochastic} \cite[Section I.4.5]{da2014stochastic}, justifies
\[
\int_t^T\int_r^T\partial Z_u^s \cdot \d X_u \d r=\int_t^T (Z_u^u-Z_u^t) \cdot \d X_u,\; \P_{s,x}^\nu\text{\rm--}\as
\]
\end{proof}

\subsection{Well-posedness}\label{Appendix:well-posedness}
In this section we work under the setting of \Cref{Section:ProblemFormulation} but with a slightly more general system. We consider mappings
\begin{align}\label{driversBSDE}
\begin{split}
\hat h: &[0,T]\times [0,T]\times \Xc\times \R \times \R^d \times \R \longrightarrow\R ,\;
 \ {g}:[0,T]^2 \times \Xc\times \R \times \R^d\times \R \times \R^d \longrightarrow \R,\\
 \xi:& [0,T]\times \Xc\longrightarrow \R,\; \eta:[0,T]^2\times \Xc\longrightarrow \R,
 \end{split}
\end{align}
which are all assumed to be jointly Borel-measurable. We also define for any $t\in [0,T],$ $x\in \Xc$, $(y,z,u,v)\in \R \times \R^d\times \R \times \R^d$, $h(t,x,y,z,u):= \hat h_t(t,x,y,z,u)$. We will work under the following assumptions.

\begin{assumption}\label{AssumptionE}

\begin{enumerate}[label=$(\roman*)$, ref=.$(\roman*)$,wide, labelindent=0pt]

\item $s\longmapsto \hat h_t(s,x,y,z,u)$ $($resp. $s \longmapsto  g_t(s,x,u,v,y,z))$ is continuous, uniformly in $(t,x,y,z,u)$ $($resp. uniformly in $(t,x,u,v,y,z));$

\item $(y,z,u)\longmapsto  h_t(x,y,z,u)$ is uniformly Lipschitz continuous, \emph{i.e.} $\exists L_h>0,$ such that for all $(t,x,y,y',z,z',u,u')$
\begin{align*}
 |h_t(x,y,z,u)-h_t(x,y',z',u')|\leq L_h\big(|y-y'|+ |\sigma_t(x)^\t(z-z')|+|u-u'|\big);
\end{align*}

\item $(u,v,y,z)\longmapsto g_t(s,x,u,v,y,z)$ is uniformly Lipschitz continuous, \emph{i.e.} $\exists L_{g } > 0,$ such that for all $(s,t,x,u,u',$ $v,v',y,y',z,z') $ 
\begin{align*}
 \ |g_t(s,x,u,v,y,z)-g_t(s,x,u',v',y,z)|\leq L_{g}\big(|u-u'|+|\sigma_t(x)^\t(v-v')|+|y-y'|+|\sigma_t(x)^\t(z-z')|\big);
\end{align*}

\item $\tilde h_\cdot :=h_\cdot(\cdot,0,0,0)\in \L^{1,2}_\xb(\F^{X,\P}_+,\P)$, and, $  \tilde g_\cdot(s):= g_\cdot(s,\cdot,0,0,0,0)\in\L^{1,2,2}_\xb(\F^{X,\P}_+,\P)$. \label{AssumptionE:gen0}
\end{enumerate}
\end{assumption}

For $(\xi,  \eta)\in \Lc^2(\Fc^X_T,\P)\times\Lc^{2,2}(\Fc_T^X,\P)$ consider the system
\begin{align*}\label{HJB:Abstract}\tag{\Sc}
\begin{split}
\Yc_t&=\xi(T,X_{\cdot\wedge T})+\int_t^T h_r(X,\Yc_r,\Zc_r, \Uc_r^r)\d r-\int_t^T  \Zc_r \cdot \d X_r-\int_t^T \d \Nc_r,\; t\in[0,T],\\
\Uc_t^s&=   \eta (s,X_{\cdot\wedge T})+\int_t^T  g_r(s,X,\Uc_r^s,\Vc_r^s, \Yc_r, \Zc_r) \d r-\int_t^T \Vc_r^s \cdot \d X_r-\int_t^T \d \Mc^s_r,\; (s,t)\in[0,T]^2,
\end{split}
\end{align*}
with $(\Yc,\Zc,\Nc,\Uc,\Vc,\Mc)\in \L^2_\xb(\F^{X,\P}_+,\P) \times \H^2_\xb(\F^{X,\P}_+,X,\P)\times  {\M^2_\xb}(\F^{X,\P}_+,\P)\times \L^{2,2}_\xb(\F^{X,\P}_+,\P) \times \H^{2,2}_\xb(\F^{X,\P}_+,X,\P) \times {\M^{2,2}_\xb}(\R^d,\P)$.

\subsubsection{\emph{A priori} estimates and regularity properties}

In order to alleviate notations, and as it is standard in the literature, we suppress the dependence on $\omega$, \emph{i.e.} on $X$ in the functions. In this section, we fix $\xb\in \Xc$ and an arbitrary probability measure $\P\in \Pc(\xb)$. To ease the notation we will write $\H^2$ for $\H^2_\xb(\F^{X,\P}_+,X,\P)$ and similarly for the other spaces involved.  Throughout this section, we define $(\Hs, \|\cdot\|_{\Hs})$, and $(\Hs^\star, \|\cdot\|_{\Hs^\star})$, where
\begin{align*}
& \mathscr{H}:=\L^2  \times \H^2 \times  {\M^2_\xb} \times \L^{2,2}  \times \H^{2,2}  \times \M^{2,2} ,\; \mathscr{H}^\star:=\S^2  \times \H^2 \times  {\M^2_\xb} \times \S^{2,2}  \times \H^{2,2}  \times \M^{2,2}, \\
& \|(\Yc,\Zc,\Nc,\Uc, \Vc,\Mc) \|^2_{\Hs}:= \|\Yc\|_{\L^2}^2 + \|\Zc\|_{\H^2}^2+\|\Nc\|_{\M^2}^2+ \|\Uc\|_{\L^{2.2}}^2 + \|\Vc\|_{\H^{2.2}}^2+\|\Mc \|_{\M^2}^2,\\
& |(\Yc,\Zc,\Nc,\Uc, \Vc,\Mc) \|^2_{{\Hs^\star}}:=\|\Yc\|_{\S^2}^2 + \|\Zc\|_{\H^2}^2+\|\Nc\|_{\M^2}^2+ \|\Uc\|_{\S^{2.2}}^2 + \|\Vc\|_{\H^{2.2}}^2+\|\Mc \|_{\M^2}^2.
\end{align*}
We  remove the dependence of the expectation operator on the underlying measure, and write $\E$ instead of $\E^\P$. To obtain estimates between the the difference of solutions, it is more convenient to work with norms defined by adding exponential weights. For instance, for any $c\in\R$, we define the norm $\|\cdot \|_{\H^{2,c}}$ by 
\[
\|\Uc\|_{\H^{2,c}}^2=\E \bigg[  \int_0^T  \e^{ct}|\sigma^\t_t\Uc_t^s|^2 \d t  \bigg].
\]
Such norms are equivalent for different values of $c$, since $[0,T]$ is compact. We also recall Young's inequality which for $\eps >0$ states $2ab\leq \eps a^2 +\eps^{-1} b^2$, and that for any finite collection of non-negative numbers $(a_i)_{i\in\{1,\dots, n\}}$ 
\begin{align}\label{Eq:ineqsquare}
\bigg(\sum_{i=1}^n a_i\bigg)^2\leq n\sum_{i=1}^n a_i^2.
\end{align}

\begin{lemmaA}\label{Apriorinormestimates}
Let $(\Yc,\Zc,\Nc,\Uc,\Vc,\Mc)\in \mathscr{H}$ be a solution to \eqref{HJB:Abstract}. Then $(\Yc,\Uc) \in \S^2 \times \S^{2,2}$. Furthermore there exists a constant $C>0$ depending only on the data of the problem such that
\begin{align*}
\|(\Yc,\Zc,\Nc,\Uc, \Vc,\Mc) \|^2_{{\Hs^\star}} \leq C \Big(  \|\xi \|^2_{\Lc^2}  + \|\eta \|^2_{\Lc^{2,2}} +\|\tilde h\|_{\L^{1,2}}^2 +\| \tilde g \|_{\L^{1,2,2}}^2\Big)  <\infty.
\end{align*}
\end{lemmaA}

\begin{proof}
We proceed in several steps. \medskip

\textbf{Step $1$:} We derive an auxiliary estimate. By applying Meyer--It\^o's formula to $\e^{\frac{c}2  t} |\Uc_t^s|$, see \citeauthor*{protter2005stochastic} \cite[Theorem 70]{protter2005stochastic}
\begin{align}\label{eq:eq1}
\begin{split}
&\e^{\frac{c}2  t}|\Uc_t^s|+ L_T^0 -\int_t^T \e^{\frac{c}2  r} \sgn( \Uc_r^s) \Vc_r^s \cdot \d X_r -\int_t^T \e^{\frac{c}2  r-} \sgn( \Uc_{r-}^s) \d \Mc_r^s \\
&\; =\e^{\frac{c}2  T}  |\eta (s)|  + \int_t^T  \e^{\frac{c}2  r} \bigg(  \sgn(  \Uc_r^s)  g_r(s,\Uc_r^s,\Vc_r^s,\Yc_r,\Zc_r)-\frac{c}2  |\Uc_r^s| \bigg) \d r ,\; t\in[0,T],
\end{split}
\end{align} 
where $L^0:=L^0(\Uc^s)$ denotes the non-decreasing and pathwise-continuous local time of the semi-martingale $\Uc^s$ at $0$, see \cite[Chapter IV, pp. 216]{protter2005stochastic}. We also notice that for any $s\in [0,T]$ the last two terms on the left-hand side are a martingale, recall that $\Vc^s\in \H^2$.\medskip

We now note that under the Lipschitz condition
\begin{align}\label{Eq:LipassUts}
|  g_r(s,\Uc_r^s, \sigma_r^\t \Vc^s_r,\Yc_r,\Zc_r))|\leq |  \tilde g(s)|+L_{  g} \big(|\Uc_r^s|+|\sigma_r^\t \Vc_r^s|+|\Yc_r|+|\sigma_r^\t \Zc_r|\big),
\end{align}

We now take conditional expectation with respect to $\Fc_t$ in \Cref{eq:eq1}. We may use \eqref{Eq:LipassUts} and the fact $\tilde L^0$ is non-decreasing to derive that for $c>2 L_{g}$
\begin{align}\label{Eq:ineqUst}
\e^{\frac{c}2 t}| \Uc_t^s| & \leq  \E_t \bigg[ \e^{\frac{c}2 T} |\eta (s)|+\int_t^T \e^{\frac{c}2 r} \Big(|\Uc_r^s| (L_{  g}-c/2)+ | \tilde g_r(s)| +L_{ g}\big ( |\sigma^\t_r   \Vc_r^s| + |\Yc_r|+ |\sigma^\t_r  \Zc_r|\big) \Big)  \d r \bigg]\notag \\
& \leq  \E_t \bigg[ \e^{\frac{c}2 T} |\eta (s)|+\int_t^T \e^{\frac{c}2 r} \Big(  | \tilde g_r(s)| +L_{ g}\big ( |\sigma^\t_r   \Vc_r^s| + |\Yc_r|+ |\sigma^\t_r  \Zc_r|\big) \Big)  \d r \bigg],\; t\in[0,T].
\end{align}

Squaring in \eqref{Eq:ineqUst}, we may use \eqref{Eq:ineqsquare} and Jensen's inequality to derive
\begin{align*}
\frac{\e^{ct}}{5} |\Uc_t^t|^2  \leq    &\  \E_t\bigg[ \e^{cT} |\eta (s)|^2+ \bigg(\int_t^T \e^{\frac{c}2 r} | \tilde g_r(t)|\d r\bigg)^2+ T L_{ g}^2 \int_t^T \e^{c r}   \big( |\Yc_r|^2 + |\sigma^\t_r \Vc_r^t|^2 \d r+ |\sigma^\t_r\Zc_r|^2\big) \d r \bigg],\; t\in[0,T].
\end{align*}

Integrating the previous expression and taking expectation, it follows from the tower property that for any $t\in[0,T]$
\begin{align*}
\frac{  1}5\E\bigg[\int_t^T \e^{cr}|\Uc_r^r|^2\d r\bigg]\leq &\ \E\bigg[ \int_t^T   \e^{cT} |\eta (r)|^2\d r\bigg]+\E \bigg[ \int_t^T  \bigg(  \int_r^T \e^{\frac{c}2 u}| \tilde g_u(r)|\d u\bigg)^2    \d r\bigg] \\
& + T L_{ g}^2 \E \bigg[ \int_t^T  \int_r^T \e^{c u}   \big( |\Yc_u|^2 + |\sigma^\t_u \Vc_u^r|^2 + |\sigma^\t_u\Zc_u|^2\big) \d u \d r\bigg]\\
\leq &\ T \sup_{r\in [0,T]} \bigg\{ \|  \e^{cT} |\eta(r)|^2]\|_{\Lc^2}+  \E \bigg[ \bigg( \int_t^T \e^{\frac{c}2 u} |  \tilde g_u(r)|\d u \bigg)^2 \bigg]+ T L_{ g}^2   \E\bigg[   \int_t^T \e^{c u} |\sigma^\t_u \Vc_u^r|^2 \d u\bigg]  \bigg\}   \\
& +    T^2 L_{g}^2 \E \bigg [ \int_t^T \e^{c u}   \big( |\Yc_u|^2 + |\sigma^\t_u\Zc_u|^2\big) \d u \bigg].
\end{align*}

Thus, we obtain for $\widetilde C=5T^2 L_{ g}^2 \e^{cT}$ and $c>2L_{ g}$, and any $t\in[0,T]$
\begin{align}\label{Eq:ineqUtt2}
 \frac{1}{\widetilde C} \E\bigg[ \int_t^T  \e^{cr}  |\Uc_r^r|^2  \d r\bigg] \leq    T^{-1} L_{  g}^{-2}\big(  \|\eta \|_{\Lc^{2,2}}^2+ \|  \tilde g\|^2_{\L^{1,2,2}}\big) +  \E\bigg[   \int_t^T \e^{cr} \big( |\Yc_r|^2 + |\sigma^\t_r\Zc_r|^2\big)\d r\bigg]+   \| \Vc\|^2_{\H^{2,2}}. 
\end{align}

\textbf{Step $2$:} Let $s\in [0,T]$, we show that $(\Yc, \Uc) \in \S^2 \times \S^{2,2}$. By \eqref{Eq:ineqsquare}, we obtain that there exists $C>0$, which may change value from line to line, such that
\[
|\Uc_t^s|^2 \leq  C \bigg(| \eta (s)|^2  + \bigg|\int_0^T|  \tilde g_r(s)|\d r \bigg|^{\!2}+\int_0^T  \big(  |\Uc^s_r|^2+| \sigma_r^\t  \Vc_r^s|^2 +|\Yc_r|^2+| \sigma_r^\t  \Zc_r|^2 \big) \d r+\bigg| \int_t^T   \Vc_r^{s} \cdot \d X_r\bigg|^{\!2}+ \bigg|\int_t^T   \d \Mc_r^s\bigg|^{\!2} \bigg).
\]
We note that by Doob's inequality
\begin{align}\label{ControlSingleSI}
\E\bigg[ \sup_{t\in [0,T]} \left( \int_0^t \Vc_r^s \cdot \d X_r\right)^2\bigg]\leq 4 \| \Vc^s\|^2_{\H^2_\xb},
\end{align}
so it is a uniformly integrable martingale. Taking supremum over $t$ and expectation we obtain
\begin{align}\label{AuxAS2}
\E\bigg[\sup_{t\in[0,T]} |\Uc_t^s|^2\bigg] \leq C\big( \|\eta(s)\|_{\Lc^2}^2+ \|  \tilde g(s)\|^2_{\L^{1,2}}+  \| \Uc^s\|^2_{\L^2}+ \|\Vc^s\|^2_{\H^2}+\| \Yc\|^2_{\L^2}+ \| \Zc\|^2_{\H^2}+ \| \Mc^s\|^2_{\M^2} \big)<\infty.
\end{align}

Given $(\eta,  \tilde g) \in \Lc^{2,2}\times \L^{1,2,2}$ and $(\Uc,\Vc,\Mc)\in \L^{2,2}\times \H^{2,2}\times \M^{2,2}$, the map $([0,T],\Bc([0,T])) \longrightarrow (\S^{2},\|\cdot \|_{ \S^{2}}): s \longmapsto \Uc^s $ in continuous. As $s\in [0,T]$, $\| \Uc\|_{\S^{2,2}}<\infty$ and consequently $\Uc\in \S^{2,2}$.\medskip

Arguing similarly in combination with \eqref{Eq:ineqUtt2}, we obtain there exists $C>0$ such that
\begin{align}\label{AuxAS3}\begin{split}
\| \Yc\|^2_{\S^2} \leq& C\big( \|\xi(T) \|_{\Lc^2}^2+ \|\eta \|_{\Lc^{2,2}}^2+ \|\tilde h\|^2_{\L^{1,2}}+ \| \tilde g\|^2_{\L^{1,2,2}} +\|\Yc\|_{\L^2}^2+ \|  \Zc\|^2_{\H^2} + \|\Vc\|^2_{\H^{2, 2}}+ \| \Nc\|^2_{\M^2}  \big)<\infty.
\end{split}
\end{align}

\textbf{Step $3$:} We obtain the estimate of the norm. By applying It\^o's formula to $\e^{ct}\big( |Y_t|^2+|U_t^s|^2\big)$ we obtain, $\P$--$\as$
\begin{align*}
&\e^{ct}|\Yc_t|^2+\e^{ct}|\Uc_t^s|^2+\int_t^T \e^{cr}| \sigma_r^\t \Zc_r|^2 \d r+\int_t^T \e^{cr}| \sigma_r^\t \Vc_r^s|^2 \d r+\int_t^T \e^{cr-} \d [ \Nc]_r+\int_t^T \e^{cr-} \d [ \Mc^s]_r\\
&= \e^{cT}(| \xi(T)|^2 +|\eta (s)|^2) +2 \int_t^T \e^{cr}( \Yc_r  h_r( \Yc_r, \Zc_r,\Uc_r^r)-c |\Yc_r|^2+\Uc_r^s g_r(s,\Uc_r^s, \Vc^s_r,\Yc_r,  \Zc_r)-c |\Uc_r^s|^2)\d r+M_T-M_t,
\end{align*}
where we use the orthogonality of $X$ and both $\Mc^s$ and $\Nc$, and introduced the notation
\begin{align*}
M_t:=2\int_0^t\e^{cr} \Yc_r\Zc_r\cdot \d X_r+2\int_0^t \e^{cr-} \Yc_{r-} \cdot \d \Nc_r +2\int_0^t \e^{cr}  \Uc^s_r\Vc_r\cdot \d X_r+2\int_0^t \e^{cr-} \Uc^s_{r-} \cdot \d \Mc_r^s.
\end{align*}
We insist       on the fact that the integrals with respect to both $\Nc$ and $\Mc^s$ account for possible jumps, see \cite[Lemma 4.24]{jacod2003limit}. Moreover, as $(\Yc, \Uc) \in \S^2 \times \S^{2,2}$ the Lipschitz assumption together with Young's inequality yield that for $\eps>0$ the left-hand side above is smaller than
\begin{align*}
\leq&\,  \e^{cT}\big(| \xi(T)|^2+|\eta (s)|^2\big) +  \int_0^T  \e^{cr}\bigg( |\Yc_r|^2(  C_1-c) +|\Uc^s_r|^2(  C_2 -c)+\frac{1}{8} | \sigma_r^\t \Zc_r|^2+\frac{1}{8} | \sigma_r^\t \Vc_r^s|^2+\frac{1}{8\widetilde C}|\Uc_r^r|^2  \bigg) \d r  \\
&  +\eps \sup_{r\in [0,T]} \e^{cr} |\Yc_r|^2+ \eps \sup_{r\in [0,T]}  \e^{cr}|\Uc_r^s|^2 +\eps^{-1} \bigg(\int_0^T \e^{\frac{c}2 r} |\tilde h_r |\d r\bigg)^2 + \eps^{-1}\bigg(\int_0^T \e^{\frac{c}2 r} |   \tilde g_r(s)|\d r \bigg)^2   + M_T-M_t,
\end{align*}

with $\widetilde C$ as in \eqref{Eq:ineqUtt2} and the constants $ C_1$ and $  C_2$ changed appropriately. As in \eqref{ControlSingleSI}, the Burkholder--Davis--Gundy inequality in combination with the fact that both $\Yc$ and $\Uc^s$ are in $\S^2$ shows that $M$ is a true martingale. Therefore, by \eqref{Eq:ineqUtt2}, we find by taking expectation that 
\begin{align*}
\begin{split}
&\E\bigg[\e^{ct}\big(|\Yc_t|^2+ |\Uc_t^s|^2\big)+ \int_0^T \e^{cr} \big( | \sigma_r^\t \Zc_r|^2 + | \sigma_r^\t \Vc_r^s|^2\big)  \d r+\int_t^T \e^{cr-} \big( \d [ \Nc]_r+ \d [ \Mc^s]_r \big) \bigg] \\
&\ \leq \e^{cT} \big( \| \xi(T)\|^2_{\Lc^2} + (1+C) \|\eta \|^2_{\Lc^{2,2}}\big) +\E \bigg[   \int_0^T  \e^{cr}\Big( |\Yc_r|^2(C_1-c) +|\Uc^s_r|^2( C_2 -c)+\frac{3}{8} | \sigma_r^\t \Zc_r|^2+\frac{3}{4} | \sigma_r^\t \Vc_r^s|^2  \Big) \d r\bigg] \\
&\   +\eps \e^{cT}  \big(\|\Yc\|_{\S^2}^2+\|\Uc\|^2_{\S^{2,2}}\big) +\eps^{-1}\e^{cT}\|\tilde h \|^2_{\L^{1,2}} +(\eps^{-1}+C)\e^{cT} \|  \tilde g  \|_{\L^{1,2,2}}^2 +\frac{1}{8} \| \Vc\|^2_{\H^{2,2}},
\end{split}
\end{align*}

with $C_1$ and $C_2$ appropriately updated. Therefore, letting $c\geq \max \{C_1,C_2\}$, the monotonicity of the integral yields
\begin{align*}
\begin{split}
&\E\bigg[\e^{ct}|\Yc_t|^2+  \e^{ct}|\Uc_t^s|^2+ \int_0^T\e^{cr} \Big( \frac{5}{8} | \sigma_r^\t \Zc_r|^2 +\frac{7}{8}| \sigma_r^\t \Vc_r^s|^2\Big)  \d r+\int_0^T \e^{cr-} \big( \d [ \Nc]_r+ \d [ \Mc^s]_r \big) \bigg] \\
&\leq \e^{cT} \Big( \| \xi(T)\|^2_{\Lc^2} + (1+C) \|\eta \|^2_{\Lc^{2,2}}+\eps^{-1}\|\tilde h \|^2_{\L^{1,2}} +(\eps^{-1}+C) \|   \tilde g \|_{\L^{1,2,2}}^2\Big)  + \eps \e^{cT} \big(\|\Yc\|_{\S^2}^2+\|\Uc\|^2_{\S^{2,2}}\big)+ \frac{1}{8} \| \Vc\|^2_{\H^{2,2}}.
\end{split}
\end{align*}

Therefore, taking $\sup$ over $s$ and $t$ to each term on the left-hand side separately  and adding we find
\begin{align}\label{AuxAS4}
\begin{split}
&\sup_{t\in [0,T]} \E\big[|\Yc_t|^2\big]+  \sup_{(s,t)\in [0,T]^2}\E\big[|\Uc_t^s|^2\big]+\frac{5}8 \|  \Zc\|^2_{\H^2} +\frac{1}{8}  \|   \Vc\|^2_{\H^{2,2}}+ \| \Nc\|_{\M^2}^2+ \| \Mc\|_{\M^{2,2}}^2  \\
&\leq 6\e^{cT} \Big( \| \xi(T)\|^2_{\Lc^2} +  (1+C)  \|\eta\|^2_{\Lc^{2,2}} +\eps^{-1}\|\tilde h \|^2_{\L^{1,2}} +(\eps^{-1}+C) \|\tilde g \|_{\L^{1,2,2}}^2\Big)+ 6\eps \e^{cT}  \big(\|\Yc\|_{\S^2}^2+\|\Uc\|^2_{\S^{2,2}}\big).
\end{split}
\end{align}

To conclude, we take $\sup$ over $s$ in \eqref{AuxAS2} and add it to \eqref{AuxAS3}. We may then use \eqref{AuxAS4} to control the right side and find $\eps>0$ small enough such that the result holds.
\end{proof}

\begin{lemmaA}\label{Apriorinormestimatesdif}
Let $(\xi^i ,\eta^i)\in  \Lc^2\times\Lc^{2,2}$ and $(h^i, g^i)$ for $i\in\{1,2\}$ satisfy {\rm\Cref{AssumptionE}} and suppose in addition that $(\Yc^i,\Zc^i,\Nc^i,\Uc^i,\Vc^i,\Mc^i)\in \mathscr{H}$ are solutions to the system \eqref{HJB:Abstract} with coefficients $(\xi^i,h^i,\eta^i, g^i)$, $i\in\{1,2\}$. Then
\begin{align*}
\|(\delta \Yc,\delta \Zc,\delta \Nc,\delta \Uc,\delta \Vc,\delta \Mc) \|^2\leq &\ C\bigg( \E\bigg[ \big|\delta \xi\big|^2 +\bigg( \int_0^T  \big|\delta h_t( \Yc_t^1, \Zc^1_t,\Uc^{1t}_t)\big| \d t\bigg)^2 \bigg]\\
&   + \sup_{s\in [0,T]}\E\bigg[  \big|\delta \eta (s)\big|^2 +\bigg(\int_0^T \big|\delta  g_t(s,\Uc^{1s}_t, \Vc^{1s}_t, \Yc^1_t, \Zc^1_t))\big| \d t\bigg)^2 \bigg]\bigg)
\end{align*}
where for $\varphi\in\{\Yc,\Zc,\Nc,\Uc,\Vc,\Mc,\xi\}$
\[
\delta \varphi^1:= \varphi^1-\varphi^2,\; \text{\rm and}\; \delta h_t( \Yc_t^1, \Zc^1_t,\Uc^{1t}_t):= h^1_t( \Yc_t^1, \Zc^1_t,\Uc^{1t}_t)- h^2_t( \Yc_t^1, \Zc^1_t,\Uc^{1t}_t),\; t\in[0,T].
\]
\end{lemmaA}

\begin{proof}
Note that by the Lipschitz assumption on $h$ and $g$ there exist bounded processes $(\alpha^i, \beta^i,\gamma^i)$, $i\in\{1,2\}$ and $\eps^2$ such that 
\begin{align*}
\delta \Yc_t =&\ \delta \xi(T) + \int_t^T \big( \delta h_r(\Yc_r^1,\Zc_r^1,\Uc_r^{1r}) +\gamma_r^1  \delta \Yc_r + \alpha^1_r \cdot  \sigma_r^\t \delta \Zc_r  +  \beta^1_r \delta \Uc_r^r\big) \d r  - \int_t^T  \delta \Zc_r \cdot  \d X_r - \int_t^T \delta \d \Nc_r,\\
\delta \Uc_t^s=&\ \delta \eta (s)+ \int_t^T  \big(\delta  {g}_r(s,\Uc_r^{1s} ,\Vc_r^{1s}, \Yc_r^1,\Zc_r^1) + \beta_r^2  \delta \Uc_r^s  +\eps^2_r  \cdot \sigma_r^\t \delta \Vc_r^s   + \gamma_r^2  \delta \Yc_r + \alpha_r^2  \cdot  \sigma_r^\t \delta \Zc_r \big) \d r\\
&  - \int_t^T  \delta \partial \Zc_r^s  \cdot  \d X_r - \int_t^T  \delta \d \Mc^s_r.
\end{align*}
We can therefore apply Lemma \ref{Apriorinormestimates} and the result follows.
\end{proof}

\subsubsection{General well-posedness}
\begin{theoremA}\label{Existent:Abstract}
Let Assumption \eqref{AssumptionE} hold. Then \eqref{HJB:Abstract} admits a unique solution in $\mathscr{H}$.
\end{theoremA} 
In light of \Cref{Apriorinormestimates} and \Cref{Apriorinormestimatesdif} existence of a unique solution in $(\mathscr{H}, \|\cdot \|_{\mathscr{H}^\star})$ follows from \Cref{Existent:Abstract}.
\begin{proof}
We first note that uniqueness follows from \Cref{Apriorinormestimatesdif}. To show existence, let us define the map
\begin{align*}
\Tf: \mathscr{H} &\longrightarrow \mathscr{H} \\
(y,z,n,u,v,m)& \longmapsto (\Yc,\Zc,\Nc,\Uc,\Vc,\Mc),
\end{align*} 
with $(\Yc,\Zc,\Nc,\Uc,\Vc,\Mc)$ given by
\begin{align*}
\Yc_t&=\xi(T,X_{\cdot\wedge T})+\int_t^T h_r(X,y_r, z_r, \Uc_r^r)\d r-\int_t^T \Zc_r\cdot  \d X_r-\int_t^T \d  \Nc_r,\\
\Uc_t^s&= \eta (s,X_{\cdot\wedge,T})+\int_t^T  g_r(s,X,u_r^s,v_r^s, y_r, z_r) \d r-\int_t^T  \Vc_r^s \cdot \d X_r-\int_t^T \d  \Mc^s_r.
\end{align*}

\medskip
{\bf Step $1$:} We first show $\Tf$ is well defined. Let $(y,z,n,u,v,m)\in \mathscr{H}$.\medskip

\begin{enumerate}[label=$(\roman*)$, ref=.$(\roman*)$,wide, labelindent=0pt]

\item Let us first consider the pair $(\Uc,\Vc,\Mc)$. Recall that for all $s$
\begin{align*}
\Uc_t^s = \E\bigg[\eta(s)+\int_t^T   g_r(s,u_r^s, v_r^s,y_r, z_r)\d r\Big|\Fc_{t+}^{X,\P} \bigg].
\end{align*}

We first show $\|\Uc\|_{\L^{2,2}}<\infty$. To do so, note that for $s\in [0,T]$  
\begin{align*}
\widetilde \Uc_t^s:= \E \bigg[\eta(s)+\int_0^T    g_r(s,u_r^s,v_r^s,y_r,z_r) \d r\Big|\Fc_{t+}^{X,\P} \bigg], \text{ is a square-integrable $\F$-martingale}.
\end{align*}
Indeed, by \Cref{AssumptionE} $  g$ is uniformly Lipschitz in $(u,v,y,z)$. Thus \eqref{Eq:ineqsquare} and Jensen's inequality yield
\begin{align*}
\E\big[|\widetilde \Uc_t^s|^2\big]\leq 6 \Big(\|\eta (s)\|_{\Lc^2}^2 +\|   \tilde g\|_{\L^{1,2}}^2 + TL_{ g}^2 \big( \|u\|^2_{\L^{2,2}}+\|v\|^2_{\H^{2,2}}+\|y\|^2_{\L^{2}}+\|z\|^2_{\H^{2}}\big)\Big)<\infty,\; \forall t\in [0,T].
\end{align*}
Now as $\Uc_t^s=\widetilde \Uc_t^s -\E\Big[\int_0^t  g_r(s,u_r^s,v_r^s,y_r,z_r)\d r\Big]$, we get the estimate
\begin{align*}
\sup_{s\in[0,T]} \E \bigg[\int_0^T | \Uc_t^s|^2\d t \bigg]\leq 6T \Big( \|\eta\|_{\Lc^{2,2}}^2 +\|  \tilde g\|_{\L^{1,2}}^2 + TL_{ g}^2 \big( \|u\|^2_{\L^{2,2}}+\|v\|^2_{\H^{2,2}}+\|y\|^2_{\L^{2}}+\|z\|^2_{\H^{2}}\big)\Big) <\infty.
\end{align*}

Let us argue the continuity of $([0,T],\Bc([0,T])) \longrightarrow (\L^{2},\|\cdot \|_{ \L^{2}}): s \longmapsto U^s $. Let $(s_n)_n\subseteq [0,T], s_n\longrightarrow s_0\in[0,T]$, as, $ n\to\infty$, and define for $\varphi\in \{\Uc,\Vc,u,v,\eta\}$, $\Delta \varphi^n:=\varphi^{s_n}-\varphi^{s_0}$. From the previous observation we have that
\begin{align*}
| \Delta \Uc_t^n|^2 \leq 2 \E \bigg[ |\Delta {\eta}^n|^2 + T \int_0^T \big|   g_r(s_n,u_r^{s_n}, v_r^{s_n},y_r,z_r)-\partial g_r(s_0,u_r^{s_0},v_r^{s_0},y_r, z_r)\big|^2 \d r \Big|\Fc_{t+}^{X,\P} \bigg].
\end{align*}

Therefore, in light of the Lipschitz assumption we obtain there is $C>0$ such that
\begin{align*}
\E\bigg[\int_0^T | \Delta \Uc_t^n|^2\d r \bigg]\leq 2T\Big( \|\Delta {\eta}^n\|^2_{\Lc^2} +T L_{ g}^2 \big(\rho_{ g}^2(|s_n-s_0|)+\|\Delta u^{n}\|^2_{\H^{2}}+\|\Delta v^{n} \|^2_{\H^{2}}\big)\Big).
\end{align*}

We conclude $\|\Uc\|_{\L^{2,2}}<\infty$ and $\Uc\in {\L^{2,2}}$. Consequently, the predictable martingale	representation property for local martingales, \cite[Theorem 4.29]{jacod2003limit}, guarantees the existence for any $s\in[0,T]$ of a unique $\F$-predictable process $\Vc^s \in  \H^2$ and an orthogonal martingale $\Mc^s \in {\M^2}$ with the desired dynamics. Moreover, as in \eqref{AuxAS2} Doob's inequality yield $\Uc^s\in \S^2$ for all $s\in [0,T]$.  \medskip

In addition, taking $\widetilde C$ as in derivation of \eqref{Eq:ineqUtt2} in \Cref{Apriorinormestimates}, we may find that for $c> 2L_{ g}$ and $t\in [0,T]$
\begin{align}\label{Eq:ineqUtt2wp}
\frac{1}{\widetilde C} \E\bigg[\int_t^T \e^{cr} |\Uc_r^r|^2  \d r \bigg] \leq  \frac{1}{T  L_{ g}^2} \big( \|\eta \|_{\Lc^{2,2}}^2+   \| \tilde g\|_{\L^{1,2}}^2 \big) +\E\bigg[ \int_t^T \e^{cr}\big( |y_r|^2+ |z_r|^2\big) \d r \bigg] + \|v\|^2_{\H^{2,2}} .
\end{align}

\item For the tuple $(\Yc,\Zc,\Nc)$, notice that $\widetilde \Yc_t:= \E \big[\xi(T)+\int_0^T h_r(y_r, z_r, \Uc_r^r )\d r\big|\Fc_{t+}^{X,\P} \big],$  is a square integrable $\F$-martingale.

Indeed, under \Cref{AssumptionE}, $g$ is uniformly Lipschitz in $(y,z,u)$, so \eqref{Eq:ineqUtt2wp} yields
\begin{align*}
\E\big[|\widetilde \Yc_t|^2\big]&\leq 4 \bigg( \|\xi(T)\|_{\Lc^2}^2+\|\tilde h\|_{\L^{1,2}}^2+ T L_h^2 \bigg( \|y\|^2_{\H^2}+\|z\|^2_{\H^2}+\E\bigg[\int_0^T  |\Uc_r^r|^2  \d r \bigg]  \bigg) \bigg) <\infty, \;  \forall t\in [0,T].
\end{align*}

Integrating the above expression, Fubini's theorem implies that $\widetilde \Yc\in \H^2$, thus the the predictable martingale	representation property for local martingales guarantees the existence of a unique $(\Zc, \Nc) \in  \H^2\times \M^2$ such that $(\Yc,\Zc,\Nc)$ satisfies the correct dynamics, where $ \Yc
:=\widetilde \Yc -\E\big[\int_0^\cdot h_r(y_r,z_r,\Uc_r^r)\d r\big].$  Furthermore, Doob's inequality implies $\Yc\in \S^2$.\medskip

\item We now show that $(\Vc,\Mc) \in \H^{2,2}\times {\M^{2,2}}$. Applying It\^o's formula to $|\Uc_r^s|^2$ we obtain
\[
|\Uc_t^s|^2+\int_t^T | \sigma_r^\t \Vc_r^s|^2 \d r+\int_t^T \d [ \Mc^s]_r=\ |\eta (s)|^2 +2\int_t^T  \Uc_r^s  g_r(s,u_r^s, v^s_r,y_r,  z_r)\d r-2\int_t^T \Uc^s_r\Vc^s_r\cdot \d X_r
-2\int_t^T \Uc^s_{r-} \cdot \d \Mc_r^s.
\]
First note \ $\Uc^s \in \S^{2}$ ensures that the last two terms are true martingale for any $s\in[0,T]$. To show $([0,T],\Bc([0,T])) \longrightarrow (\H^{2},\|\cdot \|_{ \H^{2}})\, \big($resp. $ (\M^{2},\|\cdot \|_{ \M^{2}})\big) \, : s \longmapsto \Vc^s \, \big($resp. $ \Mc^s\big)$ is continuous, let $(s_n)_n\subseteq [0,T],$ $s_n\longrightarrow s_0\in[0,T]$, as, $n\to\infty$. We then deduce there is $C>0$ such that
\begin{align*}
\E\bigg[\int_0^T | \sigma_r^\t \Delta \Vc_r^n|^2\d r+ [ \Delta \Mc^s]_T  \bigg]\leq C\Big(\|\Delta  \eta\|_{\Lc^2}^2 +\rho_{ g}^2(|s_n-s_0|)+\|\Delta u^{n}\|^2_{\H^2}+\|\Delta v^{n}\|^2_{\H^2}\Big),
\end{align*}
and, likewise, we obtain
\begin{align*}
\sup_{s\in [0,T]} \E\bigg[\int_0^T | \sigma_r^\t \Vc_r^s|^2\d r +  [ \Mc^s]_T \bigg]\leq  C\Big(\|\eta \|_{\Lc^{2,2}}^2 +\| \tilde g\|^2_{\L^{1,2}} +\|v\|^2_{\H^{2,2}} +\|u\|^2_{\L^{2,2}} \Big)<\infty.
\end{align*}

Since the first term on the right-hand side is finite from \Cref{AssumptionE}, we obtain $\|\Vc\|_{\H^{2,2}}+\|\Mc\|_{\M^{2,2}}<\infty$. All together, we have shown that $\phi(y,z,n,u,v,m)\in \mathscr{H}$.
\end{enumerate}

\medskip
{\bf Step $2$:} We show $\Tf$ is a contraction under the equivalent norms $\|\cdot\|_{ \mathscr{H}^{c}_{\cdot}}$. Let $(y^i,z^i,n^i,u^i,v^i,m^i)_{i=1,2} \in  \mathscr{H}$ and
\[\delta h_r:= h_r(y_r^1,z_r^1,\Uc_r^{1r})-h_r(y_r^2,  z_r^2,\Uc_r^{2r}),\; \delta g_r:= g_r(s,{u_r^s}^1,  {v_r^s}^1,{y_r^s}^1,  z_r^1)- g_r(s,{u_r^s}^2, {v_r^s}^2,{y_r^s}^2,  z_r^2).\]

\begin{enumerate}[label=$(\roman*)$, ref=.$(\roman*)$,wide, labelindent=0pt]
\item With $\widetilde C$ as in derivation of \eqref{Eq:ineqUtt2} in \Cref{Apriorinormestimates}, we may find that for $c> 2L_{ g}$ and $t\in [0,T]$
\begin{align}\label{Eq:ineqUtt2wpcont}
\frac{1}{\widetilde C} \E\bigg[\int_t^T \e^{cr} |\delta \Uc_r^r|^2  \d r \bigg] \leq   \E\bigg[ \int_t^T \e^{cr}\big( |\delta y_r|^2+ |\delta z_r|^2\big) \d r \bigg] + \|\delta v\|^2_{\H^{2,2}} .
\end{align}

\item Applying It\^o's formula to $\e^{cr}\big(|\delta \Yc_r|^2+|\delta \Uc_r^s|^2\big)$ and noticing that $(\delta \Yc_T, \delta \Uc_T)=(0,0)$ we obtain
\begin{align*}
&\e^{ct}\big( |\delta \Yc_t|^2 +|\delta \Uc_r^s|^2\big) + \int_t^T \e^{cr}\big( | \sigma^\t_r \delta \Zc_r|^2+  |\sigma_r^\t \delta \Vc_r^s|^2 \big) \d r+ \int_t^T \e^{cr-}  \d \big( [ \delta\Nc]_r + [\delta \Mc^s]_r\big)+ \widetilde M_T^s-\widetilde M_t^s\\
& =\int_t^T \e^{cr}\Big(2 \delta \Yc_r \delta h_r + 2 \delta \Uc_r^s \delta  g_r  - c\big( |\delta \Yc_r|^2+ |\delta \Uc_r^s|^2\big) \Big)\d r,
\end{align*}

where
\[ \widetilde M_t^s= 2\int_0^t\e^{cr}\delta \Yc_r \delta \Zc_r\cdot \d X_r+ 2\int_0^t \e^{cr-}\delta \Yc_{r-}  \d \delta\Nc_r+ 2\int_0^t \e^{cr}\delta \Uc_r^s \delta \Vc_r^s\cdot \d X_r+ 2\int_0^t \e^{cr-}\delta \Yc_{r-}  \d \delta\Mc^s_r.\]

Again, the fact that $(\delta \Yc,\delta \Uc) \in \S^{2}\times \S^{2,2}$ guarantees, via the Burkholder--Davis--Gundy inequality, that $\widetilde M^s$ is a uniformly integrable martingale, and thus a true martingale for all $s$. Additionally, under \Cref{AssumptionE} $(g, \partial g)$ are uniformly Lipschitz in $(u,v,y,z)$ which implies 
\begin{align*}
&|h_r(y_r^1, z_r^1,\Uc_r^{1r})-h_r(y_r^2, z_r^2,\Uc_r^{2r})| \leq L_h (|\delta y_r|+| \sigma_t^\t \delta z_r| + |\delta \Uc_r^r|),\\
&|  g_r(s,{u_r^s}^1,  {v_r^s}^1,{y_r^s}^1,  z_r^1)- g_r(s,{u_r^s}^2,  {v_r^s}^2,{y_r^s}^2, z_r^2)| \leq L_{  g} \big(|\delta u_r^s|+| \sigma_r^\t \delta v^s_r|+|\delta y_r|+ |\sigma^\t_r \delta z_r|\big),
\end{align*}
yielding in turn, together with \eqref{Eq:ineqUtt2wpcont} and Young's inequality, there is $C>0$ such that for any $\eps>0$
\begin{align*}
&\ \E\bigg[\e^{ct}\big( |\delta \Yc_t|^2 +|\delta \Uc_r^s|^2\big) + \int_t^T \e^{cr}\big( | \sigma^\t_r \delta \Zc_r|^2+  |\sigma_r^\t \delta \Vc_r^s|^2 \big) \d r+ \int_t^T \e^{cr-}  \d \big( [ \delta\Nc]_r + [\delta \Mc^s]_r\big) \bigg]\\
&\  \leq \E\bigg[ \int_0^T \e^{cr}\Big (  \big (|\delta \Yc_r^s|^2+|\delta \Uc_r^s|^2\big) (2C\eps^{-1} - c) + \eps\big(|\delta u_r^s|^2+| \sigma_r^\t \delta v^s_r|^2+|\delta y_r|^2+ |\sigma^\t_r \delta z_r|^2\big)\Big) \d r+ \eps \|\delta v\|^2_{\H^{2,2}} 
\end{align*}
Choosing $\eps={2C}{c^{-1}}$ taking sup over s on the right we get
\begin{align*}
&\ \E\bigg[\e^{ct}\big( |\delta \Yc_t|^2 +|\delta \Uc_r^s|^2\big) + \int_t^T \e^{cr}\big( | \sigma^\t_r \delta \Zc_r|^2+  |\sigma_r^\t \delta \Vc_r^s|^2 \big) \d r+ \int_t^T \e^{cr-}  \d \big( [ \delta\Nc]_r + [\delta \Mc^s]_r\big) \bigg]\\
&\  \leq \frac{4C}{c} \Big(\|\delta u\|_{\L^{2,2}}+\|\delta v\|_{\H^{2,2}}+\|\delta y\|_{\L^2} +\|\delta z\|_{\H^2}\Big),
\end{align*}
yielding $\| (\delta \Yc, \delta \Zc, \delta \Nc, \delta \Uc, \delta \Vc, \delta \Mc)\|_{ \mathscr{H}^c}^2\leq \frac{4C}{c} \big(\|\delta u\|_{\L^{2,2}}+\|\delta v\|_{\H^{2,2}}+\|\delta y\|_{\L^2} +\|\delta z\|_{\H^2}\big). $ We conclude $\Tf$ has a fixed-point as it is a contraction for $c$ large enough.
\end{enumerate}

\end{proof}

\subsection{Auxiliary lemmata}\label{Appendix:Lemmata}

To being with we present a result that justifies our choice of the class $\Ac(t,x)$ in our definition of equilibrium.

\begin{lemmaA}\label{Lemma:forwardcondition}
Let $(t,\tau)\in [0,T]\times\Tc_{t,T}$, $(\nu,\tilde \nu)\in \Ac(\xb)\times \Ac(t,x)$. Then, $\P^{\nu\otimes_{\tau} \tilde \nu}=\P^\nu\otimes_{\tau} \P^{\tilde \nu}_{\tau,\cdot}$, $\nu\otimes_{\tau} \tilde \nu\in \Ac(t,x)$, $\Ac(t,x,\tau)=\Ac(t,x)$.
\end{lemmaA}
\begin{proof}
The later two results follows from the first. Let $\tilde \nu\in \Ac(t,x)$, we claim that $\P^\nu\otimes_{\tau}\P^{\tilde \nu}_{\tau,X}$ is well defined and solves the martingale problem associated with $\nu\otimes_{\tau}\nu^\star$. Indeed, by \cite[Exercise 6.7.4, Theorem 6.2.2]{stroock2007multidimensional} we have that $\omega\longmapsto \P^{\nu^\star}_{\tau(\omega),X(\tau(\omega),\omega)}[A]$ is $\Fc_\tau$-measurable for any $A\in \Fc$, and $\P^{\nu^\star}_{\tau(\omega),X(\tau(\omega),\omega)}[\Omega_\tau^\omega]=1$ for all $\omega \in \Omega$. Therefore, \Cref{Thm:Concatenated:M} guarantees $\P^\nu\otimes_{\tau(\cdot)}\P^{\tilde \nu}_{\tau(\cdot),X(\cdot)}$ is well defined and  $\P^\nu\otimes_{\tau(\cdot)}\P^{\nu}_{\tau(\cdot),X(\cdot)}$ equals $\P^\nu$ on $\Fc_\tau$ and $(\delta_\omega\otimes_{\tau(\omega)} \P^{\nu^\star}_{\tau(\omega),x})_{\omega\in \Omega}$ is an r.c.p.d. of $\P^\nu\otimes_{\tau}\P^{\nu}_{\tau,X}$ given $\Fc_\tau$. In combination with \cite[Theorem 1.2.10]{stroock2007multidimensional} this yields $M^\varphi$ is a $(\F,\P^\nu\otimes_{\tau}\P^{\nu,\tau,X})$--local martingale on $[t,T]$ with control $\nu\otimes_{\tau}\nu^\star$.
\end{proof}

\begin{lemmaA}\label{Lemma:CompCond}
Let $(t,x,\tau)\in [0,T]\times \Xc\times \Tc_{t,T}$, $(\M,\widetilde\M) \in \Mf(t,x)\times \Mf(t,x)$ such that $\nu\otimes_\tau \tilde \nu \in \Ac(t,x,\P)$. Then,
$\overline \P^{\nu}\otimes_{\tau} {\overline \P}^{\tilde \nu}_{\tau,\cdot}= { \overline \P^{\nu\otimes_{\tau} \tilde \nu}} $.
\end{lemmaA}
\begin{proof}
This follows from \Cref{Lemma:forwardcondition} and the fact we can commute changes of measure and concatenation. 
\end{proof}

\begin{lemmaA}\label{Lemma:galmarinos}
For $\ell>0$ let $(\gamma_i^\ell)_{i\in \{1,\dots,n_\ell\}}$ be sample points as in {\rm \Cref{Theorem:DPP:Limit}}. $(\hat \gamma_i^\ell)_{i\in \{1,\dots,n_\ell\}}$  are $\F$--stopping times.
\end{lemmaA}
\begin{proof}
We study $(\hat \gamma_i^\ell)_{i\in \{1,\dots,n_\ell\}}$ as the argument for the other sequences in the proof is similar. The result follows from a direct application of Galmarino's test, see \cite[C. $iv$. 99--101]{dellacherie1978probabilities}, we recall it next for completeness: Let $\varrho$ be $\Fc_T$-measurable function with values in $[0,T]$. $\varrho$ is a stopping time if and only if for every $t\in [0,T]$ we have that $\varrho(\omega)\leq t$, $(X_r,W_r, \Delta_r[\varphi])( \omega)=(X_r,W_r,\Delta_r[\varphi])(\tilde \omega)$ for all $(r,\varphi)\in [0,t]\times \Cc_b([0,T]\times A)$ implies $\varrho( \omega)=\varrho(\tilde \omega)$.
\medskip

Now, in the context of Theorem \ref{Theorem:DPP:Limit} we start with $\Pi^\ell=(\tau_i^\ell)_{i\in \{1,\dots,n_\ell\}}$ a collection of stopping times that partitions the interval $[\sigma,\tau]\subseteq [0,T]$. As $\ell>0$ is fixed we drop the dependence of the partition on $\ell$ and write $\Pi=(\tau_i)_{i\in \{1,\dots,n_\ell\}}$. Without loss of generality we consider the case of a partition of $[0,T]$. For $\omega\in \Omega$ we can coincidence of the Lebesgue integral with the so called gauge integral \cite[Definition 1.5]{mcshane1986unified} to obtain a partition $\hat \Pi^\ell=(\hat \tau_i^\ell)_{i\in \{1,\dots,n_\ell\}}$. We want to show $\hat \tau_i\in \Tc_{0,T}$ for all $i\in \{1,\dots,n_\ell\}$.\medskip

A close inspection to the construction of the gauge integral allows us to see that for fixed $\omega\in \Omega$ the choice of $\hat\tau_i(\omega)$ depends solely on the application $t\longmapsto f_t(t,x,a)$ for $(t,x,a)\in [0,T]\times \Xc\times A$. We recall that as the supremum is taken over $\Ac^{\text{pw}}(t,x)$ the action process is a fix value $a \in A$ over the interval $[\tau_{i-1},\tau_i]$. We note that $f_t(t,X,\nu_i)(\omega)=f_t(t,x_{\cdot \wedge t},a)=f_t(t,X,\nu_i)(\tilde \omega)$ as is $f_t(s,x,a)$ is optional for every $(s,a)\in [0,T]\times A$. These two facts imply $\hat \tau_i(\omega)=\hat \tau_i(\tilde\omega)$ and the result follows.
\end{proof}
\end{appendix}

\end{document}